\documentclass[11pt,a4paper]{amsart}
\usepackage[T1]{fontenc}
\usepackage[english]{babel}
\usepackage[backend=biber, style=alphabetic,sorting=nyt]{biblatex}
\usepackage{csquotes}
\usepackage{bbm}
\usepackage{eucal}

\usepackage[margin=1in]{geometry}
\usepackage{amsmath,amssymb,amsthm,mathtools}
\usepackage{mathrsfs}
\usepackage{stix2,microtype,xcolor}
\usepackage{enumitem}
\usepackage[colorlinks=true,linkcolor=blue,citecolor=magenta,urlcolor=blue]{hyperref}
\usepackage{aliascnt}
\usepackage[nameinlink,capitalize]{cleveref}
\usepackage{float}
\usepackage{tikz}
\usetikzlibrary{arrows.meta}
\numberwithin{equation}{section}

\newtheorem{theorem}{Theorem}[section]
\newaliascnt{lemma}{theorem}
\newtheorem{lemma}[lemma]{Lemma}
\aliascntresetthe{lemma}
\newaliascnt{proposition}{theorem}
\newtheorem{proposition}[proposition]{Proposition}
\aliascntresetthe{proposition}
\newaliascnt{corollary}{theorem}
\newtheorem{corollary}[corollary]{Corollary}
\aliascntresetthe{corollary}
\newaliascnt{assumption}{theorem}
\newtheorem{assumption}[assumption]{Assumption}
\aliascntresetthe{assumption}

\theoremstyle{definition}
\newaliascnt{definition}{theorem}
\newtheorem{definition}[definition]{Definition}
\aliascntresetthe{definition}
\newaliascnt{remark}{theorem}
\newtheorem{remark}[remark]{Remark}
\aliascntresetthe{remark}

\newcommand{\R}{\mathbb{R}}

\allowdisplaybreaks[1]
\makeatletter
\renewcommand\paragraph{\@startsection{paragraph}{4}{\z@}{.8\baselineskip}{-.5em}{\normalfont\bfseries}}
\makeatother
\title[Jump-modified MMD gradient flow]{A Jump-Modified Wasserstein Gradient Flow for Maximum Mean Discrepancy with Energy Kernels}
\author{Ansh Desai}
\address{Department of Mathematics, University of Chicago, Chicago IL 60637}
\email{ansh@uchicago.edu}
\author{Jason Liu}
\address{Department of Mathematics, Rutgers University, New Brunswick NJ 08901}
\email{jl3089@scarletmail.rutgers.edu}
\author{Srujan Roplekar}
\address{Department of Mathematics, University of Illinois Urbana-Champaign, Urbana IL 61801}
\email{srujanr2@illinois.edu}
\author{Matthew Rosenzweig}
\address{Department of Mathematical Sciences, Carnegie Mellon University, Pittsburgh PA 15213}
\email{mrosenz2@andrew.cmu.edu}
\urladdr{\url{https://matthewrosenzweigwork-max.github.io/}}
\thanks{A.D., J.L., and S.R. were supported by NSF Grant DMS-2441170. M.R. was supported by NSF grants DMS-2441170, DMS-2345533, and DMS-2342349, and by OpenAI through their academic researcher initiative.}
\date{}
\begin{document}
\begin{abstract}
We study a jump-modified Wasserstein gradient flow of the squared maximum mean discrepancy generated by the energy kernels $K(z)=-|z|^q$, $0<q<2$, in every dimension $d\ge2$ as well as the one-dimensional range $1\le q<2$. Under subcritical density and finite-moment assumptions, we prove global well-posedness. The added jumps transfer mass toward smaller values of the first variation and dissipate the same MMD energy as the transport term. Under our hypotheses on the target and jump rates, every Borel probability measure with finite MMD energy and zero jump dissipation is equal to the target. Combining this rigidity with compactness and a second-moment identity yields convergence to the target from finite-second-moment data, without additional uniform-in-time density or moment assumptions. We also construct a global collision-free particle process whose energy is nonincreasing along each path. A quantitative mean-field estimate applies to general initial empirical approximations, including deterministic configurations; propagation of chaos for exchangeable data and long-time target approximation follow as consequences.
\end{abstract}
\maketitle

\section{Introduction}
\label{sec:introduction}
Maximum mean discrepancy (MMD) is a kernel-based discrepancy between probability distributions, with applications to two-sample testing, distribution approximation, and generative learning \cite{gretton2012kernel,arbel2019mmd,hertrich2024riesz}. Its Wasserstein gradient flow transports a source distribution toward a prescribed target. In applications, the target is often accessible only through samples, and kernel averages can then be evaluated against its empirical measure. The analytical question is whether the resulting evolution actually approaches the target, rather than merely decreasing the discrepancy.

We consider the Euclidean energy kernels
\begin{equation}\label{eq:intro-kernel}
 K(z)=-|z|^q,\qquad 0<q<2,
\end{equation}
and add nonlocal jumps to the corresponding Wasserstein gradient flow. The transport part moves the density along the negative gradient of the first variation. The jump part instead transfers mass from larger to smaller values of that first variation, including between points that are far apart. Both mechanisms dissipate the same MMD energy. Their zero-dissipation conditions are different: vanishing transport dissipation only tests the gradient where the source carries mass, whereas vanishing jump dissipation compares those points with every possible destination.

This distinction is the main idea of the paper. We prove that any Borel probability measure with finite MMD energy and zero jump dissipation must be the target. We then obtain convergence of the continuum evolution from finite-second-moment data, without additional uniform-in-time density or moment assumptions. We also construct a collision-free particle process and prove a quantitative mean-field estimate for general initial empirical approximations, including deterministic configurations. The estimate yields propagation of chaos for exchangeable initial configurations satisfying the stated approximation condition.

\subsection{Energy MMD and the ordinary gradient flow}
\label{subsec:background-formulation}
For $a>0$, let $\mathcal P_a(\mathbb R^d)$ denote the probability measures with finite moment $\mathcal M_a(\nu)=\int |x|^a\,d\nu(x)$. Throughout, $\mu$ denotes the target probability measure. We use the same symbol for an absolutely continuous measure and its Lebesgue density when there is no confusion.

For a positive-semidefinite kernel $k$ with reproducing kernel Hilbert space $\mathcal H_k$, the associated MMD is the integral probability metric generated by the unit ball of $\mathcal H_k$:
\begin{equation}\label{eq:intro-RKHS-MMD}
 \sup_{\|u\|_{\mathcal H_k}\le1}\int u\,d(\rho-\mu)
 =\left(\iint k(x,y)\,d(\rho-\mu)(x)\,d(\rho-\mu)(y)\right)^{1/2},
\end{equation}
whenever the mean embeddings exist; see \cite{muller1997integral,sejdinovic2013equivalence}. The kernel $K(x-y)=-|x-y|^q$ is conditionally positive semidefinite, rather than positive semidefinite. Centering at the origin gives the positive-semidefinite kernel
\begin{equation}\label{eq:centered-kernel}
 k_q(x,y)=\frac12\bigl(|x|^q+|y|^q-|x-y|^q\bigr).
\end{equation}
We denote its MMD by $\operatorname{MMD}_q$, which is finite on $\mathcal P_{q/2}(\mathbb R^d)$. Under the stronger assumption $\rho,\mu\in\mathcal P_q(\mathbb R^d)$, the one-variable terms cancel against the zero-mass measure $\rho-\mu$, and
\begin{equation}\label{eq:energy}
 \mathcal E_q(\rho-\mu)
 :=\operatorname{MMD}_q(\rho,\mu)^2
 =\frac12\iint K(x-y)\,d(\rho-\mu)(x)\,d(\rho-\mu)(y).
\end{equation}
For $q=1$, this is the classical energy distance, up to normalization. The restriction $q<2$ is essential for distinguishing probability measures: at $q=2$, the centered kernel is $x\cdot y$ and the discrepancy detects only the difference of means \cite{sejdinovic2013equivalence}.

Two features are useful for approximation. First, $\operatorname{MMD}_q(\delta_x,\delta_y)^2=|x-y|^q$, so the discrepancy does not saturate as two point masses separate. Second, independent sampling of a measure in $\mathcal P_q$ gives expected squared MMD of order $N^{-1}$, independently of the dimension \cite[Proposition~A.1]{rosenzweig2026wassersteingradientflowsmaximum}; this differs from the generally dimension-dependent empirical Wasserstein rates \cite{fournier2015rate}. Slicing also permits efficient evaluation of the energy-distance discrepancy and its particle gradients \cite{kolouri2022sliced,hertrich2024sliced}. These properties motivate studying the underlying dynamics.

The first variation of \eqref{eq:energy}, defined up to an additive constant, is
\begin{equation}\label{eq:witness}
 f_\rho(x)=K*(\rho-\mu)(x).
\end{equation}
The ordinary Wasserstein gradient flow is the continuity equation
\begin{equation}\label{eq:ordinary}
 \partial_t\rho_t=\nabla\!\cdot(\rho_t\nabla f_{\rho_t}),
 \qquad v_{\rho_t}=-\nabla f_{\rho_t},\qquad \rho_t\big|_{t=0}=\rho_0.
\end{equation}
Its energy dissipation is $\int|\nabla f_\rho|^2\,d\rho$. A measure for which $\nabla f_\rho=0$ $\rho$-a.e. is called a \emph{Lagrangian critical point}, whenever this field is defined in the required integrability class. This is a condition on the energy, distinct from stationarity of the modified equation introduced below.

For sufficiently smooth kernels, the MMD flow was studied by Arbel, Korba, Salim, and Gretton \cite{arbel2019mmd}. The energy kernels present two different difficulties. The force $-\nabla K(z)=q|z|^{q-2}z$ is not Lipschitz at the diagonal for any $0<q<2$. In dimensions $d\ge2$, the energy is also not displacement semiconvex. Flat convexity does not imply convexity along Wasserstein geodesics, so these dynamics fall outside the standard displacement-semiconvex gradient-flow theory \cite{hertrich2024riesz,rosenzweig2026wassersteingradientflowsmaximum}.

The work of Slep\v{c}ev, Wang, and the fourth author \cite{rosenzweig2026wassersteingradientflowsmaximum}, which should be viewed as a companion to the present paper, establishes well-posedness and proves that absolutely continuous Lagrangian critical points in its rigidity range equal the target. It also proves asymptotic criticality and conditional target convergence. A priori, a long-time limit need not be absolutely continuous: the finite-time density bounds need not persist uniformly in time. For energy distance, the companion constructs non-target singular critical spherical shells and static smooth approximations with vanishing transport dissipation and positive limiting MMD \cite[Theorem~1.5, Proposition~5.21, and Corollary~5.22]{rosenzweig2026wassersteingradientflowsmaximum}. They explain why identifying a limit from its transport criticality alone requires care.

\subsection{The jump mechanism}
\label{subsec:jump-mechanism}
The first variation suggests allowing transfers between arbitrary locations, not only infinitesimal spatial displacements. In the signed-measure direction $\delta_y-\delta_x$, the quadratic energy satisfies
\begin{equation}
 \left.\frac{d}{d\varepsilon}
 \mathcal E_q\bigl(\rho+\varepsilon(\delta_y-\delta_x)-\mu\bigr)
 \right|_{\varepsilon=0}=f_\rho(y)-f_\rho(x).
\end{equation}
We use this first-variation calculation to define a rate for mass already present at the departure point. Write $[a]_+=\max\{a,0\}$ and set
\begin{equation}\label{eq:rate}
 \zeta_\rho(x,y):=[f_\rho(x)-f_\rho(y)]_+.
\end{equation}
Let $\kappa(z)\,dz$ be the distribution of a proposed displacement, and let $\lambda(x)$ be a nonnegative spatial rate. Given departure point $x$, transitions to $y$ have rate $\lambda(x)\zeta_\rho(x,y)\kappa(y-x)\,dy$. The corresponding mass flux is weighted by $d\rho(x)$.

A transition from $x$ to $y$ changes a test observable $\varphi$ by $\varphi(y)-\varphi(x)$. Integrating this increment against the rate gives the jump operator $J$ defined by the weak form:
\begin{equation}\label{eq:Jweak}
 \int\varphi\,dJ(\rho)
 =\iint (\varphi(y)-\varphi(x))\lambda(x)\zeta_\rho(x,y)
                  \kappa(y-x)\,dy\,d\rho(x).
\end{equation}
The moment and rate assumptions below make the outgoing intensity integrable. Taking $\varphi=1$ therefore gives $J(\rho)(\mathbb R^d)=0$. Thus, the jump mechanism conserves total mass.

We study the combined evolution
\begin{equation}\label{eq:pde}
 \partial_t\rho_t=\nabla\!\cdot(\rho_t\nabla f_{\rho_t})+J(\rho_t),
 \qquad \rho_t\big|_{t=0}=\rho_0.
\end{equation}
The target remains stationary because $f_\mu=0$. The energy-dissipation relation is
\begin{equation}\label{eq:intro-dissipation-relation}
 \frac{d}{dt}\mathcal E_q(\rho_t-\mu)=-\mathcal D(\rho_t),
\end{equation}
where
\begin{align}
 \mathcal D_{\mathrm{tr}}(\rho)&:=\int|\nabla f_\rho(x)|^2\,d\rho(x),\label{eq:transport-dissipation}\\
 \mathcal D_{\mathrm j}(\rho)&:=\iint\lambda(x)[f_\rho(x)-f_\rho(x+z)]_+^2
                        \kappa(z)\,dz\,d\rho(x),\label{eq:weighted-jump-dissipation}\\
 \mathcal D(\rho)&:=\mathcal D_{\mathrm{tr}}(\rho)+\mathcal D_{\mathrm j}(\rho).
 \label{eq:total-dissipation}
\end{align}
The second term is the additional dissipation due to jumps. Proposition~\ref{prop:energy-dissipation} justifies \eqref{eq:intro-dissipation-relation} in integrated form for the constructed solutions. Identifying the zero set of $\mathcal D_{\mathrm j}$ requires a separate argument, stated in Theorem~\ref{prop:no-spur-stat-soln-intro} below.

The following assumptions are used throughout the continuum and particle constructions.
\begin{assumption}\label{ass:kappa}
Let $r\ge\max\{1,q\}$. The proposal $\kappa\in\mathcal P_{r+1}(\mathbb R^d)$ is a probability density with full support. The spatial rate belongs to $W^{1,\infty}(\mathbb R^d)$, and its Lipschitz representative satisfies $\lambda(x)>0$ for every $x\in\mathbb R^d$.
\end{assumption}
A concrete choice is a Gaussian proposal and a capped spatial rate. For this example, assume $\rho_0\ne\mu$ and set
\begin{equation}\label{eq:kappa-lambda-ex}
 \kappa(z)=(2\pi\beta^2)^{-d/2}e^{-|z|^2/(2\beta^2)},\qquad
 \lambda(x)=\min\{W_1(\delta_x,\mu)^2,W_1(\rho_0,\mu)^2\},
 \qquad\beta>0.
\end{equation}
Here, $W_1$ is the standard 1-Wasserstein distance on $\mathcal P_1(\mathbb R^d)$. If $D_0=W_1(\rho_0,\mu)>0$, then this rate is bounded by $D_0^2$ and has Lipschitz constant at most $2D_0$. It is positive at every point because the target density is not a point mass. When $\rho_0=\mu$, the cap is zero, so \eqref{eq:kappa-lambda-ex} does not satisfy Assumption~\ref{ass:kappa}. Once chosen from the reference density $\rho_0$, $\lambda$ is held fixed, including in the particle approximation and when comparing the evolution with the stationary target.

The departure-mass weighting in \eqref{eq:Jweak} is the upwind structure of the nonlocal interaction equation of Esposito, Patacchini, Schlichting, and Slep\v{c}ev \cite{esposito2021nonlocal}. For constant $\lambda$ and even $\kappa$, the edge weight is symmetric, although the mass flux remains directed by the first variation. Our equation also contains the local Wasserstein transport term.

\subsection{Main results}
\label{subsec:main-results}
We first construct the continuum evolution. The new difficulty, relative to the transport equation, is that jumps can increase the density and its moments even while decreasing the MMD. The proof controls these quantities together through kernel regularization, a positive linear jump equation, and a fixed-point argument. A coupling estimate then proves uniqueness.

For $d+q-2\ge0$, define
\begin{equation}\label{eq:continuum-critical-exponent}
 p_c=\begin{cases}d/(d+q-2),&d+q-2>0,\\ \infty,&d+q-2=0.\end{cases}
\end{equation}
The strict inequality $p>p_c$ is the subcritical convolution condition for the Hessian of $K$. Equality in the dimensional condition means $d=q=1$ and is treated using the explicit one-dimensional force. We write $p_*$ for the H\"older conjugate of $p$; subscripts on constants indicate their dependence.

The maximum estimate below controls the $L^\infty$ norm of the source density. It uses the increment bound $|f_\nu(x)-f_\nu(y)|\le A_\nu |x-y|^{\min\{1,q\}}$, with
\begin{equation}\label{eq:acceptance-coefficient}
 A_\nu=\begin{cases}
 2,&0<q\le1,\\
 C_q\bigl(\mathcal M_{q-1}(\nu)+\mathcal M_{q-1}(\mu)\bigr),&1<q<2,
 \end{cases}
\end{equation}
where $C_q$ is chosen as in Lemma~\ref{lem:jump-increment}.

\begin{theorem}[Global well-posedness]
\label{thm:continuum-wellposedness}
Let \(r\ge\max\{1,q\}\), let \(0<q<2\) satisfy \(d+q-2\ge0\), and let
\begin{equation}
\label{eq:continuum-admissible-exponent}
  p_c\le p\le\infty,
  \qquad
  p>p_c\quad\text{if }d+q-2>0.
\end{equation}
Let
\(\rho_0,\mu\in\mathcal P_r(\mathbb R^d)\cap L^p(\mathbb R^d)\)
be probability densities, and assume Assumption~\ref{ass:kappa}. Then
\eqref{eq:pde} has a unique global weak solution in the sense of
Definition~\ref{def:continuum-weak-solution}. For every \(T>0\),
\begin{equation}
\label{eq:continuum-moment-velocity-regularity}
  \sup_{0\le t\le T}\mathcal M_r(\rho_t)<\infty,
  \qquad
  v_\rho\in
  L^\infty(0,T;W^{1,\infty}(\mathbb R^d;\mathbb R^d))
  \cap C([0,T]\times\mathbb R^d;\mathbb R^d).
\end{equation}
If \(p=\infty\), then
\begin{equation}
\label{eq:continuum-maximum-estimate}
  \|\rho_t\|_{L^\infty}
  \le
  \exp\!\left(
    \|(-\Delta K)*\mu\|_{L^\infty}t
    +\|\lambda\|_{L^\infty}\mathcal M_{\min\{1,q\}}(\kappa)
      \int_0^t A_{\rho_s}\,ds
  \right)\|\rho_0\|_{L^\infty},
  \qquad t\ge0,
\end{equation}
with $A_{\rho_s}$ as in \eqref{eq:acceptance-coefficient}.
\end{theorem}

The solution is global, while the moment and density estimates are only asserted to be uniform on each finite time interval. Long-time convergence will follow from the dissipation, not from assuming these estimates remain bounded as time tends to infinity.

For an arbitrary Borel probability measure $\rho$, define
$\mathcal E_q(\rho-\mu)\in[0,\infty]$ by the Fourier formula
\eqref{eq:extended-fourier-energy}, with the normalization of \eqref{eq:energy}.
For finite-energy measures, the quantity $\mathcal D_{\mathrm j}$ introduced in \eqref{eq:weighted-jump-dissipation} is defined by
\eqref{eq:extended-weighted-jump-dissipation} using the normalized increments
\eqref{eq:extended-potential-increment}. These increments are well-defined by
Lemma~\ref{lem:energy-sublevels}. Since $\lambda$ is pointwise positive,
zero jump dissipation means that $\Delta_\rho(x,z)\le0$ for
$\rho(dx)\kappa(z)\,dz$-a.e. $(x,z)$. The Fourier conventions and
integrability details are given in Section~\ref{subsec:stationary-long-time}.

\begin{theorem}[Rigidity of jump dissipation]\label{prop:no-spur-stat-soln-intro}
Assume the target, kernel, proposal, and spatial-rate hypotheses of Theorem~\ref{thm:continuum-wellposedness}. If $\rho\in\mathcal P(\mathbb R^d)$ satisfies
\begin{equation}
 \mathcal E_q(\rho-\mu)<\infty,\qquad \mathcal D_{\mathrm j}(\rho)=0,
\end{equation}
then $\rho=\mu$.
In particular, under the hypotheses of Theorem~\ref{thm:continuum-wellposedness} with $r\ge2$, the target is the unique stationary solution of \eqref{eq:pde} among the densities $\nu\in\mathcal P_r(\mathbb R^d)\cap L^p(\mathbb R^d)$, in the sense of Definition~\ref{def:stationary-critical}.
\end{theorem}
The first assertion imposes no absolute continuity assumption on $\rho$.
The full-support proposal and pointwise positivity of $\lambda$ imply that zero jump dissipation makes the normalized first variation attain its global minimum at $\rho$-a.e. source point. Pairing that function with $\rho-\mu$ gives twice the MMD energy, but the minimum property makes the pairing nonpositive. This identifies the source with the target. Section~\ref{subsec:stationary-long-time} proves this pairing for measures with finite energy, without assuming a finite $q/2$-moment.

\begin{theorem}[Convergence to the target]\label{thm:continuum-convergence}
Let $r\ge2$ and let $(\rho_t)_{t\ge0}$ be the solution of Theorem~\ref{thm:continuum-wellposedness}. Then
\begin{equation}
 \operatorname{MMD}_q(\rho_t,\mu)\longrightarrow0\qquad\text{as }t\to\infty.
\end{equation}
\end{theorem}
The proof of Theorem~\ref{thm:continuum-convergence} first combines energy-sublevel tightness, lower semicontinuity of jump dissipation, and Theorem~\ref{prop:no-spur-stat-soln-intro} to identify narrow limits. An estimate of the displacement over a short time interval passes from selected times of small dissipation to arbitrary times. A separate second-moment identity and time-averaging argument then yield convergence in MMD. No uniform-in-time $L^p$ bound or additional uniform moment bound is assumed. For the ordinary flow, the energy-distance result of \cite[Theorem~1.7 and Corollary~5.27]{rosenzweig2026wassersteingradientflowsmaximum} also supplies the required long-time moment control from second moments, but still uses uniform density control for target identification; its general-exponent result assumes both uniform moment and density bounds.

We next describe a particle approximation of the continuum equation. Denote a configuration and its empirical measure by
\begin{equation}\label{eq:particle-configuration-notation}
 X=(x_1,\ldots,x_N),\qquad
 \rho_X^N=\frac1N\sum_{i=1}^N\delta_{x_i},\qquad
 \mathsf D_N=\{X:x_i\ne x_j\text{ whenever }i\ne j\}.
\end{equation}
Between jumps, the particles follow the diagonal-free transport dynamics
\begin{equation}
 \dot x_i=-\frac1N\sum_{j\ne i}\nabla K(x_i-x_j)+\nabla K*\mu(x_i).
\end{equation}
Moving one atom by $z$ transfers mass $1/N$, so the exact energy increment contains a quadratic correction. The nonnegative factor in the accepted jump rate is
\begin{equation}
 \zeta_i^N(X,z)=\left[f_{\rho_X^N}(x_i)-f_{\rho_X^N}(x_i+z)-\frac{|z|^q}{N}\right]_+.
\end{equation}
The rate for this move is $\lambda(x_i)\zeta_i^N(X,z)\kappa(z)\,dz$. For every accepted move, the energy decreases by $\zeta_i^N(X,z)/N$. Section~\ref{subsec:particle-construction} gives the precise process driven by independent Poisson random measures $\mathcal N_1,\ldots,\mathcal N_N$ of intensity $dt\,\kappa(z)\,dz\,du$. We denote its empirical measure by $\rho_t^N=\rho_{X(t)}^N$.

\begin{theorem}[Well-posedness and noncollision]
\label{thm:particle-wellposedness}
Assume the hypotheses of Theorem~\ref{thm:continuum-wellposedness}, and fix \(N\ge2\).
Let $X^0$ be independent of the entire family $(\mathcal N_i)_{i=1}^N$, and assume
\begin{equation}
\label{eq:particle-initial-moment}
  \mathbb P(X^0\in\mathsf D_N)=1,
  \qquad
  \mathbb E\left[\frac1N\sum_{i=1}^N(1+|X_i^0|^2)^{r/2}\right]<\infty.
\end{equation}
There exists a pathwise unique global collision-free particle
solution \(X\) with initial configuration \(X^0\), in the sense of
Definition~\ref{def:particle-solution}.  For every \(T>0\),
\begin{equation}
\label{eq:particle-finite-jump-count}
  \mathbb P\left(
  \#\{t\in(0,T]:X(t)\ne X(t-)\}<\infty
  \right)=1.
\end{equation}
Finally, the following particle-moment estimate holds:
\begin{equation}
\label{eq:particle-uniform-moment-bound}
  \mathbb E\left[\sup_{0\le t\le T}\frac1N\sum_{i=1}^N(1+|X_i(t)|^2)^{r/2}\right]
  \le C_T\mathbb E\left[\frac1N\sum_{i=1}^N(1+|X_i^0|^2)^{r/2}\right],
\end{equation}
where the constant $C_T$ is independent of $N$.
\end{theorem}

No independence among the coordinates of the initial configuration $X^0$ is required here. The independence assumption separates the initial configuration from the entire family of driving noises. The construction excludes both collisions and accumulation of accepted jump times. Its energy identity also gives pathwise monotonicity of the discrepancy from the target; see Proposition~\ref{prop:particle-pathwise-energy-dissipation}.

We measure the discrepancy between the particle empirical measure and the continuum solution using both MMD and the bounded-Lipschitz distance
\begin{equation}\label{eq:particle-BL-distance}
 d_{\rm BL}(\nu,\omega)=
 \sup_{\|\varphi\|_{L^\infty}+\operatorname{Lip}(\varphi)\le1}
        \left|\int\varphi\,d(\nu-\omega)\right|.
\end{equation}
The bounded-Lipschitz distance is needed only to estimate the jump contribution. Define the sampling factor
\begin{equation}\label{eq:particle-BL-dimensional-rate}
 \mathfrak n_d(N)=\begin{cases}
 N^{-1},&d=1,\\ N^{-1}\log(eN),&d=2,\\ N^{-2/d},&d\ge3,
 \end{cases}
\end{equation}
and the initial error
\begin{equation}\label{eq:particle-initial-empirical-error}
 \Theta_N=
 \mathbb E\left[\operatorname{MMD}_q(\rho_0^N,\rho_0)^2\right]
 +\mathbb E\left[d_{\rm BL}(\rho_0^N,\rho_0)^2\right]
 +\mathfrak n_d(N)^{r/(r+1)}.
\end{equation}
\begin{theorem}[Mean-field estimate for general initial configurations]\label{thm:particle-empirical-mean-field}
Let $(\rho_t)$ be the continuum solution of Theorem~\ref{thm:continuum-wellposedness}, and let $X$ satisfy the assumptions of Theorem~\ref{thm:particle-wellposedness}. For every $T>0$,
\begin{equation}\label{eq:particle-empirical-mean-field}
 \sup_{0\le t\le T}\mathbb E\left[
  \operatorname{MMD}_q(\rho_t^N,\rho_t)^2
  +d_{\rm BL}(\rho_t^N,\rho_t)^2\right]\le C_T\Theta_N.
\end{equation}
The constant $C_T$ depends on the fixed continuum data, the kernel, proposal, spatial rate, and $T$, but not on $N$ or the initial interacting law.
\end{theorem}
In particular, the estimate includes deterministic collision-free configurations and does not require exchangeability.
The proof of Theorem~\ref{thm:particle-empirical-mean-field} is a modulated-energy argument coupled to independent comparison processes with law $\rho_t$. Its more precise estimate, Proposition~\ref{prop:modulated-energy-closure}, also controls the integrated squared drift and acceptance-factor errors. Each interacting particle and its comparison process use the same jump proposals; a proposal accepted by both moves them by the same displacement and does not change their separation. The two dissipative error terms absorb the mismatches in drift and acceptance. A symmetrized matching of the initial configurations converts this coupled estimate into \eqref{eq:particle-empirical-mean-field}.

\begin{corollary}[Independent initial sampling]
\label{thm:particle-iid-mean-field}
Assume the hypotheses of Theorem~\ref{thm:continuum-wellposedness},
and let \((\rho_t)_{t\geq 0}\) be its weak solution with initial datum \(\rho_0\).
Let \(X(t)\) be the particle solution of Theorem~\ref{thm:particle-wellposedness}, with
\(X_1^0,\ldots,X_N^0\) independent and of common law \(\rho_0\).
For every \(T>0\), there exists \(C_T<\infty\), independent of \(N\),
such that
\begin{align}
  \sup_{0\le t\le T}
  \mathbb E\left[\operatorname{MMD}_q(\rho_t^N,\rho_t)\right]
  &\le C_T \begin{cases}
    N^{-r/(2(r+1))},&d=1,\\
    (N^{-1}\log(eN))^{r/(2(r+1))},&d=2,\\
    N^{-r/(d(r+1))},&d\ge3.
  \end{cases}
  \label{eq:particle-iid-MMD-first-power}
\end{align}
\end{corollary}

The order in \eqref{eq:particle-empirical-mean-field} and \eqref{eq:particle-iid-MMD-first-power} is a supremum of expectations. A separate path-space coupling gives propagation of chaos for exchangeable initial laws with $\Theta_N\to0$; see Corollary~\ref{thm:particle-propagation-of-chaos}. Empirical convergence alone cannot identify distinguished labels without such symmetry. Energy monotonicity and Theorem~\ref{thm:continuum-convergence} also yield target approximation when both time and particle number tend to infinity, as stated in Corollary~\ref{thm:particle-long-time-convergence}.

Under its well-posedness hypotheses, \cite[Proposition~1.14 and Lemma~7.1]{rosenzweig2026wassersteingradientflowsmaximum} shows that solutions from translations to distance $R$ of a fixed smooth compactly supported source remain near their initial locations for waiting times tending to infinity, while their MMD energy stays comparable to $R^q$. Consequently, no multiplicative decay envelope tending to zero is uniform over the initial separation.

For the ordinary flow with $q=1$, the force bound $|v_\rho|\le2$ gives $\mathcal D_{\mathrm{tr}}(\rho)\le4$. An exponential estimate with prefactor one,
\begin{equation}
 \mathcal E_1(\rho_t-\mu)\le e^{-ct}\mathcal E_1(\rho_0-\mu),
\end{equation}
on a positive interval beginning at zero would therefore require $c\mathcal E_1(\rho_0-\mu)\le4$. Since the initial energy of the translated source has order $R$, this rules out a positive $c$ independent of the separation. This initial-time argument concerns prefactor one; a larger prefactor leaves slack at zero.

For the jump-modified flow, we can overcome this separation obstruction. As a concrete example, we take $q=1$, a Gaussian proposal, and the capped spatial rate \eqref{eq:kappa-lambda-ex}, fixed from the initial density. Its cap has order $R^2$ along translated initial data. The estimate holds while the prescribed concentration and separation conditions persist, with parameters chosen in the stated a priori order. The precise statement of the result is the following proposition.
\begin{proposition}[Far-field energy decay]\label{thm:far-field-main}
Assume the hypotheses of Theorem~\ref{thm:continuum-wellposedness} with $q=1$, $r\ge2$, and the rate choice \eqref{eq:kappa-lambda-ex}. Fix $T>0$. Choose the concentration parameters and positive bandwidth in the a priori order specified in Proposition~\ref{thm:farfield-pl}, and assume its initial separation conditions. On the resulting stopped interval $[0,\tau_2]\subset[0,T]$,
\begin{equation}
 \mathcal E_1(\rho_t-\mu)
 \le e^{-\beta^2\widetilde C_{\rm PL}t}\mathcal E_1(\rho_0-\mu),
 \qquad 0\le t\le\tau_2,
\end{equation}
where $\widetilde C_{\rm PL}=C_{\rm PL}s_\mu/2$, with $C_{\rm PL}$ and $s_\mu$ as in that proposition. Its exit-time lower bounds give a positive interval when the stated initial margins are strict.
\end{proposition}

\subsection{Prior and related work}
\label{subsec:prior-work}
Several works analyze MMD gradient flows without the jump modification considered here. Arbel, Korba, Salim, and Gretton \cite{arbel2019mmd} study smooth-kernel MMD flows and sufficient conditions for convergence, including noise injection. For energy kernels, Hertrich, Gr\"af, Beinert, and Steidl \cite{hertrich2024riesz} study Wasserstein steepest descent using tangent plans. Earlier, Di Francesco, Fornasier, H\"utter, and Matthes \cite{difrancesco2014asymptotic} studied one-dimensional power-law attraction--repulsion flows through inverse distribution functions. For the distance kernel, they proved convergence to the target in \(W_2\) for bounded, compactly supported initial and target probability densities. More recently, Duong, Stein, Beinert, Hertrich, and Steidl \cite{Duo+26} analyze the one-dimensional distance-kernel flow through quantiles; related one-dimensional interaction flows and entropy solutions are treated in \cite{Bon+15}. Rux, Quellmalz, and Steidl \cite{rux2025smoothed} instead smooth the distance kernel while preserving conditional positive semidefiniteness and its large-scale behavior. Our equation uses the original nonsmooth kernel and modifies the evolution through jumps.

Recent works concern complementary pure-transport questions. Fornasier, Huang, and Sun \cite{fornasier2026nonlocal} consider power-law attraction--repulsion equations in which the exponents of the source and target interactions, and their total masses, need not coincide. Their confinement results concern dominance of attraction, rather than the equal-exponent, equal-mass MMD setting. Chizat, Colombo, Colombo, and Fern\'andez-Real \cite{chizat2026quantitative} study kernel discrepancies on the torus through negative Sobolev norms. Their quantitative results distinguish exponential convergence in the Coulomb case from local polynomial convergence for more regular discrepancies. Lin and Rigollet \cite{lin2026sharp} study pure-transport flows on closed manifolds for \(C^2\) kernels diagonal in a Laplace eigenbasis. When the nonconstant spectral coefficients are positive, they obtain convergence to uniform volume and \(o(t^{-1})\) decay of the energy gap for initial probability densities in \(L^p\), \(p>1\), using entropy dissipation without a positive lower bound on the density. The present paper treats a Euclidean equation with an additional nonlocal jump dissipation; neither its measure-level rigidity nor its particle jump approximation follows from those transport results.

The work \cite{rosenzweig2026wassersteingradientflowsmaximum} provides the closest analytical comparison. It provides the nonsmooth MMD transport theory, the natural-moment rigidity of absolutely continuous critical points, and the modulated-energy commutator used here. For $1\le q<2$, it proves asymptotic criticality without additional long-time bounds. At $q=1$, in dimensions at least two, finite initial and target second moments give compactness in every Wasserstein distance of order at least one and below two and equality between the energy of each subsequential limit and the limiting energy. Target identification still requires excluding singular critical limits; uniform $L^p$ control is a sufficient additional condition \cite[Theorem~5.23, Theorem~1.7, and Corollary~5.27]{rosenzweig2026wassersteingradientflowsmaximum}. Our jump dissipation identifies a finite-energy limit without that density condition.

Chodron de Courcel and the fourth author \cite{chodron2026coulomb} study Coulomb discrepancies, using the special structure of the Coulomb potential as the fundamental solution of $-\Delta$ (and its zero-mean periodic analogue) to obtain density regularization and convergence results that depend on the geometry and target assumptions. Their positive-order overlap with the present energy-kernel family is the endpoint $d=q=1$, while in dimensions at least two the Coulomb kernels lie outside that family.

We also mention static empirical-quantization results, such as \cite{fornasier2016consistency,colasanto2026sharp}. These concern globally optimal configurations for the discrete energy associated with the pure-transport particle system. They do not imply that every evolving particle configuration reaches the optimal quantization error.

More recently, Lee and Ghosal \cite{LeeGhosal2026CoulombMMD} obtained quantitative particle approximations for the periodic Coulomb discrepancy flow in dimensions at least two, uniformly in time and temperature (including the deterministic case), under sufficiently regular, strictly positive initial and target densities and independent initial sampling.

The directed mass transfers in \eqref{eq:Jweak} are closely related to the upwind nonlocal interaction equation of \cite{esposito2021nonlocal}. Weighting each directed transfer by its departure mass distinguishes that construction from the symmetric averages of the densities at both endpoints used in the reversible Markov-chain and nonlocal Wasserstein geometries of \cite{maas2011gradient,erbar2014gradient}; see also \cite{slepcev2023nonlocal} for metric and local-limit questions. We use the departure mass and the positive difference of first variations, and add continuous spatial transport. Under the construction hypotheses of Section~\ref{subsec:particle-construction}, the corresponding particle process is a piecewise-deterministic Markov process in the sense of Davis \cite{davis1984piecewise}, determined by its deterministic flow, total jump rate, and post-jump kernel.

For the mean-field limit, classical approaches such as Dobrushin's coupling estimate \cite{dobrushin1979vlasov} do not apply directly in their standard globally Lipschitz-force form, since $\nabla K$ is not Lipschitz at the origin. Singular forces stemming from a potential that generates a coercive energy can instead be compared through a modulated energy and its transport commutator, when the corresponding commutator estimate is available. The Coulomb and Riesz results in \cite{serfaty2020meanfield,rosenzweigserfaty2023global,rosenzweigserfaty2026sharp} require the stated regularity of the limiting density and velocity, with diagonal subtraction and finite-$N$ corrections in particle comparisons. The energy-kernel commutator used here is provided by \cite{rosenzweig2026wassersteingradientflowsmaximum}. Coupling jump systems with independent nonlinear processes driven by the same Poisson measures is an established strategy \cite{andreis2018mckean}. Here the combination is essential: the modulated energy controls the singular drift, while the bounded coupling cost also handles the jump mismatch. Appendix~\ref{app:particle-jump-obstruction} rules out the particular linear energy-only estimate that would omit this second control.

\subsection{Organization of the paper}
\label{subsec:organization}
Section~\ref{sec:continuum} constructs the continuum solution and proves stability and uniqueness. Section~\ref{subsec:stationary-long-time} establishes the energy identity, measure-level jump rigidity, and target convergence, followed by the stopped far-field estimate. Section~\ref{sec:particle-approximation} constructs the particle process and proves the mean-field estimate and its consequences. Section~\ref{sec:conclusion} summarizes the results and discusses further questions. Appendix~\ref{app:auxiliary-estimates} supplies the positive linear jump-transport equation, product, and empirical sampling estimates used in the proofs; Appendix~\ref{app:particle-jump-obstruction} disproves a linear upper bound on the particle jump comparison remainder in terms of modulated energy and the stated finite-particle correction.

\subsection{Acknowledgments}
The research behind this paper was primarily completed during the Summer Undergraduate Applied Mathematics Institute (SUAMI) under the auspices of the Mellon College of Science Summer Scholars Program at Carnegie Mellon University. The authors thank the staff, especially Jonathan Simone, for their support and hospitality.

\subsection{Use of AI}
The authors used generative-AI tools (OpenAI’s ChatGPT) during the development and preparation of this paper to identify potentially relevant literature, explore possible
proof strategies, test draft arguments for errors or gaps, improve exposition, and check internal consistency and cross-references. Outputs from these tools were treated as unverified suggestions. The
authors followed up on literature suggestions by consulting relevant sources and worked through
mathematical suggestions before deciding whether to incorporate them into the paper. Recognizing
that these tools may suggest ideas without reliably indicating their provenance, the authors made
particular efforts to identify and cite relevant antecedents in the literature. The authors made all
final decisions concerning the manuscript and take full responsibility for its contents.

\section{Well-posedness and stability of the continuum dynamics}
\label{sec:continuum}
\label{sec:continuum-wellposedness}
Unless otherwise stated, we assume the hypotheses of
Theorem~\ref{thm:continuum-wellposedness}. The auxiliary lemmas are stated
under their respective assumptions.
The transport estimates follow the aggregation-equation theory of
\cite{BLR11,Lau07} and its application to the MMD velocity in
\cite[Section~3]{rosenzweig2026wassersteingradientflowsmaximum}.
We regularize the kernel, freeze the dependence on the density, and solve
the resulting positive linear equation. A fixed-point argument then gives
the nonlinear regularized solution. The main additional issue is to control
the jump source uniformly in the regularization: its inflow can increase
the density, and its transitions can increase moments. These bounds must
be closed together before compactness can be used. Finally, a coupling
estimate identifies the limit uniquely. We first carry out this program
for finite \(p\); the bounded-density and one-dimensional endpoints are
treated in the proof of the main theorem.

For a probability measure \(\rho\geq 0\) and \(s\ge0\),
we distinguish its power and bracket moments:
\begin{equation}
\label{eq:moment}
  \mathcal M_s(\rho):=\int_{\mathbb R^d}|x|^s\,d\rho(x),
  \qquad
  \widetilde{\mathcal M}_s(\rho):=
     \int_{\mathbb R^d}\langle x\rangle^s\,d\rho(x),
  \qquad \langle x\rangle:=(1+|x|^2)^{1/2}.
\end{equation}
Note that 
\begin{equation}
\label{eq:power-bracket-moment-comparison}
  \mathcal M_s(\rho)\le\widetilde{\mathcal M}_s(\rho)
  \le2^{(s/2-1)_+}\bigl(1+\mathcal M_s(\rho)\bigr),
\end{equation}
and hence the two moment conditions are equivalent, but we distinguish the corresponding moment quantities when needed. We now give the precise solution concept.

The regularity class is fixed as follows. On a finite interval $[0,T]$, require
\begin{equation}\label{eq:weak-solution-common-class}
 \rho\in C([0,T];\mathcal P_1(\mathbb R^d)),\qquad
 \sup_{0\le t\le T}\mathcal M_r(\rho_t)<\infty,
\end{equation}
where $\mathcal P_1$ carries $W_1$, and require the associated velocity to satisfy
\begin{equation}\label{eq:weak-solution-velocity-class}
 v_{\rho_t}=-\nabla K*(\rho_t-\mu),\qquad
 v_\rho\in L^\infty(0,T;W^{1,\infty}(\mathbb R^d;\mathbb R^d))
       \cap C([0,T]\times\mathbb R^d;\mathbb R^d).
\end{equation}
For finite $p$, the density and jump source satisfy
\begin{equation}\label{eq:finite-p-weak-solution-class}
 \rho\in C([0,T];L^p)\cap C^1([0,T];W^{-1,p}),\qquad
 J(\rho)\in C([0,T];L^p).
\end{equation}
For $p=\infty$, replace \eqref{eq:finite-p-weak-solution-class} by
\begin{equation}\label{eq:infinite-p-weak-solution-class}
 \rho\in C_{w^*}([0,T];L^\infty)
       \cap\bigcap_{1\le s<\infty}C([0,T];L^s),\qquad
 \rho\in L^\infty((0,T)\times\mathbb R^d).
\end{equation}
Here the spatial spaces are over $\mathbb R^d$, and $C_{w^*}$ means continuity
in the $\sigma(L^\infty,L^1)$ topology.

\begin{definition}[Weak solution]\label{def:continuum-weak-solution}
Under the hypotheses of Theorem~\ref{thm:continuum-wellposedness}, a curve
in the preceding regularity class is a weak solution of \eqref{eq:pde}
on $[0,T]$ if, for every $\varphi\in C_c^\infty([0,T)\times\mathbb R^d)$,
\begin{align}
0={}&\int_0^T\!\!\int_{\mathbb R^d}
 (\partial_t\varphi+v_{\rho_t}\cdot\nabla\varphi)\rho_t\,dx\,dt
 +\int_{\mathbb R^d}\varphi(0,x)\rho_0(x)\,dx\nonumber \\
&+\int_0^T\!\!\int_{\mathbb R^d}\!\!\int_{\mathbb R^d}
 (\varphi(t,x+z)-\varphi(t,x))\lambda(x)\zeta_{\rho_t}(x,x+z)
 \kappa(z)\,dz\,\rho_t(x)\,dx\,dt.\label{eq:weak-jump-transport-formulation}
\end{align}
A curve on $[0,\infty)$ is a global weak solution if its restriction to
every finite interval is a weak solution in this sense.
\end{definition}
Let \(\eta\in C_c^\infty(B_1)\) be nonnegative and even, with
\(\int\eta=1\), and set
\begin{equation}
\label{eq:kernel-regularization}
  \eta_\varepsilon(x)=\varepsilon^{-d}\eta(x/\varepsilon),
  \qquad K_\varepsilon=\eta_\varepsilon*K
  \quad(\varepsilon>0),
  \qquad K_0=K,
\end{equation}
\begin{equation}
  f_\rho^\varepsilon=K_\varepsilon*(\rho-\mu),\qquad
  v_\rho^\varepsilon=-\nabla f_\rho^\varepsilon,\qquad
  \zeta_\rho^\varepsilon(x,y)
  =[f_\rho^\varepsilon(x)-f_\rho^\varepsilon(y)]_+ .
\end{equation}
We use \(\varepsilon=0\) for the unregularized quantities.
For
\(\nu\in\mathcal P_q(\mathbb R^d)\), define the inflow/outflow
operators by
\begin{subequations}\label{eq:frozen-inflow-outflow-operators}
\begin{align}
  \left(J_{\nu,\mathrm{in}}^\varepsilon u\right)(y)
  &:={}
  \int_{\mathbb R^d}
  \lambda(x)u(x)\zeta_\nu^\varepsilon(x,y)
  \kappa(y-x)\,dx,\\
  \left(J_{\nu,\mathrm{out}}^\varepsilon u\right)(x)
  &:={}
  \lambda(x)u(x)
  \int_{\mathbb R^d}
  \zeta_\nu^\varepsilon(x,y)\kappa(y-x)\,dy.
\end{align}
\end{subequations}
Thus \( J^\varepsilon(\rho)
  =J_{\rho,\mathrm{in}}^\varepsilon\rho
  -J_{\rho,\mathrm{out}}^\varepsilon\rho.\)
As above, we omit the superscript when \(\varepsilon=0\). We next recall the following singular-convolution and regularized-kernel estimates from
\cite[Lemmas~3.2-3.3]{rosenzweig2026wassersteingradientflowsmaximum}.
\begin{lemma}
\label{lem:transport-input}
Assume the hypotheses of
Theorem~\ref{thm:continuum-wellposedness} with \(p<\infty\), and let
\begin{equation}
\label{eq:transport-exponents}
  p_*=\frac{p}{p-1},\qquad
  \theta_1=\frac{(1-q)p_*}{d},\qquad
  \theta_2=\frac{(2-q)p_*}{d}<1 .
\end{equation}
For every probability density
\(\nu\in\mathcal P_r(\mathbb R^d)\cap L^p(\mathbb R^d)\),
\begin{equation}
\label{eq:transport-gradient-bound}
  \|\nabla v_\nu^\varepsilon\|_{L^\infty}
  \le C_{d,q,p}
  \left(1+\|\nu\|_{L^p}^{\theta_2}
           +\|\mu\|_{L^p}^{\theta_2}\right)
\end{equation}
uniformly for \(0\le\varepsilon\le1\). Moreover,
\begin{equation}
\label{eq:transport-velocity-bound}
  \|v_\nu^\varepsilon\|_{L^\infty}
  \le
  \begin{cases}
    C_{d,q,p}\left(1+\|\nu\|_{L^p}^{\theta_1} +\|\mu\|_{L^p}^{\theta_1}\right),
    &0<q<1,\\
    2,&q=1,\\
    C_q\left(1+\mathcal M_{q-1}(\nu)
                 +\mathcal M_{q-1}(\mu)\right),
       &1<q<2,
  \end{cases}
\end{equation}
where the exponents are defined in \eqref{eq:transport-exponents}. Finally,
\begin{equation}
\label{eq:transport-laplacian-bound}
  \Delta K_\varepsilon\le0,\qquad
  \|(-\Delta K_\varepsilon)*\mu\|_{L^\infty}
  \le C_{d,q,p}\left(1+\|\mu\|_{L^p}^{\theta_2}\right).
\end{equation}
\end{lemma}
Potential increments give jump bounds uniform in the departure point,
even for singular measures. We then bound the frozen inflow and outflow
operators and verify mass balance.
\begin{lemma}
\label{lem:jump-increment}
Let $\mu,\nu\in\mathcal P_q(\mathbb R^d)$. Choose $C_q$ in \eqref{eq:acceptance-coefficient} sufficiently large, depending only on $q$. Then, uniformly for
\(0\le\varepsilon\le1\),
\begin{equation}
\label{eq:jump-increment-bound}
  0\le\zeta_\nu^\varepsilon(x,x+z)
  \le A_\nu|z|^{1\wedge q},
  \qquad x,z\in\mathbb R^d.
\end{equation}
In particular, if \(r\ge q\), \(R>0\),
\(\mu,\nu\in\mathcal P_r(\mathbb R^d)\), and
\(\mathcal M_r(\nu)\le R\), then
\begin{equation}
\label{eq:acceptance-coefficient-on-moment-ball}
  A_\nu\le 
  A_R:=
  \begin{cases}
    2,&0<q\le1,\\
    C_q\left(2+R+\mathcal M_r(\mu)\right),&1<q<2.
  \end{cases}
\end{equation}
\end{lemma}
\begin{proof}
If \(0<q\le1\), the \(q\)-H\"older continuity of \(K\) and
\eqref{eq:kernel-regularization} give
\begin{equation}
\label{eq:regularized-kernel-sublinear-increment}
  |K_\varepsilon(x)-K_\varepsilon(y)|
  \le |x-y|^q,
  \qquad x,y\in\mathbb R^d.
\end{equation}
Since \(\nu\) and \(\mu\) are probability measures,
\eqref{eq:regularized-kernel-sublinear-increment} implies
\begin{equation}
\label{eq:sublinear-potential-increment-bound}
  |f_\nu^\varepsilon(x)-f_\nu^\varepsilon(x+z)|
  \le2|z|^q.
\end{equation}
Suppose that \(1<q<2\). Since
\((\nu-\mu)(\mathbb R^d)=0\),
\begin{equation}
\label{eq:superlinear-potential-gradient-recentering}
  \nabla f_\nu^\varepsilon(x)
  =
  \int_{\mathbb R^d}
  \left(\nabla K_\varepsilon(x-y)-\nabla K_\varepsilon(x)\right)
  \,d(\nu-\mu)(y).
\end{equation}
The \((q-1)\)-H\"older seminorms of
\(\nabla K_\varepsilon\) are bounded uniformly for
\(0\le\varepsilon\le1\). Therefore,
\eqref{eq:superlinear-potential-gradient-recentering} yields
\begin{equation}
\label{eq:superlinear-potential-gradient-bound}
  \|\nabla f_\nu^\varepsilon\|_{L^\infty}
  \le C_q\left(
  \mathcal M_{q-1}(\nu)+\mathcal M_{q-1}(\mu)
  \right).
\end{equation}
The fundamental theorem of calculus and
\eqref{eq:superlinear-potential-gradient-bound} give
\begin{equation}
\label{eq:superlinear-potential-increment-bound}
  |f_\nu^\varepsilon(x)-f_\nu^\varepsilon(x+z)|
  \le C_q\left(
  \mathcal M_{q-1}(\nu)+\mathcal M_{q-1}(\mu)
  \right)|z|.
\end{equation}
Equations \eqref{eq:sublinear-potential-increment-bound} and
\eqref{eq:superlinear-potential-increment-bound} prove
\eqref{eq:jump-increment-bound}, with \(A_{\nu}\) defined in \eqref{eq:acceptance-coefficient}.
Equation \eqref{eq:acceptance-coefficient-on-moment-ball} follows from
Lyapunov's inequality, \eqref{eq:power-bracket-moment-comparison}, and
\(r\ge q\).
\end{proof}

\begin{lemma}
\label{lem:jump-inflow-outflow}
Let \(r\ge q\), \(R>0\), and
\(\mu,\nu\in\mathcal P_r(\mathbb R^d)\), with \(\nu\) satisfying
\(\mathcal M_r(\nu)\le R\). For every \(1\le s\le\infty\),
\(u\in L^s(\mathbb R^d)\), and \(0\le\varepsilon\le1\),
\begin{equation}
\label{eq:frozen-inflow-outflow-bound}
\|J_{\nu,\mathrm{in}}^\varepsilon u\|_{L^s}
  +\|J_{\nu,\mathrm{out}}^\varepsilon u\|_{L^s}
  \le
  2\|\lambda\|_{L^\infty}A_R\mathcal{M}_{1\wedge q}(\kappa)\|u\|_{L^s}.
\end{equation}
Both operators are positive, and \(J_{\nu,\mathrm{out}}^\varepsilon\) is a
nonnegative multiplication operator. Moreover, for every
\(u\in L^1(\mathbb R^d)\),
\begin{equation}
\label{eq:frozen-mass-balance}
  \int_{\mathbb R^d}
  \left(J_{\nu,\mathrm{in}}^\varepsilon u\right)(y)\,dy
  =
  \int_{\mathbb R^d}
  \left(J_{\nu,\mathrm{out}}^\varepsilon u\right)(x)\,dx.
\end{equation}
Consequently, \(J^\varepsilon(\rho)\) has zero total mass whenever
\(\rho\) is a probability density in \(\mathcal P_r(\mathbb R^d)\).
\end{lemma}
\begin{proof}
Equations \eqref{eq:jump-increment-bound} and
\eqref{eq:acceptance-coefficient-on-moment-ball} imply
\begin{align}
    \label{eq:frozen-outflow-pointwise-bound}
  \left|\left(J_{\nu,\mathrm{out}}^\varepsilon u\right)(x)\right|
  &\le
  \|\lambda\|_{L^\infty}A_R
  \mathcal{M}_{1\wedge q}(\kappa)|u(x)|, \\
  \label{eq:frozen-inflow-pointwise-bound}
  \left|\left(J_{\nu,\mathrm{in}}^\varepsilon u\right)(y)\right|
  &\le\|\lambda\|_{L^\infty}A_R
  \int_{\mathbb R^d}|u(y-z)||z|^{1\wedge q}\kappa(z)\,dz.
\end{align}
Young's inequality applied to \eqref{eq:frozen-inflow-pointwise-bound},
together with \eqref{eq:frozen-outflow-pointwise-bound}, proves
\eqref{eq:frozen-inflow-outflow-bound}. Positivity follows from
\eqref{eq:frozen-inflow-outflow-operators}, and
\eqref{eq:frozen-mass-balance} follows from Fubini's theorem.
\end{proof}

The linear equation with fixed velocity and jump rates is treated in
Lemma~\ref{lem:linear-jump-transport}, in Appendix~\ref{app:positive-linear}.
Transport characteristics give its mild formulation \cite[Section~8.1]{AGS08},
and separating inflow from outflow preserves positivity.
Equivalence with the distributional formulation will recover strong time
continuity after compactness.

Contracting the frozen solution map in \(W_1\) gives local existence;
uniform \(L^p\) and moment bounds yield global continuation. For \(q\le1\),
the density estimate closes first; for \(q>1\), moments must first control
the acceptance coefficient.
\begin{lemma}
\label{prop:regularized-jump-transport}
Assume the hypotheses of
Theorem~\ref{thm:continuum-wellposedness} with \(p<\infty\).
For every \(\varepsilon>0\), the equation
\begin{equation}
\label{eq:regularized-jump-transport}
  \partial_t\rho^\varepsilon
  +\nabla\!\cdot(\rho^\varepsilon v_{\rho^\varepsilon}^\varepsilon)
  =J^\varepsilon(\rho^\varepsilon),
  \qquad \rho^\varepsilon_{t=0}=\rho_0,
\end{equation}
has a unique global nonnegative, mass-preserving mild solution. For every
\(T>0\), there is a constant \(C_T<\infty\), independent of
\(\varepsilon\in(0,1]\), such that
\begin{equation}
\label{eq:regularized-uniform-bounds}
  \sup_{0\le t\le T}
  \left(
    \|\rho_t^\varepsilon\|_{L^p}
    +\mathcal M_r(\rho_t^\varepsilon)
    +\|v_{\rho_t^\varepsilon}^\varepsilon\|_{W^{1,\infty}}
  \right)
  \le C_T.
\end{equation}
The solutions also satisfy
\begin{equation}
\label{eq:regularized-W1-time-modulus}
  W_1(\rho_t^\varepsilon,\rho_s^\varepsilon)
  \le C_T|t-s|,
  \qquad 0\le s,t\le T.
\end{equation}
\end{lemma}

\begin{proof}
We divide the proof into three steps.

\emph{Step 1: Local construction.}
Fix \(R>\|\rho_0\|_{L^p}+\widetilde{\mathcal M}_r(\rho_0)+1\)
and, for \(\tau>0\), let
\begin{equation}
\label{eq:finite-p-fixed-point-space}
\begin{aligned}
\mathsf X_{\tau,R}:=\left\{
  \nu\in C([0,\tau];\mathcal P_1(\mathbb R^d)):\:\nu_0=\rho_0,\: \nu_t\in L^p(\mathbb R^d), \:\sup_{0\le t\le\tau}
  \left(\|\nu_t\|_{L^p}+\widetilde{\mathcal M}_r(\nu_t)\right)\le R
\right\}
\end{aligned}
\end{equation}
with metric
\begin{equation}
\label{eq:finite-p-fixed-point-metric}
  d_\tau(\nu,\omega):=\sup_{0\le t\le\tau}W_1(\nu_t,\omega_t).
\end{equation}
We claim that $(\textsf{X}_{\tau, R}, d_\tau)$ is complete: let $(\nu^{(n)})_n$ be a $d_\tau$-Cauchy sequence. By completeness of $(\mathcal P_1, W_1)$, there is $\nu\in C([0,\tau];\mathcal P_1(\mathbb R^d))$ with $\nu_0 = \rho_0$ such that $\nu^{(n)} \to \nu$ in $d_\tau$. When $t$ is fixed, the sequence $(\nu^{(n)}_t)_n$ is bounded in $L^p$, so Banach-Alaoglu gives a subsequence that converges weakly in $L^p$ to a limit. Since convergence in $W_1$ implies narrow convergence, we then use uniqueness of the distributional limit to deduce that this limit must be $\nu_t$; hence $\nu_t \in L^p$.  
Finally, lower semicontinuity of the $L^p$ norm and of $\widetilde{\mathcal M}_r$ demonstrates the third requirement.

By \eqref{eq:power-bracket-moment-comparison}, every
\(\nu\in\mathsf X_{\tau,R}\) satisfies \(\mathcal M_r(\nu_t)\le R\). For such
\(\nu\), consider
\begin{equation}
\label{eq:frozen-nonlinear-map-equation}
  \partial_tu+\nabla\!\cdot(uv_\nu^\varepsilon)
  =J_{\nu,\mathrm{in}}^\varepsilon u
  -J_{\nu,\mathrm{out}}^\varepsilon u,
  \qquad u(0)=\rho_0.
\end{equation}
By Lemma~\ref{lem:linear-jump-transport},
\eqref{eq:frozen-nonlinear-map-equation} has a unique solution, denoted by
\(\Phi_\varepsilon(\nu)\). The frozen analogues of
\eqref{eq:regularized-moment-differential-inequality} and
\eqref{eq:regularized-Lp-differential-inequality} imply, uniformly for
\(\nu\in\mathsf X_{\tau,R}\),
\begin{equation}
\label{eq:frozen-map-invariant-bound}
  \|\Phi_\varepsilon(\nu)_t\|_{L^p}
  +\widetilde{\mathcal M}_r(\Phi_\varepsilon(\nu)_t)
  \le
  \left(\|\rho_0\|_{L^p}
  +\widetilde{\mathcal M}_r(\rho_0)\right)
  e^{C_{\varepsilon,R}t},
  \qquad 0\le t\le\tau.
\end{equation}
The frozen version of \eqref{eq:regularized-dual-time-bound} also gives
\(W_1\)-continuity of \(\Phi_\varepsilon(\nu)\). Thus
\eqref{eq:frozen-map-invariant-bound} implies
\(\Phi_\varepsilon(\mathsf X_{\tau,R})
\subset\mathsf X_{\tau,R}\) for all sufficiently small \(\tau\).

For \(\nu,\omega\in\mathsf X_{\tau,R}\), set \(  c_\nu^\varepsilon(t,x,z)
  :=\lambda(x)\zeta_{\nu_t}^\varepsilon(x,x+z),\)
and define \(c_\omega^\varepsilon\) analogously. Start from the diagonal coupling \(  \Pi_0=(\mathrm{Id},\mathrm{Id})_\#\rho_0\)
and consider the common-proposal coupling with generator
\begin{align}
\mathcal L_t\Psi(x,y)
&=v_{\nu_t}^\varepsilon(x)\cdot\nabla_x\Psi
  +v_{\omega_t}^\varepsilon(y)\cdot\nabla_y\Psi
\nonumber \\
&+\int_{\R^d}
  \left(\Psi(x+z,y+z)-\Psi(x,y)\right)
  (c_\nu^\varepsilon(t,x,z)\wedge c_{\omega}^\varepsilon(t,y,z))
  \kappa(z)\,dz
\nonumber \\
&+\int_{\R^d}
  \left(\Psi(x+z,y)-\Psi(x,y)\right)
  [c_\nu^\varepsilon(t,x,z)-c_{\omega}^\varepsilon(t,y,z)]_+
  \kappa(z)\,dz
\nonumber \\
&+\int_{\R^d}
  \left(\Psi(x,y+z)-\Psi(x,y)\right)
  [c_{\omega}^\varepsilon(t,y,z)-c_\nu^\varepsilon(t,x,z)]_+
  \kappa(z)\,dz .\label{eq:common-proposal-generator}
\end{align} 
To construct this coupling, set
\begin{equation}
  h_R(z):=\|\lambda\|_{L^\infty}A_R|z|^{1\wedge q}\kappa(z),
  \qquad \Lambda:=\int_{\mathbb R^d}h_R(z)\,dz<\infty.
\end{equation}
Use a common Poisson proposal process with intensity
\(dt\,h_R(z)\,dz\,du\), where \(0\le u\le1\). At a proposal
\((t,z,u)\), the first coordinate accepts the displacement \(z\) if
\(u\le c_\nu^\varepsilon(t,x,z)\kappa(z)/h_R(z)\), and the second
coordinate uses the same mark with
\(c_\omega^\varepsilon(t,y,z)\kappa(z)/h_R(z)\); both ratios are set to
zero when \(h_R(z)=0\). These ratios lie in \([0,1]\) by
\eqref{eq:jump-increment-bound}. Between proposals, the coordinates
follow their respective transport flows. There are almost surely
finitely many proposals on every finite interval, so this construction
defines a process whose law \(\Pi_t\) is a narrowly continuous
probability solution with generator
\eqref{eq:common-proposal-generator}.

The frozen equations are also unique among narrowly continuous
probability-measure solutions. Indeed, their characteristic formula
holds with the time integral interpreted weakly against test functions.
Transport pushforward preserves total variation, while the jump
operator has total-variation norm at most \(2\Lambda\). Hence two
solutions \(\sigma,\widetilde\sigma\) with the same initial datum satisfy
\begin{equation}
  \|\sigma_t-\widetilde\sigma_t\|_{\mathrm{TV}}
  \le 2\Lambda\int_0^t
       \|\sigma_s-\widetilde\sigma_s\|_{\mathrm{TV}}\,ds,
\end{equation}
and Gronwall's inequality gives uniqueness.
Testing the coupled equation by functions of either coordinate,
with cutoffs in the other, shows that its marginals solve the
corresponding frozen equations. This uniqueness therefore identifies
them as \(\Phi_\varepsilon(\nu)_t\) and
\(\Phi_\varepsilon(\omega)_t\).

Then $c^{\varepsilon}_\nu$ satisfies
\begin{align}
|c_\nu^\varepsilon(t,x,z)&-c_{\omega}^\varepsilon(t,y,z)|
\nonumber \\
&\quad\le
\|\lambda\|_{L^\infty}
|\zeta_{\nu_t}^\varepsilon(x,x+z)
 -\zeta_{\omega_t}^\varepsilon(y,y+z)|
+\|\nabla\lambda\|_{L^\infty} A_R|z|^{1\wedge q}|x-y|,\label{eq:frozen-effective-rate-difference}
\end{align}
and the bound \(\|D^2K_\varepsilon\|_{L^\infty}<\infty\) gives
\begin{equation}
\label{eq:frozen-acceptance-rate-difference}
  |\zeta_{\nu_t}^\varepsilon(x,x+z)
 -\zeta_{\omega_t}^\varepsilon(y,y+z)|
  \le C_{\varepsilon,R}|z|
  \left(|x-y|+W_1(\nu_t,\omega_t)\right).
\end{equation}
With \(D\) defined by
\begin{equation}
\label{eq:frozen-coupling-distance}
  D(t):=\int_{\R^d\times \R^d}|x-y|\,d\Pi_t(x,y),
\end{equation}
testing the coupled equation~\eqref{eq:common-proposal-generator} against smooth truncations of \(\Psi(x, y) = |x - y|\) and then using 
\eqref{eq:frozen-effective-rate-difference} and
\eqref{eq:frozen-acceptance-rate-difference} yields 
\begin{equation} 
\label{eq:frozen-coupling-differential-inequality}
  D'(t)\le C_{\varepsilon,R}
  \left(D(t)+W_1(\nu_t,\omega_t)\right).
\end{equation}
Since $\Pi_0$ is the diagonal coupling, \(D(0)=0\), so
\eqref{eq:frozen-coupling-differential-inequality} gives
\begin{equation}
\label{eq:frozen-map-contraction-bound}
  \sup_{0\le t\le\tau}
  W_1(\Phi_\varepsilon(\nu)_t,\Phi_\varepsilon(\omega)_t)
  \le
  \tau C_{\varepsilon,R}e^{C_{\varepsilon,R}\tau}d_\tau(\nu,\omega).
\end{equation}
After decreasing \(\tau\), \eqref{eq:frozen-map-contraction-bound} shows that
\(\Phi_\varepsilon\) is a contraction.
Banach's theorem gives a unique local solution, and
Lemmas~\ref{lem:jump-inflow-outflow} and~\ref{lem:linear-jump-transport} give positivity and mass conservation.

\emph{Step 2: Uniform estimates.}
It remains to derive estimates that are uniform in \(\varepsilon\). We apply
the transport-moment argument from
\cite[proof of Lemma~3.5]{rosenzweig2026wassersteingradientflowsmaximum}
to $x \mapsto \langle x\rangle ^r$. 
We shall repeatedly use the interpolation inequality
\begin{equation}
\label{eq:bracket-moment-product-bound}
  \widetilde{\mathcal M}_a(\nu)
  \widetilde{\mathcal M}_b(\nu)
  \le
  \widetilde{\mathcal M}_r(\nu),
  \qquad
  a,b\ge0,\qquad a+b\le r.
\end{equation}
Since
\(
  |\nabla\langle x\rangle^r|
  \le C_r\langle x\rangle^{r-1},
\)
the transport contribution satisfies
\begin{equation}
\label{eq:transport-bracket-moment-contribution}
\begin{aligned}
\left|
  \int_{\mathbb R^d}
  \nabla\langle x\rangle^r\cdot
  v_{\rho^\varepsilon}^\varepsilon(x)
  \,d\rho^\varepsilon(x)
\right|
\le
\begin{cases}
C\left(1+\|\rho^\varepsilon\|_{L^p}^{\theta_1}\right)
\widetilde{\mathcal M}_r(\rho^\varepsilon),
&0<q<1,\\[1mm]
C\widetilde{\mathcal M}_r(\rho^\varepsilon),
&1\le q<2.
\end{cases}
\end{aligned}
\end{equation}
For \(0<q<1\), this follows directly from the first estimate in
\eqref{eq:transport-velocity-bound}. When \(q=1\), it follows from
\(\|v_{\rho^\varepsilon}^\varepsilon\|_{L^\infty}\le2\). Finally, when
\(1<q<2\), \eqref{eq:transport-velocity-bound} gives
\(
  \|v_{\rho^\varepsilon}^\varepsilon\|_{L^\infty}
  \le
  C\left(1+\mathcal M_{q-1}(\rho^\varepsilon)\right),
\)
where the contribution of \(\mu\) has been absorbed into \(C\). Therefore,
by \eqref{eq:power-bracket-moment-comparison} and
\eqref{eq:bracket-moment-product-bound},
\begin{equation}
\label{eq:superlinear-transport-moment-product}
\begin{aligned}
&\left(1+\mathcal M_{q-1}(\rho^\varepsilon)\right)
\widetilde{\mathcal M}_{r-1}(\rho^\varepsilon)\le
C\widetilde{\mathcal M}_{q-1}(\rho^\varepsilon)
 \widetilde{\mathcal M}_{r-1}(\rho^\varepsilon)
\le
C\widetilde{\mathcal M}_r(\rho^\varepsilon).
\end{aligned}
\end{equation}
The last inequality is valid because
\(
  (q-1)+(r-1)\le r
\).

We next estimate the jump contribution. Note
\begin{equation}
\label{eq:moment-weight-increment}
  |\langle x+z\rangle^r - \langle x\rangle^r|
  \le
  C_r|z|
  \left(
    \langle x\rangle^{r-1}
    +\langle z\rangle^{r-1}
  \right).
\end{equation}
Combining \eqref{eq:jump-increment-bound} with
\eqref{eq:moment-weight-increment}, we obtain
\begin{align}
&\left|
  \int_{\mathbb R^d}\int_{\mathbb R^d}
  \lambda(x)
  \left(\langle x+z\rangle^r - \langle x\rangle^r\right)
  \zeta_{\rho^\varepsilon}^\varepsilon(x,x+z)
  \kappa(z)\,dz\,d\rho^\varepsilon(x)
\right|
\nonumber \\
&\quad\le
C
\left(1+\mathcal M_{(q-1)_+}(\rho^\varepsilon)\right)
\left(1+\widetilde{\mathcal M}_{r-1}(\rho^\varepsilon)\right)
\int_{\mathbb R^d}
|z|^{1+(1\wedge q)}
\left(1+\langle z\rangle^{r-1}\right)
\kappa(z)\,dz
\nonumber \\
&\quad\le
C\widetilde{\mathcal M}_r(\rho^\varepsilon),\label{eq:jump-moment-contribution}
\end{align}
where we have absorbed the moments of $\kappa$ into the constant.

Combining \eqref{eq:transport-bracket-moment-contribution} and
\eqref{eq:jump-moment-contribution} gives
\begin{equation}
\label{eq:preliminary-regularized-moment-differential-inequality}
  \frac{d}{dt}
  \widetilde{\mathcal M}_r(\rho_t^\varepsilon)
  \le
  \begin{cases}
  C\left(1+\|\rho_t^\varepsilon\|_{L^p}^{\theta_1}\right)
  \widetilde{\mathcal M}_r(\rho_t^\varepsilon),
  &0<q<1,\\[1mm]
  C\widetilde{\mathcal M}_r(\rho_t^\varepsilon),
  &1\le q<2.
  \end{cases}
\end{equation}
When \(0<q<1\), we first close the \(L^p\)-estimate
\eqref{eq:regularized-Lp-differential-inequality}; its coefficient is
independent of the moments because
\(A_{\rho_t^\varepsilon}=2\). The resulting finite-time \(L^p\)-bound then
closes the first case of
\eqref{eq:preliminary-regularized-moment-differential-inequality}. When
\(1\le q<2\), the moment inequality already closes independently of the
\(L^p\)-norm. Thus, in either case,
\begin{equation}
\label{eq:regularized-moment-differential-inequality}
  \frac{d}{dt}
  \widetilde{\mathcal M}_r(\rho_t^\varepsilon)
  \le
  C_T\widetilde{\mathcal M}_r(\rho_t^\varepsilon),
  \qquad 0\le t\le T.
\end{equation}

For the \(L^p\) estimate, a renormalized test by
\((\rho^\varepsilon)^{p-1}\) gives
\begin{align}
\frac{d}{dt}\|\rho^\varepsilon\|_{L^p}^p
&=(p-1)\int_{\R^d}(\rho^\varepsilon)^p
  \Delta K_\varepsilon*(\rho^\varepsilon-\mu)\,dx\nonumber \\
&\quad
+p\int_{\R^d}(\rho^\varepsilon)^{p-1}
  \left(J_{\rho^\varepsilon,\mathrm{in}}^\varepsilon
  \rho^\varepsilon\right)\,dx
-p\int_{\R^d}(\rho^\varepsilon)^{p-1}
  \left(J_{\rho^\varepsilon,\mathrm{out}}^\varepsilon
  \rho^\varepsilon\right)\,dx\label{eq:regularized-Lp-identity}
\end{align}
and since \(\Delta K_\varepsilon\le 0\) by
\eqref{eq:transport-laplacian-bound} and
\(J_{\rho^\varepsilon,\mathrm{out}}^\varepsilon\) is positive, we have
\begin{equation}
\label{eq:Lp-nonpositive-contributions}
\begin{aligned}
  (p-1)\int_{\R^d}(\rho^\varepsilon)^p
  (\Delta K_\varepsilon*\rho^\varepsilon)\,dx\le0,\quad
  -p\int_{\R^d}(\rho^\varepsilon)^{p-1}
  \left(J_{\rho^\varepsilon,\mathrm{out}}^\varepsilon
  \rho^\varepsilon\right)\,dx\le0.
\end{aligned}
\end{equation}
The remaining jump inflow satisfies
\begin{equation}
\label{eq:jump-inflow-Lp-bound}
  \int_{\R^d}(\rho^\varepsilon)^{p-1}
  \left(J_{\rho^\varepsilon,\mathrm{in}}^\varepsilon
  \rho^\varepsilon\right)\,dx
  \le
  \|\lambda\|_{L^\infty}A_{\rho^\varepsilon}\mathcal{M}_{1\wedge q}(\kappa)
  \|\rho^\varepsilon\|_{L^p}^p.
\end{equation}
It follows from \eqref{eq:transport-laplacian-bound},
\eqref{eq:regularized-Lp-identity}, and \eqref{eq:jump-inflow-Lp-bound} that
\begin{equation}
\label{eq:regularized-Lp-differential-inequality}
\begin{aligned}
  \frac{d}{dt}\|\rho_t^\varepsilon\|_{L^p}^p
  \le{}&\Bigl[
  (p-1)
  \|(-\Delta K_\varepsilon)*\mu\|_{L^\infty}
  +p\|\lambda\|_{L^\infty}A_{\rho_t^\varepsilon}\mathcal{M}_{1\wedge q}(\kappa)
  \Bigr]
  \|\rho_t^\varepsilon\|_{L^p}^p.
\end{aligned}
\end{equation}
To close \eqref{eq:regularized-moment-differential-inequality} and
\eqref{eq:regularized-Lp-differential-inequality}, distinguish the two cases
as in \cite[Section~3.2]{rosenzweig2026wassersteingradientflowsmaximum}.
For \(q\le1\), \eqref{eq:acceptance-coefficient} gives
\(A_{\rho_t^\varepsilon}=2\), so
\eqref{eq:regularized-Lp-differential-inequality} closes independently of the
moment and then \eqref{eq:regularized-moment-differential-inequality} closes.
For \(q>1\), \eqref{eq:regularized-moment-differential-inequality} closes
independently of the \(L^p\)-norm; its bound controls
\(A_{\rho_t^\varepsilon}\) in
\eqref{eq:regularized-Lp-differential-inequality}. Gronwall's inequality in
both cases proves the \(L^p\)- and moment components of
\eqref{eq:regularized-uniform-bounds};
the bounds
\eqref{eq:transport-gradient-bound} and
\eqref{eq:transport-velocity-bound}, together with these \(L^p\)- and
moment estimates, give the velocity component of
\eqref{eq:regularized-uniform-bounds}. The estimates are justified for
the mild solution by approximation of the data and proposal density,
renormalization, and cutoff; the bounds are uniform in the approximation.

Set
\begin{equation}
\label{eq:uniform-acceptance-coefficient}
  A_T:=\sup_{0<\varepsilon\le1}\sup_{0\le t\le T}
  A_{\rho_t^\varepsilon}<\infty.
\end{equation}
The coefficient in \eqref{eq:uniform-acceptance-coefficient} is finite by
\eqref{eq:regularized-uniform-bounds} and
\eqref{eq:acceptance-coefficient} (and equals \(2\) when \(q\le1\)). Finally,
testing \eqref{eq:regularized-jump-transport} by smooth truncations of a
\(1\)-Lipschitz function and then removing the truncation gives
\begin{equation}
\label{eq:regularized-dual-time-bound}
  \left|\frac{d}{dt}\int\varphi\,d\rho_t^\varepsilon\right|
  \le
  \|v_{\rho_t^\varepsilon}^\varepsilon\|_{L^\infty}
  +\|\lambda\|_{L^\infty}A_T\mathcal{M}_{2 \wedge (q+1)}(\kappa)
  .
\end{equation}
Taking the supremum in \eqref{eq:regularized-dual-time-bound} over
\(1\)-Lipschitz \(\varphi\) yields
\eqref{eq:regularized-W1-time-modulus}.

\emph{Step 3: Global continuation.}
The local construction may be restarted at any time at which the \(L^p\)-norm
and bracket \(r\)-moment are finite. The bounds in
\eqref{eq:regularized-uniform-bounds} rule out blowup of either quantity on a
finite interval; hence the maximal regularized solution is global.
\end{proof}

The uniform estimates now provide compactness of the regularized curves.
At \(r=1\), a common proposal process supplies the uniform integrability
needed for \(W_1\) compactness. The jump source passes to the limit by
combining local convergence of potential increments with control of large
displacements.

\begin{lemma}
\label{prop:limit-regularized-jump-transport}
Assume the hypotheses of
Theorem~\ref{thm:continuum-wellposedness} with \(p<\infty\).
For every \(T>0\), the solutions from
\Cref{prop:regularized-jump-transport} have a subsequence
\(\varepsilon_n\downarrow0\) and a limit \(\rho\) such that
\begin{equation}
\label{eq:regularized-W1-convergence}
  \rho^{\varepsilon_n}\to\rho
  \quad\text{in }C([0,T];\mathcal P_1(\mathbb R^d))
  \text{ with respect to }W_1,
\end{equation}
and
\begin{equation}
\label{eq:regularized-weak-Lp-convergence}
  \rho^{\varepsilon_n}\rightharpoonup\rho
  \quad\text{weakly in }L^p((0,T)\times\mathbb R^d).
\end{equation}
The limit is a weak solution in the sense of
Definition~\ref{def:continuum-weak-solution}, and
\begin{equation}
\label{eq:limit-apriori-bounds}
  \sup_{0\le t\le T}
  \left(
    \|\rho_t\|_{L^p}
    +\widetilde{\mathcal M}_r(\rho_t)
    +\mathcal M_r(\rho_t)
  \right)
  \le C_T.
\end{equation}
Moreover,
\begin{equation}
\label{eq:limit-solution-regularity}
  \rho\in C([0,T];L^1(\mathbb R^d)\cap L^p(\mathbb R^d))
  \cap C^1([0,T];W^{-1,p}(\mathbb R^d)).
\end{equation}
The limiting velocity satisfies
\begin{equation}
\label{eq:limit-velocity-regularity}
  v_\rho\in L^\infty(0,T;W^{1,\infty}(\mathbb R^d;\mathbb R^d))
  \cap C([0,T];L^\infty(\mathbb R^d;\mathbb R^d)).
\end{equation}
\end{lemma}
\begin{proof}
If \(r>1\), the uniform moment bound
\eqref{eq:regularized-uniform-bounds} gives uniform integrability of the
first moments. In the remaining case \(r=1\), the assumption \(r\ge q\) forces
\(q\le1\). The velocity bound \eqref{eq:transport-velocity-bound} controls the
transport displacement, and
\begin{equation}
\label{eq:dominating-jump-rate}
  c_{\rho^\varepsilon}(t,x,z)
  =\lambda(x)\zeta_{\rho_t^\varepsilon}^\varepsilon(x,x+z)
  \le2\|\lambda\|_{L^\infty}|z|^q.
\end{equation}
By \eqref{eq:dominating-jump-rate}, the standard thinning construction
therefore couples every accepted jump to
a compound-Poisson process with L\'evy measure
\begin{equation}
\label{eq:dominating-Levy-measure-r1}
  \nu_J(dz):=2\|\lambda\|_{L^\infty}|z|^q\kappa(z)\,dz.
\end{equation}
Its accumulated absolute displacement on \([0,T]\) is integrable because
\begin{equation}
\label{eq:dominating-Levy-first-moment}
  \int_{\R^d}|z|\,\nu_J(dz)
  =2\|\lambda\|_{L^\infty}\mathcal{M}_{q+1}(\kappa)<\infty.
\end{equation}
The same measure has finite total mass, since
\begin{equation}
\label{eq:dominating-Levy-total-mass}
  \nu_J(\mathbb R^d)
  \le2\|\lambda\|_{L^\infty}\left(1+\mathcal{M}_2(\kappa)\right)<\infty.
\end{equation}
More precisely, take a common initial variable \(X_0\) of law \(\rho_0\)
and an independent Poisson random measure \(\mathcal N_J\) with intensity
\(dt\,\nu_J(dz)\,da\) on
\([0,T]\times\mathbb R^d\times[0,1]\).
Between its proposal times, let \(X^\varepsilon\) follow
\(v_{\rho_t^\varepsilon}^\varepsilon\). At a point \((t,z,a)\), add \(z\)
precisely when
\begin{equation}
  a\le
  \frac{c_{\rho^\varepsilon}(t,X^\varepsilon(t-),z)}
       {2\|\lambda\|_{L^\infty}|z|^q},
\end{equation}
with the ratio set to zero when the denominator vanishes.
The ratio lies in \([0,1]\) by \eqref{eq:dominating-jump-rate}.
Transport between proposals gives the measure-valued forward equation
with coefficients frozen at \(\rho^\varepsilon\). The measure-valued
uniqueness argument in the proof of
\Cref{prop:regularized-jump-transport} therefore identifies
\(\operatorname{Law}(X^\varepsilon(t))=\rho_t^\varepsilon\).
If
\begin{equation}
\label{eq:uniform-transport-speed}
  V_T:=\sup_{0<\varepsilon\le1}\sup_{0\le t\le T}
  \|v_{\rho_t^\varepsilon}^\varepsilon\|_{L^\infty}<\infty,
\end{equation}
which is finite by \eqref{eq:regularized-uniform-bounds}, and
\(N(dt,dz):=\mathcal N_J(dt,dz,[0,1])\) is the projection onto proposal
times and displacements, define
\(S_T:=\int_{[0,T]\times\mathbb R^d}|z|\,N(dt,dz)\). The common thinning
construction then gives
\begin{equation}
\label{eq:r1-common-integrable-dominator}
  \sup_{0<\varepsilon\le1}\sup_{0\le t\le T}|X^\varepsilon(t)|
  \le |X_0|+TV_T+S_T,
  \qquad
  \mathbb E[S_T]=T\int_{\R^d}|z|\,\nu_J(dz)<\infty.
\end{equation}
The right-hand side of \eqref{eq:r1-common-integrable-dominator} is an
integrable random variable independent of \(\varepsilon\). Thus the first
moments are uniformly integrable also when \(r=1\).
The time modulus \eqref{eq:regularized-W1-time-modulus} now gives
equicontinuity, so Arzel\`a-Ascoli yields
\eqref{eq:regularized-W1-convergence}. Weak compactness in \(L^p\) and
uniqueness of the narrow limit give \eqref{eq:regularized-weak-Lp-convergence}
for the same curve.

For each fixed \(t\), lower semicontinuity under the convergence in
\eqref{eq:regularized-W1-convergence}, the weak \(L^p\)-compactness furnished
by \eqref{eq:regularized-uniform-bounds}, and uniqueness of the narrow limit
give
\begin{equation}
\label{eq:limit-bracket-and-Lp-bound}
  \|\rho_t\|_{L^p}
  +\widetilde{\mathcal M}_r(\rho_t)
  \le C_T.
\end{equation}
Taking the supremum in \eqref{eq:limit-bracket-and-Lp-bound} and using
\eqref{eq:power-bracket-moment-comparison} proves
\eqref{eq:limit-apriori-bounds}. In particular, the power-moment condition in
\eqref{eq:weak-solution-common-class} is inherited by the limit.

For completeness, we identify the velocity uniformly in time. Fix a compact
set \(Q\Subset\mathbb R^d\) and use smooth cutoffs to split each velocity
convolution into near, intermediate, and far parts corresponding to
\begin{equation}
\label{eq:velocity-near-middle-far-decomposition}
  |x-y|\le\delta,\qquad
  \delta<|x-y|<M,\qquad
  |x-y|\ge M .
\end{equation}
The singular-convolution estimates of
\cite[Lemmas~3.2-3.3]{rosenzweig2026wassersteingradientflowsmaximum} applied to
\eqref{eq:velocity-near-middle-far-decomposition} bound the near field,
uniformly in \(t,n\), by
\(C_T\delta^{\alpha_0}\) for some \(\alpha_0>0\). On the intermediate
annulus, \(\nabla K_{\varepsilon_n}\to\nabla K\) uniformly, and the cut-off
kernels are uniformly bounded-Lipschitz; hence the uniform \(W_1\)
convergence of the measures gives uniform convergence of this part. Finally,
if \(\mathcal V_{n,M}^{\mathrm{far}}\) denotes the far-field contribution on
\(Q\), the power-moment bound in
\eqref{eq:regularized-uniform-bounds} gives
\begin{equation}
\label{eq:velocity-far-field-bound}
  \sup_{0\le t\le T}\sup_{x\in Q}
  |\mathcal V_{n,M}^{\mathrm{far}}(t,x)|
  \le C_{Q,T}M^{(q-1)_+-r}.
\end{equation}
Using \eqref{eq:velocity-far-field-bound}, and letting first
\(n\to\infty\), then \(\delta\downarrow0\), and finally
\(M\uparrow\infty\), gives \(v_{\rho_t^{\varepsilon_n}}^{\varepsilon_n}
  \to v_{\rho_t}\)
locally uniformly in \((t,x)\). This is the near-, intermediate-, and
far-field identification used in \cite[Section~3.3]{rosenzweig2026wassersteingradientflowsmaximum}, adapted here to
the present regularization. The gradient estimate
\eqref{eq:transport-gradient-bound} and
\eqref{eq:regularized-uniform-bounds}, followed by weak-star lower
semicontinuity, give
\(v_\rho\in L^\infty(0,T;W^{1,\infty})\). The convergence also identifies
the transport term.

We record the additional jump argument. For a test function \(\varphi\), set
\begin{equation}
\label{eq:limit-acceptance-coefficient}
  A_T:=\sup_n\sup_{0\le t\le T}A_{\rho_t^{\varepsilon_n}}<\infty,
\end{equation}
and define
\begin{equation}
\label{eq:jump-test-increment-and-rate}
  A_\varphi(t,x,z)=\varphi(t,x+z)-\varphi(t,x),\quad 
  R_n(t,x,z)=
  [f_{\rho_t^{\varepsilon_n}}^{\varepsilon_n}(x)
   -f_{\rho_t^{\varepsilon_n}}^{\varepsilon_n}(x+z)]_+ .
\end{equation}
With the notation in \eqref{eq:jump-test-increment-and-rate} and the bound
\eqref{eq:jump-increment-bound}, the coefficient
\eqref{eq:limit-acceptance-coefficient} gives, uniformly in \(n\),
\begin{align}
&\int_0^T\int_{|z|>L}\int_{\R^d}
 |A_\varphi(t,x,z)|\lambda(x)R_n(t,x,z)
 \,d\rho_t^{\varepsilon_n}(x)\kappa(z)\,dz\,dt
\nonumber \\
&\hspace{8em}\le
2T\|\varphi\|_{L^\infty}\|\lambda\|_{L^\infty}A_T
\int_{|z|>L}|z|^{1\wedge q}\kappa(z)\,dz,\label{eq:jump-tail-bound}
\end{align}
and the right-hand side tends to zero as \(L\to\infty\).

For \(|z|\le L\), only a fixed compact set of \(x\)'s contributes.
The convergence \(v_{\rho_t^{\varepsilon_n}}^{\varepsilon_n}
  \to v_{\rho_t}\) implies
uniform convergence of the potential increments, since
\begin{equation}
\label{eq:potential-increment-convergence-bound}
\begin{aligned}
&\left|
(f_n(x)-f_n(x+z))-(f(x)-f(x+z))
\right|\le
|z|\int_0^1
|v_n(x+\theta z)-v(x+\theta z)|\,d\theta.
\end{aligned}
\end{equation}
For \(|z|\le L\), consider the functions \(x\mapsto
  A_\varphi(t,x,z)\lambda(x)
  [f_{\rho_t}(x)-f_{\rho_t}(x+z)]_+\), which
 are supported in a fixed compact
set and, uniformly in \((t,z)\), are bounded
and Lipschitz in \(x\). Here we use
\(\lambda\in W^{1,\infty}\) and the uniform Hessian bound on
\(f_{\rho_t}\). The convergence in
\eqref{eq:regularized-W1-convergence} therefore controls the measure term
uniformly on compact \(z\)-sets. Combining this with
\eqref{eq:jump-tail-bound}--\eqref{eq:potential-increment-convergence-bound} proves convergence of the jump term in
\eqref{eq:weak-jump-transport-formulation}.

Finally, freeze the limiting coefficients in
Lemma~\ref{lem:linear-jump-transport}. The distributional-to-mild
identification gives \( \rho\in C([0,T];L^1(\mathbb R^d))
  \cap C([0,T];L^p(\mathbb R^d))\).
We next verify the time continuity needed for the final regularity upgrade.
If \(0<q<1\), the
singular-convolution estimate with \(a=1-q\) gives
\begin{equation}
\label{eq:sublinear-velocity-time-difference}
  \|v_{\rho_t}-v_{\rho_s}\|_{L^\infty}
  \le
  C\|\rho_t - \rho_s\|_{L^1}^{1-\theta_1}
   \|\rho_t - \rho_s\|_{L^p}^{\theta_1}
  \to0
\end{equation}
as \(s\to t\). If \(q=1\), convergence to zero follows from
\begin{equation}
\label{eq:linear-velocity-time-difference}
  \|v_{\rho_t}-v_{\rho_s}\|_{L^\infty}
  \le\|\rho_t - \rho_s\|_{L^1}.
\end{equation}
If \(1<q<2\), the fact that $\rho_t - \rho_s$ has zero total mass, together with the global
\((q-1)\)-H\"older continuity of \(\nabla K\), implies, for \(M>1\),
\begin{equation}
\label{eq:superlinear-velocity-time-difference}
  \|v_{\rho_t}-v_{\rho_s}\|_{L^\infty}
  \le
  C_M\|\rho_t-\rho_s\|_{L^1}
  +C_T M^{q-1-r}.
\end{equation}
The strong continuity \( \rho\in C([0,T];L^1(\mathbb R^d))
  \cap C([0,T];L^p(\mathbb R^d))\) gives the
limit in \eqref{eq:sublinear-velocity-time-difference} and
\eqref{eq:linear-velocity-time-difference}; in
\eqref{eq:superlinear-velocity-time-difference}, first let \(s\to t\) and
then \(M\to\infty\). Thus in every case \( v_\rho\in C([0,T];L^\infty(\mathbb R^d;\mathbb R^d))\),
and consequently \( \rho v_\rho\in C([0,T];L^p).\)

It remains to treat the jump source. The one-Lipschitz property of the
positive part gives
\begin{equation}
\label{eq:acceptance-time-difference}
  |\zeta_{\rho_t}(x,x+z)-\zeta_{\rho_s}(x,x+z)|
  \le |z|\|v_{\rho_t}-v_{\rho_s}\|_{L^\infty}.
\end{equation}
Then, \eqref{eq:acceptance-time-difference} gives
\begin{equation}
\label{eq:inflow-outflow-operator-time-continuity}
  \|J_{\rho_t,\mathrm{in}}-J_{\rho_s,\mathrm{in}}\|_{\mathcal L(L^p)}
  +\|J_{\rho_t,\mathrm{out}}-J_{\rho_s,\mathrm{out}}\|_{\mathcal L(L^p)}
  \le
  2\|\lambda\|_{L^\infty}\mathcal{M}_1(\kappa)
  \|v_{\rho_t}-v_{\rho_s}\|_{L^\infty}.
\end{equation}
Together, \eqref{eq:inflow-outflow-operator-time-continuity}, the uniform
operator bounds in \eqref{eq:frozen-inflow-outflow-bound}, and the strong continuity of \(\rho\)
give, with \(A_T=\sup_{t\le T}A_{\rho_t}\),
\begin{align}
\|J(\rho_t)&-J(\rho_s)\|_{L^p}
\le \nonumber \\ &2\|\lambda\|_{L^\infty}A_T\mathcal{M}_{1\wedge q}(\kappa)
       \|\rho_t-\rho_s\|_{L^p}
+2\|\lambda\|_{L^\infty}\mathcal{M}_1(\kappa)\|\rho_s\|_{L^p}
  \|v_{\rho_t}-v_{\rho_s}\|_{L^\infty}.\label{eq:jump-source-time-difference}
\end{align}
\(v_\rho\in C([0,T];L^\infty(\mathbb R^d;\mathbb R^d))\) and
\eqref{eq:jump-source-time-difference} imply \(J(\rho)\in C([0,T];L^p)\).
Since \(\rho v_\rho\in C([0,T];L^p)\) and
\(J(\rho)\in C([0,T];L^p)\), the right-hand side is
continuous as a \(W^{-1,p}\)-valued map. This yields the
\(C^1\)-regularity in \eqref{eq:limit-solution-regularity}. Together with
the preceding weak-star gradient bound,
\( v_\rho\in C([0,T];L^\infty(\mathbb R^d;\mathbb R^d))\) also proves
\eqref{eq:limit-velocity-regularity}.
\end{proof}
Uniqueness follows from a Dobrushin coupling argument
\cite{dobrushin1979vlasov}: evolve a coupling
of the initial laws and estimate its mean separation. Simultaneous accepted jumps preserve this
separation, so the jump contribution comes only from mismatched
acceptances; density bounds control this contribution and the transport
term. 
\begin{proposition}[\(W_1\)-stability]
\label{prop:continuum-W1-stability}
Assume \(d+q-2>0\) and \(p<\infty\).
Let \(\rho^{(1)}\) and \(\rho^{(2)}\) be weak solutions
of \eqref{eq:pde} on \([0,T]\) such that
\begin{equation}
\label{eq:stability-solution-class}
  \rho^{(i)}\in C([0,T];\mathcal P_1(\mathbb R^d))
  \cap C([0,T];L^p(\mathbb R^d)),
  \qquad
  \sup_{0\le t\le T}\mathcal M_r(\rho_t^{(i)})<\infty.
\end{equation}
Then
\begin{equation}
\label{eq:W1-stability-estimate}
  W_1(\rho_t^{(1)},\rho_t^{(2)})
  \le e^{C_Tt}W_1(\rho_0^{(1)},\rho_0^{(2)}),
  \qquad 0\le t\le T,
\end{equation}
where \(C_T\) depends only on the fixed data and on uniform
\(L^p\)- and \(r\)-th moment bounds for \(\rho^{(1)}\)
and \(\rho^{(2)}\) on \([0,T]\).
\end{proposition}
\begin{proof}
Start from an optimal coupling of the initial data and apply the
common-proposal generator \eqref{eq:common-proposal-generator}, now with the
unregularized velocities and rates \(c_i(t,x,z)=\lambda(x)\zeta_{\rho_t^{(i)}}(x,x+z)\).
The total rates are bounded on \([0,T]\), so the common-proposal construction
in the proof of \Cref{prop:regularized-jump-transport}
gives a probability measure \(\Pi_t\). Testing its equation by functions of
\(x\) alone and of \(y\) alone, and using the measure-valued uniqueness
argument given there, verifies that its marginals are
\(\rho_t^{(1)}\) and \(\rho_t^{(2)}\). Set
\begin{equation}
\label{eq:stability-coupling-distance}
  D(t)=\int_{\mathbb R^d\times\mathbb R^d}|x-y|\,d\Pi_t(x,y).
\end{equation}
Testing the coupled equation by smooth approximations of \(|x-y|\) gives
\begin{align}
D'(t)\le{}&
\int_{\R^d\times \R^d}|v_{\rho_t^{(1)}}(x)-v_{\rho_t^{(2)}}(y)|\,d\Pi_t(x,y)
\nonumber \\
&+
\int_{\R^d\times \R^d}\int_{\R^d}
|z|\,|c_1(t,x,z)-c_2(t,y,z)|
\kappa(z)\,dz\,d\Pi_t(x,y).\label{eq:stability-coupling-inequality}
\end{align}
The transport estimate in the proof of
\cite[Proposition~3.7]{rosenzweig2026wassersteingradientflowsmaximum}, based on
\cite[Lemma~3.6]{rosenzweig2026wassersteingradientflowsmaximum}, gives
\begin{equation}
\label{eq:stability-transport-contribution}
  \int_{\R^d\times \R^d}|v_{\rho_t^{(1)}}(x)-v_{\rho_t^{(2)}}(y)|\,d\Pi_t(x,y)
  \le C_TD(t).
\end{equation}

For the jump term, set
\begin{equation}
\label{eq:stability-acceptance-coefficient}
  A_T:=\left(\sup_{0\le t\le T}A_{\rho_t^{(1)}}\right)\vee\left(\sup_{0\le t\le T}A_{\rho_t^{(2)}}\right)<\infty.
\end{equation}
Then
\begin{align}
|c_1(t,x,z)&-c_2(t,y,z)|
\nonumber \\
&\le
\|\lambda\|_{L^\infty}
|\zeta_{\rho_t^{(1)}}(x,x+z)
 -\zeta_{\rho_t^{(2)}}(y,y+z)|+
\|\nabla\lambda\|_{L^\infty} A_T|z|^{1\wedge q}|x-y|.\label{eq:stability-effective-rate-difference}
\end{align}
Write \(\delta_zf(x)=f(x)-f(x+z)\). Since the positive part is
\(1\)-Lipschitz,
\begin{align}
|\zeta_{\rho_t^{(1)}}(x,x+z)
 -\zeta_{\rho_t^{(2)}}(y,y+z)|
&\le
|\delta_zf_{\rho_t^{(1)}}(x)
 -\delta_zf_{\rho_t^{(1)}}(y)|+
|\delta_z(f_{\rho_t^{(1)}}-f_{\rho_t^{(2)}})(y)| \nonumber \\
&\le  \|D^2f_{\rho_t^{(1)}}\|_{L^\infty}|x-y||z|+
|\delta_z(f_{\rho_t^{(1)}}-f_{\rho_t^{(2)}})(y)|.\label{eq:acceptance-difference-decomposition}
\end{align}
For the second term above, use the fact that 
\begin{equation}
\label{eq:potential-difference-coupling-representation}
  (f_{\rho_t^{(1)}}-f_{\rho_t^{(2)}})(y)
  =
  \int_{\R^d\times \R^d}\left(K(y-a)-K(y-b)\right)\,d\Pi_t(a,b).
\end{equation}
Substituting \eqref{eq:potential-difference-coupling-representation},
writing \(c_\theta=b+\theta(a-b)\), and applying the fundamental theorem of calculus and
the kernel-gradient difference estimate give
\begin{align}
|K(y-a)-K(y+z-a)&-K(y-b)+K(y+z-b)|
\nonumber \\
&\le
C_{d,q}|a-b||z|
\int_0^1
\left(|y-c_\theta|^{q-2}
     +|y+z-c_\theta|^{q-2}\right)\,d\theta .\label{eq:mixed-kernel-increment-bound}
\end{align}
After integrating \eqref{eq:mixed-kernel-increment-bound} against
\(\rho_t^{(2)}(y)\,dy\),
\cite[Lemma~3.2]{rosenzweig2026wassersteingradientflowsmaximum} yields
\begin{equation}
\label{eq:mixed-potential-increment-bound}
  \int_{\R^d}
  |\delta_z(f_{\rho_t^{(1)}}-f_{\rho_t^{(2)}})(y)|
  \rho_t^{(2)}(y)\,dy
  \le C_T|z|D(t).
\end{equation}
Using both marginals of \(\Pi_t\) and combining
\eqref{eq:acceptance-difference-decomposition}
and
\eqref{eq:mixed-potential-increment-bound}, we conclude that
\begin{equation}
\label{eq:coupled-acceptance-difference-bound}
\begin{aligned}
&\int_{\R^d\times \R^d}\int_{\R^d}
|z|\,
|\zeta_{\rho_t^{(1)}}(x,x+z)
 -\zeta_{\rho_t^{(2)}}(y,y+z)|
\kappa(z)\,dz\,d\Pi_t(x,y)
\le C_T\mathcal{M}_2(\kappa)D(t).
\end{aligned}
\end{equation}
The term containing \(\|\nabla\lambda\|_{L^\infty}\) is bounded by \(  C_T\mathcal{M}_{2 \wedge (q+1)}(\kappa)D(t)\).
Combining then gives
\begin{equation}
\label{eq:stability-jump-contribution}
  \int_{\R^d\times \R^d}\int_{\R^d} |z|\,|c_1(t,x,z)-c_2(t,y,z)|
  \kappa(z)\,dz\,d\Pi_t(x,y)
  \le C_TD(t).
\end{equation}
It follows from \eqref{eq:stability-coupling-inequality},
\eqref{eq:stability-transport-contribution}, and
\eqref{eq:stability-jump-contribution} that \( D'(t)\le C_TD(t)\).
Since \(W_1(\rho_t^{(1)},\rho_t^{(2)})\le D(t)\), with \(D\) defined in
\eqref{eq:stability-coupling-distance}, Gronwall's inequality then proves
\eqref{eq:W1-stability-estimate}.
\end{proof}
The remaining endpoint arguments use the same jump estimates.
For \(p=\infty\) and \(d+q-2>0\), a finite admissible exponent gives the
construction, followed by a maximum estimate. For \(d=q=1\), the
cumulative-distribution formula replaces the singular-convolution
estimates.

\begin{proof}[Proof of Theorem~\ref{thm:continuum-wellposedness}]
Assume first that \(p<\infty\).
\Cref{prop:regularized-jump-transport,prop:limit-regularized-jump-transport} give a global solution on each
finite interval. A diagonal extraction gives a solution on
\([0,\infty)\), and Proposition~\ref{prop:continuum-W1-stability} gives
uniqueness. The strong \(L^p\)-continuity, velocity regularity, and
\(C^1_tW^{-1,p}_x\)-regularity were obtained in
\eqref{eq:limit-solution-regularity} and
\eqref{eq:limit-velocity-regularity}; the \(L^p\)- and moment bounds are
recorded in \eqref{eq:limit-apriori-bounds}.

Assume now that \(p=\infty\). If \(d+q-2>0\), choose \(p_c<p_0<\infty\).
Since \(L^1\cap L^\infty\subset L^{p_0}\), the case \(p=p_0\) proved above
gives a unique global weak solution satisfying
\eqref{eq:finite-p-weak-solution-class} with \(p=p_0\). We next prove
\eqref{eq:continuum-maximum-estimate}. To justify the calculation, freeze
the coefficients at this solution on \([0,T]\).
Lemma~\ref{lem:jump-inflow-outflow} bounds the frozen operators uniformly
on \(L^\infty\), and the transport evolution from
Appendix~\ref{app:positive-linear} satisfies
\(\|U_v(t,s)\|_{L^\infty\to L^\infty}\le e^{C_T(t-s)}\) for \(s\le t\).
Since \(\rho_0\in L^\infty\), the positive Picard iterates used there
are uniformly bounded in \(L^\infty((0,T)\times\mathbb R^d)\), by the
iterated-integral estimate. Their convergence in
\(C([0,T];L^1\cap L^{p_0})\) therefore gives
\(\rho\in L^\infty((0,T)\times\mathbb R^d)\).
Along a.e. characteristic \(\Phi(t)\) of the transport velocity,
the mild formula gives, for a.e. \(t\),
\begin{equation}
\label{eq:endpoint-characteristic-identity}
  \frac{d}{dt}\rho_t(\Phi(t))
  =
  -\rho_t\nabla\!\cdot v_{\rho_t}(\Phi(t))
  +\left(J_{\rho_t,\mathrm{in}}\rho_t\right)(\Phi(t))
  -\left(J_{\rho_t,\mathrm{out}}\rho_t\right)(\Phi(t)).
\end{equation}
Because \(\Delta K\le0\), the self-interaction part of the divergence term in
\eqref{eq:endpoint-characteristic-identity} is nonpositive. For the
regularized kernels in \eqref{eq:kernel-regularization},
\begin{equation}
\label{eq:endpoint-regularized-target-coefficient}
  \|(-\Delta K_\varepsilon)*\mu\|_{L^\infty}
  =\|\eta_\varepsilon*((-\Delta K)*\mu)\|_{L^\infty}
  \le\|(-\Delta K)*\mu\|_{L^\infty}.
\end{equation}
The target contribution therefore satisfies
\begin{equation}
\label{eq:endpoint-divergence-contribution}
  -\rho_t\nabla\!\cdot v_{\rho_t}
  \le\|(-\Delta K)*\mu\|_{L^\infty}\rho_t.
\end{equation}
The outflow term is also nonpositive, and
\begin{equation}
\label{eq:endpoint-inflow-bound}
  \|J_{\rho_t,\mathrm{in}}\rho_t\|_{L^\infty}
  \le\|\lambda\|_{L^\infty}A_{\rho_t}\mathcal{M}_{1\wedge q}(\kappa)\|\rho_t\|_{L^\infty}.
\end{equation}
Taking essential suprema in the integral form of
\eqref{eq:endpoint-characteristic-identity}, using
\eqref{eq:endpoint-divergence-contribution} and
\eqref{eq:endpoint-inflow-bound}, and applying Gronwall's inequality
gives \eqref{eq:continuum-maximum-estimate}. The estimate holds for every
\(t\) by strong \(L^1\)-continuity and lower semicontinuity.
For the regularized solutions, the same computation uses
\eqref{eq:endpoint-regularized-target-coefficient} and is uniform in
\(\varepsilon\) by \eqref{eq:regularized-uniform-bounds}. The finite-time
power-\(r\)-moment bound
in \eqref{eq:limit-apriori-bounds} and the acceptance coefficient
\eqref{eq:acceptance-coefficient} make its exponent finite. Strong continuity in \(L^1\cap L^{p_0}\), interpolation below \(p_0\), and
interpolation with \eqref{eq:continuum-maximum-estimate} above \(p_0\) give
\(\rho\in C([0,T];L^s)\) for every finite \(s\). Uniform boundedness in
\(L^\infty\), together with strong \(L^1\)-continuity, gives weak-star
continuity in \(L^\infty\). For every weak solution satisfying
\eqref{eq:infinite-p-weak-solution-class}, the time-continuity argument in
the proof of \Cref{prop:limit-regularized-jump-transport}
also gives \(J(\rho)\in C([0,T];L^{p_0})\) and
\(\rho\in C^1([0,T];W^{-1,p_0})\). Thus it satisfies
\eqref{eq:finite-p-weak-solution-class} with \(p=p_0\), and uniqueness
follows from Proposition~\ref{prop:continuum-W1-stability}.

It remains to treat \(d+q-2=0\), that is,
\((d,q)=(1,1)\). Here
\begin{equation}
\label{eq:one-dimensional-velocity-identities}
  K''=-2\delta_0,\qquad
  v_\rho(x)=2(F_\rho(x)-F_\mu(x)),\qquad
  \partial_xv_\rho=2(\rho-\mu),
\end{equation}
where \(F_\nu(x):=\int_{-\infty}^x\nu(y)\,dy\)
denotes the cumulative distribution function of \(\nu\). The pure-transport problem for \((d,q)=(1,1)\) has been studied
in \cite{Bon+15,Duo+26}; the point here is to control the additional jump
perturbation. The identity for \(\partial_xv_\rho\) in
\eqref{eq:one-dimensional-velocity-identities} shows that bounded densities
generate globally Lipschitz velocities, while the velocity formula and its
mollified analogue give
\begin{equation}
\label{eq:one-dimensional-jump-increment}
  0\le\zeta_\nu^\varepsilon(x,x+z)\le2|z|,
  \qquad 0\le\varepsilon\le1.
\end{equation}
We divide the remainder of the proof into three steps.

\emph{Step 1: Regularized construction.}
Fix \(1<s<\infty\). For \(R,M>0\), use the complete metric space
\begin{equation}
\label{eq:one-dimensional-fixed-point-space}
\begin{aligned}
\mathsf Y_{\tau,R,M}:=\left\{
  \nu\in C([0,\tau];\mathcal P_1(\mathbb R)):\:
  \nu_0=\rho_0,\:
  \sup_{t\le\tau}\widetilde{\mathcal M}_r(\nu_t)\le R,\:\operatorname*{ess\,sup}_{t\le\tau}
  \|\nu_t\|_{L^\infty}\le M
\right\}.
\end{aligned}
\end{equation}
The space in \eqref{eq:one-dimensional-fixed-point-space} is equipped with
the uniform \(W_1\) metric. Completeness follows from weak-star
compactness in \(L^\infty\), lower semicontinuity of
\(\widetilde{\mathcal M}_r\), and uniqueness of the narrow limit.

For a frozen \(\nu\in\mathsf Y_{\tau,R,M}\),
\begin{equation}
\label{eq:one-dimensional-frozen-velocity-bound}
  \partial_xv_\nu^\varepsilon
  =2\eta_\varepsilon*(\nu-\mu),\qquad
  \|\partial_xv_\nu^\varepsilon\|_{L^\infty}
  \le2(M+\|\mu\|_{L^\infty}).
\end{equation}
The derivative estimate \eqref{eq:one-dimensional-frozen-velocity-bound} and
Lemma~\ref{lem:jump-inflow-outflow}, together with
\eqref{eq:one-dimensional-jump-increment}, bound the inflow and outflow
operators on
\(L^1\), \(L^s\), and \(L^\infty\). Thus
Lemma~\ref{lem:linear-jump-transport} defines the frozen map in
\(L^1\cap L^s\). The bounded Picard-iteration argument above also gives
\(u\in L^\infty((0,\tau)\times\mathbb R)\). Along a.e.
characteristic, for a.e. \(t\),
\begin{equation}
\label{eq:one-dimensional-frozen-characteristic-inequality}
  \frac{d}{dt}u_t(\Phi(t))
  \le
  2\|\mu\|_{L^\infty}u_t(\Phi(t))
  +\left(J_{\nu,\mathrm{in}}^\varepsilon u_t\right)(\Phi(t)).
\end{equation}
Taking the spatial essential supremum in the integral form of
\eqref{eq:one-dimensional-frozen-characteristic-inequality} and using the
\(L^\infty\) inflow bound in \eqref{eq:frozen-inflow-outflow-bound} gives
\begin{equation}
\label{eq:one-dimensional-frozen-maximum-bound}
  \|u_t\|_{L^\infty}
  \le
  e^{(2\|\mu\|_{L^\infty}
       +2\|\lambda\|_{L^\infty}\mathcal{M}_1(\kappa))t}
  \|\rho_0\|_{L^\infty}.
\end{equation}
The frozen version of the calculation leading to
\eqref{eq:regularized-moment-differential-inequality} gives
\begin{equation}
\label{eq:one-dimensional-frozen-moment-bound}
  \widetilde{\mathcal M}_r(u_t)
  \le e^{C_{r,\lambda,\kappa}t}
       \widetilde{\mathcal M}_r(\rho_0),
  \qquad 0\le t\le\tau.
\end{equation}
Together, \eqref{eq:one-dimensional-frozen-maximum-bound} and
\eqref{eq:one-dimensional-frozen-moment-bound} show that the frozen map sends
\(\mathsf Y_{\tau,R,M}\) into itself for suitable \(R,M\) and sufficiently
small \(\tau\). Smoothness of \(K_\varepsilon\) and the common-proposal
generator \eqref{eq:common-proposal-generator} give, by the argument leading
to \eqref{eq:frozen-map-contraction-bound},
\begin{equation}
\label{eq:one-dimensional-frozen-map-contraction-bound}
\begin{aligned}
&\sup_{0\le t\le\tau}
W_1\left(\Phi_\varepsilon(\nu)_t,\Phi_\varepsilon(\omega)_t\right)
\le
\tau C_{\varepsilon,R,M}e^{C_{\varepsilon,R,M}\tau}
\sup_{0\le t\le\tau}W_1(\nu_t,\omega_t).
\end{aligned}
\end{equation}
After decreasing \(\tau\),
\eqref{eq:one-dimensional-frozen-map-contraction-bound} makes the frozen map
a contraction, and hence gives a unique local regularized solution.

At a fixed point \(\nu=u=\rho^\varepsilon\), the characteristic identity is
\begin{align}
  \frac{d}{dt}\rho_t^\varepsilon(\Phi(t))
  ={}&
  \left(J_{\rho_t^\varepsilon,\mathrm{in}}^\varepsilon
     \rho_t^\varepsilon\right)(\Phi(t))
  -\left(J_{\rho_t^\varepsilon,\mathrm{out}}^\varepsilon
     \rho_t^\varepsilon\right)(\Phi(t))
  \nonumber \\
  &\quad\quad\quad+2(\eta_\varepsilon*\mu)(\Phi(t))\rho_t^\varepsilon(\Phi(t))
  -2(\eta_\varepsilon*\rho_t^\varepsilon)(\Phi(t))
     \rho_t^\varepsilon(\Phi(t)).\label{eq:one-dimensional-fixed-point-characteristic-identity}
\end{align}
Discarding the outflow and self-interaction terms in
\eqref{eq:one-dimensional-fixed-point-characteristic-identity} gives,
uniformly in \(\varepsilon\),
\begin{equation}
\label{eq:one-dimensional-uniform-maximum-bound}
  \|\rho_t^\varepsilon\|_{L^\infty}
  \le
  e^{(2\|\mu\|_{L^\infty}
       +2\|\lambda\|_{L^\infty}\mathcal{M}_1(\kappa))t}
  \|\rho_0\|_{L^\infty}.
\end{equation}
Using \(\|v_{\rho^\varepsilon}^\varepsilon\|_{L^\infty}\le2\),
\eqref{eq:moment-weight-increment}, and
\eqref{eq:one-dimensional-jump-increment}, the moment calculation
gives
\begin{equation}
\label{eq:one-dimensional-moment-differential-inequality}
  \frac{d}{dt}\widetilde{\mathcal M}_r(\rho_t^\varepsilon)
  \le C_{r,\lambda,\kappa}
  \widetilde{\mathcal M}_r(\rho_t^\varepsilon).
\end{equation}
Here \(C_{r,\lambda,\kappa}<\infty\) by
\(\mathcal{M}_{r+1}(\kappa)<\infty\). Gronwall's inequality in
\eqref{eq:one-dimensional-moment-differential-inequality} yields
\begin{equation}
\label{eq:one-dimensional-uniform-moment-bound}
  \sup_{0<\varepsilon\le1}\sup_{0\le t\le T}
  \widetilde{\mathcal M}_r(\rho_t^\varepsilon)\le C_T.
\end{equation}

The bounds \eqref{eq:one-dimensional-uniform-maximum-bound} and
\eqref{eq:one-dimensional-uniform-moment-bound} give global continuation.
Testing by truncated \(1\)-Lipschitz functions also yields
\begin{equation}
\label{eq:one-dimensional-W1-time-modulus}
  W_1(\rho_t^\varepsilon,\rho_s^\varepsilon)
  \le C_T|t-s|,
\end{equation}
uniformly in \(\varepsilon\); here \(\mathcal{M}_2(\kappa)<\infty\) controls the jump
displacements.

\emph{Step 2: Compactness and passage to the limit.}
If \(r>1\), the uniform power-moment bound
\eqref{eq:one-dimensional-uniform-moment-bound} and
\eqref{eq:one-dimensional-W1-time-modulus} yield compactness in
\(C([0,T];\mathcal P_1(\mathbb R))\). If \(r=1\), the thinning argument in
\Cref{prop:limit-regularized-jump-transport}, using
\eqref{eq:dominating-Levy-measure-r1} with \(q=1\), gives the required
uniform integrability of first moments and hence compactness in
\(C([0,T];\mathcal P_1(\mathbb R))\). If
\(\rho^{\varepsilon_n}\to\rho\) in this space, lower semicontinuity and
\eqref{eq:power-bracket-moment-comparison} give
\begin{equation}
\label{eq:one-dimensional-limit-moment-bounds}
  \sup_{0\le t\le T}\mathcal M_r(\rho_t)\le C_T.
\end{equation}
Moreover,
\begin{equation}
\label{eq:one-dimensional-regularized-CDF-velocity}
  v_{\rho^{\varepsilon_n}}^{\varepsilon_n}
  =
  2\left(F_{\eta_{\varepsilon_n}*\rho^{\varepsilon_n}}
         -F_{\eta_{\varepsilon_n}*\mu}\right).
\end{equation}
The bound \eqref{eq:one-dimensional-uniform-maximum-bound} makes all the CDFs
uniformly Lipschitz. Consequently, convergence in
\(C([0,T];\mathcal P_1)\), together with these Lipschitz bounds, gives
convergence of the CDFs uniformly in \((t,x)\), while
\begin{equation}
\label{eq:CDF-mollification-error}
  \|F_{\eta_{\varepsilon_n}*\nu}
       -F_\nu\|_{L^\infty}
  \le
  \varepsilon_n\|\nu\|_{L^\infty}
  \int|w|\eta(w)\,dw .
\end{equation}
It follows from \eqref{eq:one-dimensional-regularized-CDF-velocity} and
\eqref{eq:CDF-mollification-error} that
\begin{equation}
\label{eq:one-dimensional-uniform-velocity-convergence}
  \left\|v_{\rho^{\varepsilon_n}}^{\varepsilon_n}-v_\rho\right\|_{L^\infty([0,T]\times\mathbb R)}
  \to0.
\end{equation}
This convergence and
the compact/large-\(z\) argument in
\Cref{prop:limit-regularized-jump-transport}, using
\eqref{eq:one-dimensional-jump-increment}, allow passage to the limit in the jump
term; its tail is controlled exactly as in \eqref{eq:jump-tail-bound}.
Freezing the limiting coefficients gives strong \(L^s\)-continuity for every
finite \(s\), and \eqref{eq:one-dimensional-uniform-maximum-bound} gives
weak-star \(L^\infty\)-continuity. Passing to the limit in
\eqref{eq:one-dimensional-uniform-maximum-bound} gives
\eqref{eq:continuum-maximum-estimate}, because
\((-\Delta K)*\mu=2\mu\), \(A_{\rho_s}=2\), and
\(\mathcal M_{1\wedge q}(\kappa)=\mathcal M_1(\kappa)\).
The moment bounds in \eqref{eq:one-dimensional-limit-moment-bounds} and the
maximum estimate \eqref{eq:continuum-maximum-estimate} prove existence and
\eqref{eq:infinite-p-weak-solution-class}.

\emph{Step 3: Uniqueness.}
Apply the common-proposal generator \eqref{eq:common-proposal-generator} to
two weak solutions satisfying
\eqref{eq:infinite-p-weak-solution-class}, with the rates  $c_i(t,x,z)$, starting from an optimal coupling of
their initial data. As in the proof of
Proposition~\ref{prop:continuum-W1-stability}, the resulting curve \(\Pi_t\)
has marginals \(\rho_t^{(1)}\) and \(\rho_t^{(2)}\). Define \(D_1\) by
\begin{equation}
\label{eq:one-dimensional-coupling-distance}
  D_1(t):=\int_{\mathbb R\times\mathbb R}|x-y|\,d\Pi_t(x,y).
\end{equation}
Testing the coupled equation by smooth approximations of \(|x-y|\) gives
\begin{align}
  D_1'(t)\le{}&
  \int_{\R\times \R}|v_{\rho_t^{(1)}}(x)-v_{\rho_t^{(2)}}(y)|\,d\Pi_t(x,y)
  \nonumber \\
  &+\int_{\R\times \R}\int_{\R} |z|\,|c_1(t,x,z)-c_2(t,y,z)|
  \kappa(z)\,dz\,d\Pi_t(x,y).\label{eq:one-dimensional-coupling-inequality}
\end{align}
Since \(v_\rho=2(F_\rho-F_\mu)\), the finite-time density bounds give
\(\operatorname{Lip}(v_{\rho_t^{(1)}})\le C_T\). Moreover, the coupling
representation and
\begin{equation}
  \int_{\R}
  |\operatorname{sgn}(y-a)-\operatorname{sgn}(y-b)|
  \rho_t^{(2)}(y)\,dy
  \le2\|\rho_t^{(2)}\|_{L^\infty}|a-b|
\end{equation}
give
\(\int_{\R}|v_{\rho_t^{(1)}}-v_{\rho_t^{(2)}}|
\,d\rho_t^{(2)}
\le2\|\rho_t^{(2)}\|_{L^\infty}D_1(t)\).
Splitting the velocity difference at \(y\), we therefore obtain
\begin{equation}
\label{eq:one-dimensional-transport-coupling-bound}
  \int_{\R\times\R}
  |v_{\rho_t^{(1)}}(x)-v_{\rho_t^{(2)}}(y)|
  \,d\Pi_t(x,y)
  \le C_TD_1(t).
\end{equation}
For the jump term, writing \(\delta_zf(x)=f(x)-f(x+z)\), we have
\begin{align}
  |c_1(t,x,z)-c_2(t,y,z)|
  &\le \|\lambda\|_{W^{1,\infty}}
  \Bigl(
    |\delta_zf_{\rho_t^{(1)}}(x)-\delta_zf_{\rho_t^{(1)}}(y)|
    +|\delta_z(f_{\rho_t^{(1)}}-f_{\rho_t^{(2)}})(y)|
    +2|x-y||z|
  \Bigr)\nonumber \\
  &\le C_T\bigl(|x-y||z|+D_1(t)\bigr).
\end{align}
Multiplying by \(|z|\) and integrating yields
\begin{equation}
\label{eq:one-dimensional-jump-contribution}
\begin{aligned}
\int_{\R\times\R}\int_{\R}
  |z|\,|c_1(t,x,z)-c_2(t,y,z)|
  \le C_TD_1(t).
\end{aligned}
\end{equation}
Combining \eqref{eq:one-dimensional-coupling-inequality}--\eqref{eq:one-dimensional-jump-contribution}, and then
applying Gronwall's inequality yields
\begin{equation}
\label{eq:one-dimensional-W1-stability}
\begin{aligned}
  W_1(\rho_t^{(1)},\rho_t^{(2)})
  &\le D_1(t)
  \le e^{C_Tt}D_1(0)=e^{C_Tt}W_1(\rho_0^{(1)},\rho_0^{(2)}).
\end{aligned}
\end{equation}
Equation \eqref{eq:one-dimensional-W1-stability} proves uniqueness. The
maximum estimate \eqref{eq:continuum-maximum-estimate} was obtained in
Step~2.
\end{proof}

\section{Stationary solutions and long-time behavior}
\label{subsec:stationary-long-time}
\subsection{Stationarity and energy dissipation}
\label{subsec:stationarity-energy-dissipation}
For statements about the evolving source density in this section, assume the hypotheses of
Theorem~\ref{thm:continuum-wellposedness} under the stronger assumption \(r\ge2\), and let
\(\rho_t\) be its global weak solution. In particular,
\(\rho_0,\mu\in\mathcal P_2(\mathbb R^d)\) and
\(\kappa\in\mathcal P_3(\mathbb R^d)\).
The measure definitions and rigidity statements below have their own explicitly
stated domains.

The long-time problem is to determine whether the solution converges to
the target \(\mu\). Although \(\mu\) is the unique minimizer of
\(\mathcal E_q(\,\cdot-\mu)\), decrease of the energy alone does not
identify the limiting measure. One must also understand which measures
can have vanishing dissipation. We first distinguish stationarity of
the equation from Lagrangian criticality of the energy. For the latter
terminology, we follow Boufad\`ene--Vialard
\cite[Definition~3.2]{boufadene2023global}, using the a.e.
canonical-force formulation of
\cite[Definition~4.3 and Remark~4.4]{rosenzweig2026wassersteingradientflowsmaximum}.

\begin{definition}[Stationary solutions and Lagrangian critical points]
\label{def:stationary-critical}
Let $\nu,\mu\in\mathcal P(\mathbb R^d)$, let $\lambda$ be a bounded,
nonnegative Borel function, and let $\kappa$ be a probability density
in $\mathcal P_{1\wedge q}(\mathbb R^d)$. When $1<q<2$, assume
$\int |y|^{q-1}\,d|\nu-\mu|(y)<\infty$.
Use the normalized increment $\Delta_\nu$ in
\eqref{eq:extended-potential-increment}. With the diagonal convention
$\nabla K(0)=0$, define the canonical force by
\begin{equation}\label{eq:stationary-canonical-force}
 \nabla f_\nu(x):=\int_{\mathbb R^d}\nabla K(x-y)\,d(\nu-\mu)(y)
\end{equation}
where the integral is absolutely convergent, and set it to zero elsewhere.
Require
\begin{equation}\label{eq:stationary-force-integrability}
 \int_{\mathbb R^d}|\nabla K(x-y)|\,d|\nu-\mu|(y)<\infty
 \quad\text{for $\nu$-a.e. $x$},\qquad
 \nabla f_\nu\in L^1_{\mathrm{loc}}(\nu).
\end{equation}
\begin{enumerate}[label=\textup{(\roman*)}]
\item
We call $\nu$ a \emph{stationary solution} of \eqref{eq:pde} if
\eqref{eq:stationary-jump-transport-weak} holds for every
$\varphi\in C_c^\infty(\mathbb R^d)$, with an absolutely convergent jump integral.
\item
We call $\nu$ a \emph{Lagrangian critical point} of
$\mathcal E_q(\,\cdot-\mu)$ if
$\mathcal E_q(\nu-\mu)<\infty$ and
$\nabla f_\nu=0$ $\nu$-a.e.
\end{enumerate}
A stationary solution or Lagrangian critical point is
\emph{spurious} if it differs from $\mu$.
\end{definition}

The stationary identity in the definition is
\begin{align}
  0={}&-\int_{\mathbb R^d}
       \nabla\varphi(x)\cdot\nabla f_\nu(x)\,d\nu(x)\nonumber \\
     &+\int_{\mathbb R^d}\int_{\mathbb R^d}
       \bigl(\varphi(x+z)-\varphi(x)\bigr)
       \lambda(x)[\Delta_\nu(x,z)]_+\kappa(z)\,dz\,d\nu(x).\label{eq:stationary-jump-transport-weak}
\end{align}
The normalized increments are absolutely convergent and continuous under
the stated moment condition. The argument of
Lemma~\ref{lem:jump-increment}, applied directly to these increments, gives
\begin{equation}\label{eq:stationary-increment-bound}
 |\Delta_\nu(x,z)|\le
 \begin{cases}
  2|z|^q,&0<q\le1,\\
  C_q|z|\displaystyle\int |y|^{q-1}\,d|\nu-\mu|(y),&1<q<2.
 \end{cases}
\end{equation}
Indeed, for $q\le1$ this follows from the $q$-H\"older continuity of $K$.
For $q>1$, subtract $\nabla K(x)$ in
\eqref{eq:stationary-canonical-force}, use the $(q-1)$-H\"older continuity
of $\nabla K$, and integrate along the segment from $x$ to $x+z$.
Thus, boundedness of $\lambda$ and
$\kappa\in\mathcal P_{1\wedge q}$ ensure absolute convergence of every weak
jump test. For $q\ge1$, the canonical force is also absolutely convergent
and bounded under these assumptions. For $0<q<1$, the local force
conditions in \eqref{eq:stationary-force-integrability} must be checked
separately; finite energy alone does not imply them.

Lagrangian criticality implies pure-transport stationarity because the
transport pairing in \eqref{eq:stationary-jump-transport-weak} vanishes.
Adding jumps changes the stationary equation while leaving the energy
and its Lagrangian critical points unchanged. Boufad\`ene--Vialard use
vanishing of the force on the support. This agrees with the
a.e. condition above whenever the canonical force is
continuous. The diagonal convention specifies the force at atoms.

The rigidity theorem \cite[Theorem~1.5]{rosenzweig2026wassersteingradientflowsmaximum} excludes non-target Lagrangian critical points for $\mu,\nu\in\mathcal P_q(\mathbb R^d)$, with $\nu$ absolutely continuous and $d+q\ge2$, under the canonical-force and integrability convention above. A narrow limit can lose absolute continuity because the finite-time density bounds need not be uniform. The jump condition below identifies such limits. For comparison, the ordinary-flow asymptotic-criticality theorem is \cite[Theorem~5.23]{rosenzweig2026wassersteingradientflowsmaximum}; its energy-distance compactness refinement and conditional target-convergence result are \cite[Theorem~1.7 and Corollary~5.27]{rosenzweig2026wassersteingradientflowsmaximum}.

Recall the dissipations $\mathcal D_{\mathrm{tr}}$, $\mathcal D_{\mathrm j}$, and $\mathcal D$ from \eqref{eq:transport-dissipation}--\eqref{eq:total-dissipation}, using the canonical force and normalized increments on the domain of Definition~\ref{def:stationary-critical}.
For finite-energy measures in this domain, Lagrangian criticality is
equivalent to $\mathcal D_{\mathrm{tr}}(\nu)=0$. If both dissipations vanish, the
stationary identity follows directly on that domain.
For the converse, let $\nu\in\mathcal P_r(\mathbb R^d)\cap L^p(\mathbb R^d)$
and assume the target, kernel, proposal, and spatial-rate hypotheses of
Theorem~\ref{thm:continuum-wellposedness}, with $r\ge2$.
In this class,
\begin{equation}
\label{eq:stationarity-zero-total-dissipation}
  \nu\text{ is stationary for \eqref{eq:pde}}
  \quad\iff\quad
  \mathcal D_{\mathrm{tr}}(\nu)=\mathcal D_{\mathrm j}(\nu)=0.
\end{equation}
Indeed, the convolution and inflow--outflow bounds from
Section~\ref{sec:continuum-wellposedness} put the constant curve
$\rho_t\equiv\nu$ in the regularity class of
Definition~\ref{def:continuum-weak-solution}. Its weak formulation is
equivalent to \eqref{eq:stationary-jump-transport-weak}. If $\nu$ is
stationary, Proposition~\ref{prop:energy-dissipation} applied to this
constant curve gives $t\mathcal D(\nu)=0$ for every $t>0$.
Nonnegativity gives vanishing of both terms, proving the remaining implication.

The jump term yields a rigidity argument that also applies to the
narrow limits considered below; see Proposition~\ref{prop:no-spur-stat-solns}.
To describe the long-time limits, we denote the $\omega$-limit set of an initial source $\rho_0$ by
\begin{equation}
  \omega(\rho_0):=
  \bigl\{\nu\in\mathcal P(\mathbb R^d):
  \rho_{t_n}\rightharpoonup\nu
  \text{ for some }t_n\to\infty\bigr\}.
\end{equation}
We prove that \(\omega(\rho_0)=\{\mu\}\) by combining tightness of
energy sublevels, lower semicontinuity of jump dissipation, and
rigidity of zero jump dissipation (Proposition~\ref{prop:no-spur-stat-solns}).
Lemma~\ref{lem:W1-dissipation} transfers convergence from
times of small dissipation to the full trajectory. The second-moment
argument then gives MMD convergence.

\begin{proposition}[Energy-dissipation identity]
\label{prop:energy-dissipation}
The map
\(
  t\mapsto\mathcal E_q(\rho_t-\mu)
\)
is locally absolutely continuous.  For every \(0\le s\le t<\infty\),
\begin{equation}
\label{eq:energy-dissipation-identity}
  \mathcal E_q(\rho_t-\mu)
  +\int_s^t\mathcal D(\rho_\tau)\,d\tau
  =\mathcal E_q(\rho_s-\mu).
\end{equation}
In particular,
\begin{equation}
\label{eq:energy-dissipation-differential}
  \frac{d}{dt}\mathcal E_q(\rho_t-\mu)
  =-\mathcal D(\rho_t)
  \quad\text{for a.e. }t>0.
\end{equation}
\end{proposition}

\begin{proof}
The corresponding pure-transport identity is
\cite[Proposition~3.8]
{rosenzweig2026wassersteingradientflowsmaximum}.  Here we use the
regularized approximations from Section~\ref{sec:continuum-wellposedness} and compute the additional jump
pairing. Set
\begin{equation}
\label{eq:regularized-energy-definition}
  \mathcal E_q^\varepsilon(\rho_t^\varepsilon-\mu)
  :=\frac12
  \int_{\mathbb R^d}\int_{\mathbb R^d}
  K_\varepsilon(x-y)
  \,d(\rho_t^\varepsilon-\mu)(x)
  \,d(\rho_t^\varepsilon-\mu)(y).
\end{equation}
Using the regularized mild equation, one first tests against compactly
supported spatial truncations of the regularized first variation and then applies
a time regularization.  The uniform moment bounds allow both
regularizations to be removed, so that
\eqref{eq:regularized-energy-definition} may be differentiated, and hence the regularized analogue of
\eqref{eq:energy-dissipation-identity} follows.

When \(d+q-2>0\), we
pass to the limit along the approximation obtained in
\Cref{prop:limit-regularized-jump-transport}; for
\(d=q=1\), we use the approximation in Step~2 of the proof of
Theorem~\ref{thm:continuum-wellposedness},
with \eqref{eq:one-dimensional-uniform-moment-bound}--\eqref{eq:one-dimensional-W1-time-modulus}.  Thus, in either
case, the approximations converge uniformly in \(W_1\) on \([0,T]\). Combined with the MMD-transport comparison
\cite[Proposition~A.3(1)]
{rosenzweig2026wassersteingradientflowsmaximum}, this gives, uniformly for
\(t\in[0,T]\),
\begin{equation}
\label{eq:regularized-MMD-convergence}
  \operatorname{MMD}_q(\rho_t^{\varepsilon_n},\rho_t)
  \le C_{d,q}
  W_1(\rho_t^{\varepsilon_n},\rho_t)^{q/2}
  \to0.
\end{equation}
The kernel regularization satisfies
\begin{equation}
\label{eq:regularized-energy-kernel-error}
  \left|
  \mathcal E_q^\varepsilon(\nu-\mu)
  -\mathcal E_q(\nu-\mu)
  \right|
  \le
  \begin{cases}
    C\varepsilon^q,&0<q\le1,\\
    C\varepsilon
    \bigl(
      1+\mathcal M_{q-1}(\nu)+\mathcal M_{q-1}(\mu)
    \bigr)+C\varepsilon^q,&1<q<2.
  \end{cases}
\end{equation}
Equations
\eqref{eq:regularized-MMD-convergence}-
\eqref{eq:regularized-energy-kernel-error}, together with
\eqref{eq:regularized-uniform-bounds} when \(d+q-2>0\) and
\eqref{eq:one-dimensional-uniform-moment-bound} when
\((d,q)=(1,1)\), identify the energies at both ends of the integrated
regularized identity.

When \(d+q-2>0\), the convergence \(v_{\rho_t^{\varepsilon_n}}^{\varepsilon_n}
  \to v_{\rho_t}\), the global velocity bound
in \eqref{eq:regularized-uniform-bounds}, and uniform tightness identify
the transport dissipation after splitting \(x\) into a ball and its
complement. When \((d,q)=(1,1)\), one uses instead
\eqref{eq:one-dimensional-uniform-velocity-convergence}, the bound
\(\|v_{\rho^{\varepsilon_n}}^{\varepsilon_n}\|_{L^\infty}\le2\), and
\eqref{eq:one-dimensional-uniform-moment-bound}.  For the jump
dissipation, the local uniform convergence of the potential increments
from \eqref{eq:potential-increment-convergence-bound} is combined with
\eqref{eq:jump-increment-bound}, the boundedness and
continuity of \(\lambda\), and a compact-tail decomposition.  Uniformly
in \(n\) and \(t\in[0,T]\), \(  \lambda(x)
  \zeta_{\rho_t^{\varepsilon_n}}^{\varepsilon_n}(x,x+z)^2
  \le C_T|z|^{2(1\wedge q)},\)
where the finiteness of \(C_T\) follows from
\eqref{eq:regularized-uniform-bounds} and
\eqref{eq:acceptance-coefficient}; in the case \((d,q)=(1,1)\), it follows
instead from \eqref{eq:one-dimensional-jump-increment}. Uniform tightness handles the \(x\)-tail, while
local uniform increment convergence handles the compact core.  Thus both
dissipation terms converge with an integrable bound uniform in time.
Dominated convergence passes directly in the integrated regularized
identity and gives
\eqref{eq:energy-dissipation-identity}.  Absolute continuity and
\eqref{eq:energy-dissipation-differential} follow.
\end{proof}

\begin{lemma}
\label{lem:W1-dissipation}
For every \(0\le s\le t<\infty\),
\begin{align}
  W_1(\rho_t,\rho_s)
  &\le \left(1+(\|\lambda\|_{L^\infty}\mathcal M_2(\kappa))^{1/2}\right)
  \int_s^t\mathcal D(\rho_\tau)^{1/2}\,d\tau,
  \label{eq:W1-dissipation-speed}\\
  W_1(\rho_t,\rho_s)
  &\le \left(1+(\|\lambda\|_{L^\infty}\mathcal M_2(\kappa))^{1/2}\right)|t-s|^{1/2}
  \left(
    \int_s^t\mathcal D(\rho_\tau)\,d\tau
  \right)^{1/2}.
  \label{eq:W1-dissipation-time-modulus}
\end{align}
\end{lemma}
\begin{proof}
Let \(\varphi\in C_c^\infty(\mathbb R^d)\) be \(1\)-Lipschitz.  The weak formulation
\eqref{eq:weak-jump-transport-formulation} and Cauchy-Schwarz give
\begin{align}
  \left|
  \frac{d}{dt}\int_{\mathbb R^d}\varphi\,d\rho_t
  \right|
  &\le
  \left(1+(\|\lambda\|_{L^\infty}\mathcal M_2(\kappa))^{1/2}\right)\mathcal D(\rho_t)^{1/2}.
  \label{eq:test-function-dissipation-speed}
\end{align}
Integrating over $[s,t]$ gives 
\begin{equation}
    \sup_{\phi \in C_c^\infty(\R^d) \cap \mathrm{Lip}_1(\R^d)}\int \phi \, d(\rho_t - \rho_s) \leq \left(1+(\|\lambda\|_{L^\infty}\mathcal M_2(\kappa))^{1/2}\right) \int_s^t \mathcal D(\rho_\tau)^{1/2} \, d\tau.
    \label{eqn:kr-C_inf-est}
\end{equation}
By an appeal to density, see that (\ref{eqn:kr-C_inf-est}) holds when the supremum is taken over all of $\mathrm{Lip}_1(\R^d)$. Kantorovich-Rubinstein duality then establishes (\ref{eq:W1-dissipation-speed}), which together with Cauchy-Schwarz implies (\ref{eq:W1-dissipation-time-modulus}).
\end{proof}

\subsection{Energy preliminaries and rigidity of jump dissipation}
\label{subsec:energy-preliminaries-rigidity}

Narrow limits need not preserve the moments available at finite times. We therefore use a representation of the energy that remains meaningful without
these moment assumptions. With the conventions
\begin{equation}\label{eq:fourier-convention}
 \widehat\nu(\xi)=\int e^{-2\pi i\xi\cdot x}\,d\nu(x),
 \qquad \widehat u(\xi)=\int e^{-2\pi i\xi\cdot x}u(x)\,dx,
\end{equation}
set, for finite signed measures of total mass zero,
\begin{equation}\label{eq:extended-fourier-energy}
 \mathcal E_q(\sigma)=c_{d,q}\int_{\mathbb R^d}
       \frac{|\widehat\sigma(\xi)|^2}{|\xi|^{d+q}}\,d\xi\in[0,\infty],
 \qquad \sigma(\mathbb R^d)=0.
\end{equation}
Set \(s_q:=(d+q)/2\), so \(d/2<s_q<d/2+1\). We use the homogeneous Sobolev norm
\begin{equation}
\label{eq:particle-homogeneous-Sobolev-norm}
  \|u\|_{\dot H^{s_q}}^2:=\int_{\mathbb R^d}
           |\xi|^{2s_{q}}|\widehat u(\xi)|^2\,d\xi,
  \qquad \widehat{|\nabla|^{s_q} u}(\xi):=|\xi|^{s_q}\widehat u(\xi),
\end{equation}
initially on Schwartz functions and then by homogeneous completion modulo constants. Since \(d/2<s_q<d/2+1\), choosing the continuous representative with \(u(0)=0\) identifies this space, up to norm normalization, with the RKHS \(\mathcal H_{k_q}\). Constants do not affect its pairing with a zero-mass measure; see \cite[Chapter~10]{wendland2005scattered} and
\cite[Remark~2.2]{rosenzweig2026wassersteingradientflowsmaximum}. For a finite signed measure $\sigma$ of zero total mass, \eqref{eq:extended-fourier-energy} reads $\mathcal E_q(\sigma)=c_{d,q}\|\sigma\|_{\dot H^{-s_q}}^2$.
The constant $c_{d,q}>0$ agrees with \eqref{eq:energy}. For differences of measures in \(\mathcal{P}_{q/2}(\R^d)\), this is the centered squared MMD; see
\cite{fornasier2016consistency} and
\cite[Lemma~2.1]{rosenzweig2026wassersteingradientflowsmaximum}.
This expression is lower semicontinuous under narrow convergence of probability measures by Fatou's lemma, and
\begin{equation}
\label{eq:extended-energy-rigidity}
  \mathcal E_q(\sigma)=0\quad\iff\quad\sigma=0,
\end{equation}
by continuity and uniqueness of Fourier transforms. 

The next measure-level lemmas use only $\mu\in\mathcal P_{q/2}$ and the displayed energy and increment hypotheses, not the assumption $r\ge2$ on the evolving solution. They first extract tightness and moments below $q/2$ from bounded energy.

\begin{lemma}
\label{lem:energy-sublevels}
Let $\mu\in\mathcal P_{q/2}(\mathbb R^d)$ and $(\nu_i)_{i\in I}\subset\mathcal P(\mathbb R^d)$ satisfy
\begin{equation}
\label{eq:energy-sublevel-assumption}
  \sup_{i\in I}\mathcal E_q(\nu_i-\mu)<\infty.
\end{equation}
Then it holds that
\begin{equation}
\label{eq:energy-sublevel-moment-bound}
  \sup_{i\in I}\mathcal M_\alpha(\nu_i)<\infty,\qquad \forall\:\alpha\in(0,q/2),
\end{equation}
and in particular $(\nu_i)_{i \in I}$ is tight.
\end{lemma}
\begin{proof}
First suppose that $\nu_i\in\mathcal P_{q/2}(\mathbb R^d)$. By
\cite[Proposition~A.2]{rosenzweig2026wassersteingradientflowsmaximum},
\begin{equation}\label{eq:RSW-subcritical-moment-control}
 \mathcal M_\alpha(\nu_i)
 \le\mathcal M_\alpha(\mu)+C_{d,q,\alpha}\mathcal E_q(\nu_i-\mu)^{\alpha/q}.
\end{equation}
To remove this moment assumption, fix $i$ and set
$\nu_{i,R}=\mathbf1_{B(0,R)}\nu_i+\nu_i(B(0,R)^c)\delta_0$.
Since $\mu\in\mathcal P_{q/2}$ and $\mathcal E_q(\nu_i-\mu)<\infty$,
the Fourier transform $\widehat\nu_i-1$ belongs to the weighted Hilbert space in
\cite[Proposition~A.7]{rosenzweig2026wassersteingradientflowsmaximum}.
That proposition gives
\begin{equation}
 \widehat\nu_{i,R}-\widehat\nu_i\longrightarrow0
 \quad\text{in }L^2(\mathbb R^d,|\xi|^{-d-q}\,d\xi).
\end{equation}
Therefore $\mathcal E_q(\nu_{i,R}-\mu)\to\mathcal E_q(\nu_i-\mu)$.
Apply \eqref{eq:RSW-subcritical-moment-control} to the compactly supported
$\nu_{i,R}$ and let $R\to\infty$. The moments on the left increase to
$\mathcal M_\alpha(\nu_i)$, so the same bound holds for $\nu_i$.
Taking the supremum over $i$ proves \eqref{eq:energy-sublevel-moment-bound};
Markov's inequality gives tightness. In particular, when $q>1$, finite energy already implies a finite $(q-1)$-moment because $q-1<q/2$.
\end{proof}

For limits that need not have a finite \(q\)-moment, define the normalized
potential increment
\begin{equation}
\label{eq:extended-potential-increment}
  \Delta_\rho(x,z)
  :=\int_{\mathbb R^d}
  \bigl(K(x-y)-K(x+z-y)\bigr)
  \,d(\rho-\mu)(y).
\end{equation}
If \(0<q\le1\), the integral in
\eqref{eq:extended-potential-increment} is absolutely convergent because
\begin{equation}
\label{eq:sublinear-kernel-increment}
  \bigl|K(x-y)-K(x+z-y)\bigr|\le |z|^q.
\end{equation}
If \(1<q<2\), it is locally uniformly absolutely convergent whenever
$\int |y|^{q-1}\,d|\rho-\mu|(y)<\infty$, in particular when
$\rho,\mu\in\mathcal P_{q-1}(\mathbb R^d)$, since,
for every compact \(Q\Subset\mathbb R^d\times\mathbb R^d\),
\begin{equation}
\label{eq:superlinear-kernel-increment}
  \bigl|K(x-y)-K(x+z-y)\bigr|
  \le C_Q(1+|y|^{q-1}),
  \qquad (x,z)\in Q.
\end{equation}
Because \(q-1<q/2\), when \(1<q<2\) we may choose
\(\alpha\in(q-1,q/2)\). Lemma~\ref{lem:energy-sublevels} supplies a uniform
\(\alpha\)-moment bound on bounded energy sublevels. We therefore extend
\eqref{eq:weighted-jump-dissipation} to these limits by
\begin{equation}
\label{eq:extended-weighted-jump-dissipation}
  \mathcal D_{\mathrm j}(\rho)
  :=\int_{\mathbb R^d}\int_{\mathbb R^d}
  \lambda(x)[\Delta_\rho(x,z)]_+^2
  \kappa(z)\,dz\,d\rho(x)
  \in[0,\infty].
\end{equation}
This agrees with \eqref{eq:weighted-jump-dissipation} in the solution
class. 

\begin{lemma}
\label{lem:potential-increment-convergence}
Suppose that $(\rho_n)_n\subset\mathcal P(\mathbb R^d)$ satisfies
\begin{equation}
\label{eq:potential-increment-convergence-assumptions}
  \rho_n\rightharpoonup\rho_\infty
  \quad\text{narrowly},
  \qquad
  \sup_n\mathcal E_q(\rho_n-\mu)<\infty.
\end{equation}
Then for every \(0<\alpha<q/2\)
\begin{equation}
\label{eq:limit-subcritical-moments-and-energy}
  \mathcal M_\alpha(\rho_\infty)<\infty,
  \qquad
  \mathcal E_q(\rho_\infty-\mu)<\infty,
\end{equation}
and
\begin{equation}
\label{eq:local-uniform-potential-increment-convergence}
  \Delta_{\rho_n}\to\Delta_{\rho_\infty}
  \quad\text{locally uniformly on }
  \mathbb R^d\times\mathbb R^d.
\end{equation}
\end{lemma}
\begin{proof}
The moment assertion follows from
\eqref{eq:energy-sublevel-moment-bound} and lower semicontinuity.  Fatou's
lemma in \eqref{eq:extended-fourier-energy} gives the energy assertion. For fixed \((x,z)\), let
\begin{equation}
\label{eq:increment-test-function}
  G_{x,z}(y):=K(x-y)-K(x+z-y).
\end{equation}
When \(0<q\le1\), \eqref{eq:sublinear-kernel-increment} and
\begin{equation}
\label{eq:increment-parameter-equicontinuity}
  \|G_{x,z}-G_{x',z'}\|_{L^\infty}
  \le |x-x'|^q+|x+z-x'-z'|^q
\end{equation}
reduce \eqref{eq:local-uniform-potential-increment-convergence} on each
compact parameter set to narrow convergence and a finite-net argument.

When \(1<q<2\), choose \(\alpha\in(q-1,q/2)\).  On a compact parameter
set, \eqref{eq:superlinear-kernel-increment} and
\eqref{eq:energy-sublevel-moment-bound} make the contribution from
\(|y|>R\) uniformly small as \(R\to\infty\).  After truncating to
\(|y|\le R\), the family in \eqref{eq:increment-test-function} is bounded
and equicontinuous, so narrow convergence and another finite-net argument
give uniform convergence.  This proves
\eqref{eq:local-uniform-potential-increment-convergence}.
\end{proof}

\begin{lemma}
\label{lem:weighted-jump-dissipation-lsc}
Under the hypotheses of
Lemma~\ref{lem:potential-increment-convergence},
\begin{equation}
\label{eq:weighted-jump-dissipation-lsc}
  \mathcal D_{\mathrm j}(\rho_\infty)
  \le\liminf_{n\to\infty}
  \mathcal D_{\mathrm j}(\rho_n).
\end{equation}
\end{lemma}

\begin{proof}
For \(R>0\), let \(\chi_R,\Theta_R\in C_c^\infty(B(0,R))\) be nonnegative cutoffs such that
\begin{equation}
0\leq \chi_R,\Theta_R\leq 1,\qquad
\chi_R=\Theta_R=1\quad\text{on }B(0,R/2),
\end{equation}
and such that \(\chi_R\uparrow 1\) and \(\Theta_R\uparrow 1\) pointwise as \(R\to\infty\).  By
\eqref{eq:local-uniform-potential-increment-convergence}, define
\begin{equation}
\label{eq:truncated-weighted-dissipation-function}
  H_{n,R}(x)
  :=\chi_R(x)\lambda(x)
  \int_{\mathbb R^d}\Theta_R(z)
  [\Delta_{\rho_n}(x,z)]_+^2\kappa(z)\,dz
\end{equation}
so that \(H_{n,R}\to H_{\infty,R}\) uniformly and
\(H_{\infty,R}\in C_c(\mathbb R^d)\).  Narrow convergence therefore
identifies the corresponding truncated integrals.  Since the integrands
are nonnegative, monotone convergence as \(R\to\infty\) gives
\eqref{eq:weighted-jump-dissipation-lsc}.
\end{proof}

Normalizing the first variation removes a potentially divergent constant and
permits pairing with the Fourier energy without a \(q\)-moment. For
\(\rho\in\mathcal P(\mathbb R^d)\), set
\begin{equation}
\label{eq:normalized-witness}
  \widetilde f_\rho(x)
  :=\int_{\mathbb R^d}
  \bigl(K(x-y)-K(-y)\bigr)\,d(\rho-\mu)(y),
\end{equation}
whenever the integral is defined. For $\rho,\mu\in\mathcal P_q(\mathbb R^d)$,
this is precisely $\widetilde f_\rho(x)=f_\rho(x)-f_\rho(0)$. Under the hypotheses used in
\eqref{eq:extended-potential-increment}, \(\widetilde f_\rho\) is finite and continuous
and satisfies
\begin{equation}
\label{eq:normalized-witness-increment}
  \widetilde f_\rho(x)-\widetilde f_\rho(x+z)=\Delta_\rho(x,z).
\end{equation}
\begin{lemma}
\label{lem:normalized-witness-pairing}
Let \(\rho\in\mathcal P(\mathbb R^d)\) and assume that
\(\mathcal E_q(\rho-\mu)<\infty\). If
\(\widetilde f_\rho\in L^1(|\rho-\mu|)\), then
\begin{equation}
\label{eq:normalized-witness-pairing}
  \int_{\mathbb R^d}\widetilde f_\rho\,d(\rho-\mu)
  =2\mathcal E_q(\rho-\mu).
\end{equation}
\end{lemma}
\begin{proof}
Set $\sigma:=\rho-\mu$ and, for $R>0$, define
\begin{equation}
  \rho_R:=\mathbf 1_{B(0,R)}\rho+\rho(B(0,R)^c)\delta_0,
  \qquad
  \mu_R:=\mathbf 1_{B(0,R)}\mu+\mu(B(0,R)^c)\delta_0,
  \qquad \sigma_R:=\rho_R-\mu_R.
\end{equation}
Since $\mu\in\mathcal P_{q/2}(\mathbb R^d)$ and
$\mathcal E_q(\sigma)<\infty$,
\cite[Proposition~A.7]{rosenzweig2026wassersteingradientflowsmaximum}
gives
\begin{equation}
\label{eq:normalized-witness-truncation-energy}
  \widehat\sigma_R\to\widehat\sigma
  \quad\text{in }L^2(\mathbb R^d,|\xi|^{-d-q}\,d\xi).
\end{equation}
Define
\begin{equation}
  U_R(x):=\int_{\mathbb R^d}
  \bigl(K(x-y)-K(-y)\bigr)\,d\sigma_R(y).
\end{equation}
The increment bounds
\eqref{eq:sublinear-kernel-increment}--\eqref{eq:superlinear-kernel-increment}
imply $U_R\to\widetilde f_\rho$ locally uniformly. For $q>1$, the tail
is controlled by the $(q-1)$-moments of $\rho$ and $\mu$, which are finite by Lemma~\ref{lem:energy-sublevels} and $q-1<q/2$.
For finite $R,S$, the measures $\sigma_R,\sigma_S$ have compact support
and zero mass. Polarization of the Fourier identity in
\cite[Lemma~2.1]{rosenzweig2026wassersteingradientflowsmaximum} gives
\begin{equation}
\label{eq:normalized-witness-truncation-pairing}
  \int_{\mathbb R^d}U_S\,d\sigma_R
  =2c_{d,q}\operatorname{Re}
  \int_{\mathbb R^d}
  \frac{\widehat\sigma_S(\xi)\overline{\widehat\sigma_R(\xi)}}
       {|\xi|^{d+q}}\,d\xi.
\end{equation}
First let $S\to\infty$, using local uniform convergence on the left
and \eqref{eq:normalized-witness-truncation-energy} on the right.
Since $\widetilde f_\rho(0)=0$, this yields
\begin{equation}
  \int_{B(0,R)}\widetilde f_\rho\,d\sigma
  =2c_{d,q}\operatorname{Re}
  \int_{\mathbb R^d}
  \frac{\widehat\sigma(\xi)\overline{\widehat\sigma_R(\xi)}}
       {|\xi|^{d+q}}\,d\xi.
\end{equation}
Now let $R\to\infty$. The assumed integrability
$\widetilde f_\rho\in L^1(|\sigma|)$ and
\eqref{eq:normalized-witness-truncation-energy} identify the two limits,
which proves \eqref{eq:normalized-witness-pairing}.
\end{proof}
For comparison, Lagrangian criticality permits non-target singular measures. For $d\ge2$, $q=1$, and a rotation-invariant absolutely continuous target $\mu\in\mathcal P_1(\mathbb R^d)$, \cite[Proposition~5.21]{rosenzweig2026wassersteingradientflowsmaximum} gives a unique positive radius for which the uniform probability measure on the sphere centered at the origin is a non-target Lagrangian critical point. Its MMD energy is finite and positive, whereas its transport dissipation vanishes. If the target is smooth and compactly supported, \cite[Corollary~5.22]{rosenzweig2026wassersteingradientflowsmaximum} gives smooth approximations supported in a common compact set, with positive limiting MMD energy and transport dissipation tending to zero. Their $L^p$ norms diverge for every $p>1$. This is a static approximation of a critical measure; the occurrence of such a measure as a limit of a density trajectory remains a separate dynamical question.

We now establish jump-dissipation rigidity for finite-energy measures,
using only their normalized first-variation increments. Together with
\eqref{eq:stationarity-zero-total-dissipation} in its stated density class,
this proves Theorem~\ref{prop:no-spur-stat-soln-intro}. The additional condition $r\ge2$ is not needed for this proposition.
\begin{proposition}[Rigidity of jump dissipation]
\label{prop:no-spur-stat-solns}
Assume the target, kernel, proposal, and spatial-rate hypotheses of Theorem~\ref{thm:continuum-wellposedness}. Let $\rho\in\mathcal P(\mathbb R^d)$ satisfy $\mathcal E_q(\rho-\mu)<\infty$. If \(\mathcal{D}_{\mathrm j}(\rho)=0\), then $\rho = \mu$.
\end{proposition}
\begin{proof}
Since \(\mathcal D_j(\rho)=0\) and \(\lambda(x)>0\), for
\(\rho\)-a.e. \(x\),
\begin{equation}
  \int_{\mathbb R^d}
  [\Delta_\rho(x,z)]_+^2\kappa(z)\,dz=0.
\end{equation}
Fix such an \(x\).  The map
\(
  z\mapsto [\Delta_\rho(x,z)]_+
\)
is continuous.  Since the measure \(\kappa(z)\,dz\) has full support, it follows that
\([\Delta_\rho(x,z)]_+=0\) for every \(z\in\mathbb R^d\).
Using \eqref{eq:normalized-witness-increment}, we obtain
\(\widetilde f_\rho(x)\le\widetilde f_\rho(x+z)\)
  for every \(z\in\mathbb R^d\).
Thus, for \(\rho\)-a.e. \(x\), \(\widetilde f_\rho(x)\) is the
global minimum of \(\widetilde f_\rho\).  Denote this minimum by
\(m_\rho\).  Then
\(\widetilde f_\rho=m_\rho\)  \(\rho\)-a.e., and \(\widetilde f_\rho\ge m_\rho\) on \(\mathbb R^d\).

For \(0<q\le1\), the function \(\widetilde f_\rho\) is globally
\(q\)-H\"older continuous.  For \(1<q<2\), it is globally Lipschitz
because $\rho,\mu\in\mathcal P_{q-1}(\mathbb R^d)$ by Lemma~\ref{lem:energy-sublevels}.
Since \(\mu\in\mathcal P_r(\mathbb R^d)\), where
\(r\ge\max\{1,q\}\), it follows that
\(\widetilde f_\rho\in L^1(\mu)\).
Moreover, \(\widetilde f_\rho=m_\rho\) \(\rho\)-a.e., so
\(\widetilde f_\rho\in L^1(\rho)\).  Hence
\(\widetilde f_\rho\in L^1(|\rho-\mu|)\), and
Lemma~\ref{lem:normalized-witness-pairing} gives
\begin{equation}
\begin{aligned}
  2\mathcal E_q(\rho-\mu)
  &=\int_{\mathbb R^d}\widetilde f_\rho\,d(\rho-\mu)=m_\rho-\int_{\mathbb R^d}\widetilde f_\rho\,d\mu
  \le0.
\end{aligned}
\end{equation}
Since \(\mathcal E_q(\rho-\mu)\ge0\), we conclude that
\(\mathcal E_q(\rho-\mu)=0\).  The Fourier representation
\eqref{eq:extended-fourier-energy} then implies \(\rho=\mu\).
\end{proof}

\subsection{Convergence to the target}
\label{subsec:convergence-target}

We now prove Theorem~\ref{thm:continuum-convergence}, beginning with
narrow convergence.

\begin{lemma}
\label{thm:narrow-convergence}
\begin{equation}
\label{eq:unconditional-narrow-convergence}
  \rho_t\rightharpoonup\mu
  \quad\text{narrowly in }\mathcal P(\mathbb R^d)
  \quad\text{as }t\to\infty.
\end{equation}
\end{lemma}

\begin{proof}
Equation~\eqref{eq:energy-dissipation-identity} gives
\begin{equation}
\label{eq:uniform-energy-and-integrated-dissipation}
  \sup_{t\ge0}\mathcal E_q(\rho_t-\mu)
  \le\mathcal E_q(\rho_0-\mu),
  \qquad
  \int_0^\infty\mathcal D(\rho_t)\,dt
  \le\mathcal E_q(\rho_0-\mu)<\infty,
\end{equation}
so Lemma~\ref{lem:energy-sublevels} applies to give tightness of $(\rho_t)_{t \geq 0}$. Now fix any increasing sequence $(t_n)_n$ that grows to $\infty$. By (\ref{eq:uniform-energy-and-integrated-dissipation}), we may choose $s_n \in [t_n, t_n + 1]$ such that
\begin{equation}
\label{eqn:selected-small-dissipiation-time}
    \mathcal D(\rho_{s_n}) \leq \int_{t_n}^{t_n + 1} \mathcal D(\rho_\tau) \, d\tau \to 0.
\end{equation}
By tightness, every subsequence of $(\rho_{s_n})_n$ has a narrowly convergent further subsequence $(\rho_{s_{n_j}})_j$; let the limit be denoted $\rho_\infty$. By (\ref{eqn:selected-small-dissipiation-time}) and Lemma~\ref{lem:weighted-jump-dissipation-lsc}, we conclude that $\mathcal D_{\mathrm j}(\rho_\infty) = 0$.  When \(q>1\), choose
\(\alpha\in(q-1,q/2)\) in
\eqref{eq:limit-subcritical-moments-and-energy} to see that $\rho_\infty \in \mathcal P_{q - 1}(\R^d)$. Proposition~\ref{prop:no-spur-stat-solns} then shows that $\rho_\infty = \mu$, and therefore $(\rho_{s_n})_n \to \mu$ narrowly. Finally, Lemma~\ref{lem:W1-dissipation} proves that 
\begin{equation}
  W_1(\rho_{t_n},\rho_{s_n})
  \le \left(1+(\|\lambda\|_{L^\infty}\mathcal M_2(\kappa))^{1/2}\right)\left(\int_{t_n}^{t_n+1}\mathcal D(\rho_t)\,dt\right)^{1/2}\to 0
\end{equation}
and therefore \(\rho_{t_n}\rightharpoonup\mu\). Since $(t_n)_n$ was arbitrary, conclude that \eqref{eq:unconditional-narrow-convergence} is true.
\end{proof}

To upgrade narrow convergence to MMD convergence, we use the
second-moment identity to show that the time-averaged energy vanishes.
Energy monotonicity then gives convergence at every time.

\begin{proposition}[Convergence in MMD]
\label{thm:critical-MMD-second-moment-convergence}
\begin{equation}
\label{eq:critical-MMD-second-moment-convergence}
  \mathcal E_q(\rho_t-\mu)
  =\operatorname{MMD}_q^2(\rho_t,\mu)
  \to0
  \qquad\text{as }t\to\infty.
\end{equation}
\end{proposition}

\begin{proof}
By \eqref{eq:energy-dissipation-identity},
\begin{equation}
\label{eq:critical-MMD-energy-control}
  \sup_{t\ge0}\mathcal E_q(\rho_t-\mu)
  \le\mathcal E_q(\rho_0-\mu),
  \qquad
  \int_0^\infty\mathcal D(\rho_t)\,dt
  \le\mathcal E_q(\rho_0-\mu)<\infty.
\end{equation}
For \(q\le1\), \eqref{eq:acceptance-coefficient} gives
\(A_{\rho_t}=2\). For \(q>1\), use \(q-1<q/2\) in
\eqref{eq:RSW-subcritical-moment-control} and then
\eqref{eq:critical-MMD-energy-control}. Thus, in either case,
\begin{equation}
\label{eq:critical-MMD-uniform-acceptance}
  \sup_{t\ge0}A_{\rho_t}<\infty,
  \qquad
  \lambda(x)\zeta_{\rho_t}(x,x+z)
  \le C|z|^{1\wedge q}.
\end{equation}
All constants in this proof may depend on the fixed data, but are
independent of time. Cauchy-Schwarz and
\eqref{eq:transport-dissipation}--\eqref{eq:total-dissipation} give
\begin{subequations}\label{eq:critical-MMD-second-moment-pairings}
\begin{align}
  \int_{\mathbb R^d}
       |x||v_{\rho_t}(x)|\,d\rho_t(x)
  &\le
  \mathcal M_2(\rho_t)^{1/2}
  \mathcal D_{\mathrm{tr}}(\rho_t)^{1/2},\\
  \int_{\mathbb R^d}\int_{\mathbb R^d}
       |x||z|\lambda(x)\zeta_{\rho_t}(x,x+z)
       \kappa(z)\,dz\,d\rho_t(x)
  &\le
  \bigl(\|\lambda\|_{L^\infty}\mathcal M_2(\kappa)\bigr)^{1/2}
  \mathcal M_2(\rho_t)^{1/2}
  \mathcal D_{\mathrm j}(\rho_t)^{1/2}.
\end{align}
\end{subequations}
Moreover, \eqref{eq:critical-MMD-uniform-acceptance} implies
\begin{equation}
\label{eq:critical-MMD-quadratic-jump-bound}
  \int_{\mathbb R^d}\int_{\mathbb R^d}
       |z|^2\lambda(x)\zeta_{\rho_t}(x,x+z)
       \kappa(z)\,dz\,d\rho_t(x)
  \le C\mathcal M_{2+(1\wedge q)}(\kappa)<\infty.
\end{equation}

To test \eqref{eq:pde} with \(|x|^2\), choose
\(\chi\in C_c^\infty(B_2)\), equal to one on \(B_1\), and set
\(\varphi_R(x)=|x|^2\chi(x/R)\). Then
\begin{equation}
  |\nabla\varphi_R(x)|\le C|x|,
  \qquad
  \|D^2\varphi_R\|_{L^\infty}\le C,
  \qquad
  |\varphi_R(x+z)-\varphi_R(x)|
  \le C\bigl(|x||z|+|z|^2\bigr).
\end{equation}
The finite-time second-moment bound in
Definition~\ref{def:continuum-weak-solution},
\eqref{eq:critical-MMD-energy-control}, and
\eqref{eq:critical-MMD-second-moment-pairings}-
\eqref{eq:critical-MMD-quadratic-jump-bound}
justify dominated convergence as \(R\to\infty\).
For each \(R\), time cutoffs in the weak formulation and continuity
of \(t\mapsto\int\varphi_R\,d\rho_t\) give its integrated form
on every interval \([s,t]\).
Passing \(R\to\infty\) in this identity for every \(0\le s\le t\)
shows that the actual map \(t\mapsto\mathcal M_2(\rho_t)\) is locally
absolutely continuous and, for a.e. \(t>0\),
\begin{equation}
\label{eq:critical-MMD-second-moment-identity}
  \frac{d}{dt}\mathcal M_2(\rho_t)
  =2\int_{\mathbb R^d}x\cdot v_{\rho_t}(x)\,d\rho_t(x)
  +\int_{\mathbb R^d}\int_{\mathbb R^d}
    \bigl(2x\cdot z+|z|^2\bigr)\lambda(x)\zeta_{\rho_t}(x,x+z)
    \kappa(z)\,dz\,d\rho_t(x).
\end{equation}

We next estimate the time integrals in
\eqref{eq:critical-MMD-second-moment-identity}.
For \(0<S<T\), \eqref{eq:critical-MMD-energy-control} gives
\begin{equation}
  \frac1{\sqrt T}\int_0^T\mathcal D(\rho_t)^{1/2}\,dt
  \le
  \frac1{\sqrt T}\int_0^S\mathcal D(\rho_t)^{1/2}\,dt
  +\left(\int_S^\infty\mathcal D(\rho_t)\,dt\right)^{1/2}.
\end{equation}
Taking \(T\to\infty\) and then \(S\to\infty\), we obtain
\begin{equation}
\label{eq:critical-MMD-square-root-dissipation}
  \int_0^T\mathcal D(\rho_t)^{1/2}\,dt=o(\sqrt T).
\end{equation}
For each \(R\ge1\), Cauchy-Schwarz on \(\{|z|\le R\}\) and
\eqref{eq:critical-MMD-uniform-acceptance} on \(\{|z|>R\}\) yield
\begin{align}
  &\frac1T\int_0^T\int_{\mathbb R^d}\int_{\mathbb R^d}
       |z|^2\lambda(x)\zeta_{\rho_t}(x,x+z)
       \kappa(z)\,dz\,d\rho_t(x)\,dt\nonumber \\
  &\quad\le
  \left(\|\lambda\|_{L^\infty}
         \int_{|z|\le R}|z|^4\kappa(z)\,dz\right)^{1/2}
  \left(\frac{\mathcal E_q(\rho_0-\mu)}T\right)^{1/2}+
  C\int_{|z|>R}|z|^{2+(1\wedge q)}\kappa(z)\,dz.\label{eq:critical-MMD-quadratic-jump-average}
\end{align}
Letting \(T\to\infty\) and then \(R\to\infty\) proves that the
left-hand side of
\eqref{eq:critical-MMD-quadratic-jump-average} tends to zero.

Integrate \eqref{eq:critical-MMD-second-moment-identity}, use
\eqref{eq:critical-MMD-second-moment-pairings}, and take the supremum
over \(0\le s\le T\). Young's inequality gives
\begin{align}
  \sup_{0\le s\le T}\mathcal M_2(\rho_s)
  \le{}&2\mathcal M_2(\rho_0)
  +C\left(\int_0^T\mathcal D(\rho_t)^{1/2}\,dt\right)^2\nonumber \\
  &+2\int_0^T\int_{\mathbb R^d}\int_{\mathbb R^d}
       |z|^2\lambda(x)\zeta_{\rho_t}(x,x+z)
       \kappa(z)\,dz\,d\rho_t(x)\,dt
  =o(T).\label{eq:critical-MMD-sublinear-second-moment}
\end{align}
Combining \eqref{eq:critical-MMD-second-moment-pairings},
\eqref{eq:critical-MMD-square-root-dissipation}, and
\eqref{eq:critical-MMD-sublinear-second-moment}, we also have
\begin{equation}
\label{eq:critical-MMD-linear-jump-average}
  \frac1T\int_0^T
    \left|\iint_{\mathbb R^{2d}}x\cdot z\,\lambda(x)
      \zeta_{\rho_t}(x,x+z)\kappa(z)\,dz\,d\rho_t(x)\right|dt
  \le\frac C T\sup_{0\le s\le T}\mathcal M_2(\rho_s)^{1/2}
    \int_0^T\mathcal D_{\mathrm j}(\rho_t)^{1/2}\,dt,
\end{equation}
whose right-hand side tends to zero.

Define the scalar function
\begin{equation}
\label{eq:critical-MMD-target-function}
  G_\mu(x):=
  \int_{\mathbb R^d}
       y\cdot(x-y)|x-y|^{q-2}\,d\mu(y).
\end{equation}
If \(q=1\), then \(G_\mu\) is bounded by \(\mathcal M_1(\mu)\);
dominated convergence, using the absolute continuity of \(\mu\),
gives continuity.
If \(0<q<1\), Cauchy-Schwarz gives
\begin{equation}
  |G_\mu(x)|
  \le\mathcal M_2(\mu)^{1/2}
     \left(\int_{\mathbb R^d}
       |x-y|^{-2(1-q)}\,d\mu(y)\right)^{1/2}.
\end{equation}
The admissibility condition on \(p\) implies
\(2(1-q)p_*<(2-q)p_*<d\), with \(p_*=1\) when \(p=\infty\).
Splitting at \(|x-y|=1\) proves boundedness. In addition,
\begin{equation}
  \sup_x
  \left|\int_{|x-y|<\delta}
       y\cdot(x-y)|x-y|^{q-2}\,d\mu(y)\right|
  \le C\mathcal M_2(\mu)^{1/2}\|\mu\|_{L^p}^{1/2}
       \delta^{d/(2p_*)-(1-q)}
  \to 0.
\end{equation}
A continuous cutoff outside this ball, followed by dominated
convergence against \(|y|\,d\mu(y)\), proves continuity.
Thus \(G_\mu\in C_b(\mathbb R^d)\) when \(q\le1\). If \(1<q<2\), dominated convergence gives continuity, and
\begin{equation}
\label{eq:critical-MMD-target-growth}
  |G_\mu(x)|
  \le C_q\left(
      \mathcal M_1(\mu)|x|^{q-1}+\mathcal M_q(\mu)
  \right).
\end{equation}
Choose \(q-1<\alpha<q/2\).
Equations \eqref{eq:RSW-subcritical-moment-control} and
\eqref{eq:critical-MMD-energy-control} imply
\begin{equation}
  \sup_{t\ge0}\mathcal M_\alpha(\rho_t)<\infty,
  \qquad
  \sup_{t\ge0}\int_{|x|>R}|x|^{q-1}\,d\rho_t(x)
  \le R^{q-1-\alpha}
       \sup_{t\ge0}\mathcal M_\alpha(\rho_t)
  \to 0.
\end{equation}
Consequently, in all cases,
\Cref{thm:narrow-convergence} implies
\begin{equation}
\label{eq:critical-MMD-target-pairing-limit}
  \int_{\mathbb R^d}G_\mu(x)\,d(\rho_t-\mu)(x)
  \to 0.
\end{equation}

Since \(v_{\rho_t}(x)
=q\int(x-y)|x-y|^{q-2}\,d(\rho_t-\mu)(y)\), oddness gives
\begin{align}
  2\int_{\mathbb R^d}x\cdot v_{\rho_t}(x)\,d\rho_t(x)
  ={}&
  q\iint_{\mathbb R^d\times\mathbb R^d}
       |x-y|^q\,d\rho_t(x)\,d\rho_t(y)\nonumber \\
  &-2q\iint_{\mathbb R^d\times\mathbb R^d}
       |x-y|^q\,d\rho_t(x)\,d\mu(y)
       -2q\int_{\mathbb R^d}G_\mu\,d\rho_t\nonumber \\
  ={}&
  -2q\mathcal E_q(\rho_t-\mu)
       -2q\int_{\mathbb R^d}G_\mu\,d(\rho_t-\mu).\label{eq:critical-MMD-second-moment-transport}
\end{align}
The integrals are absolutely convergent: for \(q\ge1\) this follows
from the second moments; for \(q<1\), split at \(|x-y|=1\) and use
the finite-time \(L^p\) bounds and \((1-q)p_*<d\).

Substitute \eqref{eq:critical-MMD-second-moment-transport} into
\eqref{eq:critical-MMD-second-moment-identity}, integrate over
\([0,T]\), and divide by \(T\). Since
\(\mathcal M_2(\rho_T)\ge0\), we obtain
\begin{equation}
  0\le
  \frac{2q}{T}\int_0^T\mathcal E_q(\rho_t-\mu)\,dt
  \le\frac{\mathcal M_2(\rho_0)}T+o(1)
  \to 0
\end{equation}
by \eqref{eq:critical-MMD-quadratic-jump-average},
\eqref{eq:critical-MMD-linear-jump-average}, and
\eqref{eq:critical-MMD-target-pairing-limit}.
Finally, \eqref{eq:energy-dissipation-identity} implies
\begin{equation}
\label{eq:energy-decay}
      0\le\mathcal E_q(\rho_T-\mu)
  \le\frac1T\int_0^T\mathcal E_q(\rho_t-\mu)\,dt
  \to 0.
\end{equation}
\end{proof}

\begin{remark}[Comparison with pure transport]
\label{rem:pure-transport-critical-MMD}
The second-moment argument also applies to the pure transport equation
\eqref{eq:ordinary}. Let $\rho_0,\mu\in\mathcal P_2(\mathbb R^d)\cap
L^p(\mathbb R^d)$, with $d,q,p$ as in
Theorem~\ref{thm:continuum-wellposedness}, and let $(\rho_t)_{t\ge0}$
be the corresponding pure transport solution. If
$\rho_t\rightharpoonup\mu$ narrowly, then the same argument leads to
\eqref{eq:critical-MMD-target-pairing-limit}.
Since the jump contribution is absent,
\eqref{eq:critical-MMD-second-moment-transport} yields
\begin{equation}
      \frac{d}{dt}\mathcal M_2(\rho_t)
  =-2q\mathcal E_q(\rho_t-\mu)
   -2q\int_{\mathbb R^d}G_\mu\,d(\rho_t-\mu).
\end{equation}
Integrating over $[0,T]$, using $\mathcal M_2(\rho_T)\ge0$, and applying
energy monotonicity, we obtain the same conclusion \eqref{eq:energy-decay}.
Thus narrow convergence to the target implies convergence in
$\operatorname{MMD}_q$ under finite initial second moments. The assumption $\rho_t\rightharpoonup\mu$ must still be established.
Vanishing transport dissipation gives only
$\nabla f_\rho=0$ $\rho$-a.e., which does not identify
$\rho$ with $\mu$ outside the hypotheses of the rigidity results in
\cite[Theorem~1.5 and Section~5]
{rosenzweig2026wassersteingradientflowsmaximum}.
Finite-time moment and $L^p$ bounds do not by themselves place narrow limits
in these rigidity classes.
\end{remark}

\subsection{Conditional far-field decay}
\label{subsec:farfield}
We now specialize to energy distance, $q=1$, and assume $r\ge2$.
The goal is the stopped exponential decay in Proposition~\ref{thm:farfield-pl},
which gives Proposition~\ref{thm:far-field-main}. We use a Gaussian proposal
and the capped spatial rate in \eqref{eq:kappa-lambda-ex}, fixed from
$\rho_0\ne\mu$. The proof combines two static dissipation estimates,
bounds uniform over a preliminary bandwidth interval, and control of the
concentration radii. We then choose the parameters in a priori order and
apply the estimates until the concentration or separation conditions fail.
The exit-time bounds give a positive interval under strict initial margins;
the full statement follows the required parameter definitions below.

We first fix the notation for concentration balls and the Gaussian proposal. Write $r_*=r/(r-1)$ and let $\gamma_d$ denote the standard Gaussian density,
so $\kappa_\beta(z)=\beta^{-d}\gamma_d(z/\beta)$. Constants associated with
$\gamma_d$ will be independent of the bandwidth $\beta$.
For a probability measure $\nu$ with a finite first moment, set
\begin{equation}
 x_\nu=\int x\,d\nu(x),\qquad
 F_\nu(x)=\int |x-y|\,d\nu(y),\qquad
 I_1(\nu)=\iint|x-y|\,d\nu(x)\,d\nu(y).
\end{equation}
Thus $f_\rho=F_\mu-F_\rho$. In this subsection, $B_s(x)$ denotes the open
ball of radius $s$ centered at $x$. Choose positive radii $s,s_\mu$ such that
\begin{equation}
 \rho(B_s(x_\rho))\ge1-\eta,\qquad
 \mu(B_{s_\mu}(x_\mu))\ge1-\eta_\mu,
 \qquad 0<\eta,\eta_\mu<1,
\end{equation}
and write
\begin{subequations}\label{eq:farfield-static-notation}
\begin{gather}
 B_\rho=B_s(x_\rho),\qquad B_\mu=B_{s_\mu}(x_\mu),\qquad
 R=\operatorname{dist}(B_\rho,B_\mu)>0,\\
 \varepsilon_\rho=\eta^{1/r_*}\mathcal M_r(\rho)^{1/r},\qquad
 \varepsilon_\mu=\eta_\mu^{1/r_*}\mathcal M_r(\mu)^{1/r}.
\end{gather}
\end{subequations}
The separation $R$ is always calculated using these same balls.

The static estimates assume $H=\|D^2f_\rho\|_\infty<\infty$, so we use
the $C^{1,1}$ representative of $f_\rho$. This assumption excludes atoms
of $\rho$: at an atom, its contribution to $F_\rho$ is a nonzero multiple
of a norm, whereas the remaining nonatomic contribution and $F_\mu$ are
differentiable there by dominated convergence. In particular, the force
integrals used below are well defined. Set
\begin{equation}
 Z(x)=\int [f_\rho(x)-f_\rho(x+z)]_+\kappa_\beta(z)\,dz,
\end{equation}
and, for any unit vector $e$, define
\begin{equation}\label{eq:farfield-Gaussian-cone-constant}
 a_d=\int_{\{-e\cdot z\ge |z|/2\}}(-e\cdot z)\gamma_d(z)\,dz>0.
\end{equation}
Rotation invariance makes this constant independent of $e$. The vector
integral over the cone is $-a_de$, and the Gaussian second moment is
$d\beta^2$. Taylor's inequality and $[u]_+\ge u$ therefore give
\begin{equation}\label{eq:farfield-cone-Taylor}
 Z(x)\ge a_d\beta\nabla f_\rho(x)\cdot e-\frac d2\beta^2H.
\end{equation}
We may choose $e$ along $\nabla f_\rho(x)$; if the gradient is zero, the
resulting nonpositive lower bound is still valid.

\begin{lemma}\label{prop:far-field-pl-prop}
Suppose that $\rho,\mu\in\mathcal P_r(\mathbb R^d)$ and the preceding
notation and regularity assumptions hold. Assume
\begin{equation}
 I_1(\rho)\ge\theta s,\qquad
 \varepsilon_\rho\le\theta s/16,\qquad
 \eta_\mu\le\theta/8,\qquad R\ge56s/\theta,
 \qquad\beta H\le a_d\theta/(16d).
\end{equation}
If $\lambda(x)\ge c_\lambda R^2$ on $B_\rho$, then
\begin{equation}\label{eq:farfield-static-one-dissipation}
 \mathcal D(\rho)\ge\mathcal D_{\mathrm j}(\rho)
 \ge\frac{c_\lambda a_d^2\theta^2}{1024}\,\beta^2R^2.
\end{equation}
Moreover,
\begin{equation}\label{eq:farfield-static-one-energy}
 \mathcal E_1(\rho-\mu)
 \le R+2s+2s_\mu+2\varepsilon_\rho+2\varepsilon_\mu-\theta s/2.
\end{equation}
\end{lemma}
\begin{proof}
Since $I_1(\rho)\le2\mathcal M_r(\rho)^{1/r}$, the assumptions imply
$\eta^{1/r_*}\le1/8$, hence $\eta\le1/8$. By H\"older's inequality on the
tail set, whose mass is at most $\eta$,
\begin{equation}\label{eqn:tail-err-est-1}
 \int_{B_\rho^c}|x|\,d\rho(x)\le\varepsilon_\rho\le\theta s/16.
\end{equation}
Using $|x-y|\le|x|+|y|$ and $\eta\le\eta^{1/r_*}$, the contribution to
$I_1(\rho)$ outside $B_\rho^2$ is at most $4\varepsilon_\rho$. Consequently,
\begin{equation}
 J_\rho:=\iint_{B_\rho^2}|x-y|\,d\rho(x)\,d\rho(y)\ge3I_1(\rho)/4.
\end{equation}
Let $m=\rho(B_\rho)$ and $\bar x=m^{-1}\int_{B_\rho}x\,d\rho(x)$.
Pair symmetrization inside the ball and the bounded force of the tail give
\begin{equation}
 \int_{B_\rho}\nabla F_\rho(x)\cdot(x-\bar x)\,d\rho(x)
 \ge J_\rho/2-\eta\int_{B_\rho}|x-\bar x|\,d\rho(x).
\end{equation}
Jensen's inequality bounds the last integral by $J_\rho/m$. It follows that
\begin{equation}\label{eqn:F-rho-est-1}
 \int_{B_\rho}\nabla F_\rho(x)\cdot(x-\bar x)\,d\rho(x)
 \ge\left(\frac12-\frac{\eta}{1-\eta}\right)J_\rho
 \ge\frac{15}{56}I_1(\rho).
\end{equation}
The field of $\mu|_{B_\mu}$ has Hessian norm at most $R^{-1}$ throughout
the convex source ball. Subtracting its value at $\bar x$ gives
\begin{equation}\label{eqn:F-mu-est-1}
 \left|\int_{B_\rho}\nabla F_{\mu|_{B_\mu}}(x)\cdot(x-\bar x)\,d\rho(x)\right|
 \le R^{-1}\int_{B_\rho}|x-\bar x|^2\,d\rho(x)\le s^2/R.
\end{equation}
Here the variance about $\bar x$ is at most the second moment about
$x_\rho$. The remaining target field is bounded by $\eta_\mu$, so
\begin{equation}\label{eqn:F-mu-est-2}
 \left|\int_{B_\rho}\nabla F_{\mu|_{B_\mu^c}}(x)\cdot(x-\bar x)\,d\rho(x)\right|
 \le\eta_\mu s.
\end{equation}
Combining these bounds yields
\begin{equation}
 \int_{B_\rho}-\nabla f_\rho(x)\cdot(x-\bar x)\,d\rho(x)
 \ge15I_1(\rho)/56-s^2/R-\eta_\mu s
 \ge\theta s/7-s^2/R\ge\theta s/8.
\end{equation}
Since $|x-\bar x|\le2s$ on the source ball,
\begin{equation}\label{eqn:witness-lb}
 \int_{B_\rho}|\nabla f_\rho|\,d\rho\ge\theta/16.
\end{equation}
Integrating \eqref{eq:farfield-cone-Taylor} in the gradient direction,
and using the bandwidth assumption, gives
$\int_{B_\rho}Z\,d\rho\ge a_d\beta\theta/32$.
Jensen's inequality in the proposal and Cauchy--Schwarz against the source
subprobability give
\begin{equation}
 \mathcal D_{\mathrm j}(\rho)
 \ge\int_{B_\rho}\lambda Z^2\,d\rho
 \ge c_\lambda R^2\left(\int_{B_\rho}Z\,d\rho\right)^2,
\end{equation}
which proves \eqref{eq:farfield-static-one-dissipation}.

For the energy, expand the source and target terms:
\begin{equation}\label{eqn:far-field-energy-est-1}
 2\mathcal E_1(\rho-\mu)
 =2\iint|x-y|\,d\rho(x)\,d\mu(y)-I_1(\rho)-I_1(\mu)
 \le2\iint|x-y|\,d\rho(x)\,d\mu(y)-\theta s.
\end{equation}
The center distance is $R+s+s_\mu$, while
\begin{equation}
 \int|x-x_\rho|\,d\rho(x)\le s+2\varepsilon_\rho.
\end{equation}
Indeed, the tail of this integral is at most
$\int_{B_\rho^c}|x|\,d\rho+\eta|x_\rho|\le2\varepsilon_\rho$.
The analogous target bound proves \eqref{eq:farfield-static-one-energy}.
\end{proof}

\begin{lemma}\label{prop:spr-pl-ineq}
With the preceding notation and regularity assumptions, suppose $R>s_\mu$,
and set
\begin{equation}
 S=(1-\eta_\mu)\sqrt{1-s_\mu^2/R^2}-\eta_\mu>0,
 \qquad L_0=S(1-\eta)-I_1(\rho)/R-2\eta>0.
\end{equation}
Assume
\begin{equation}\label{eqn:beta-choice-2}
 \beta H\le a_dS/d,
\end{equation}
and $\lambda(x)\ge c_\lambda R^2$ for $|x-x_\mu|\ge R$. Then
\begin{equation}\label{eq:farfield-static-two-dissipation}
 \mathcal D(\rho)\ge\mathcal D_{\mathrm j}(\rho)
 \ge\frac{c_\lambda a_d^2}{4}\beta^2R^2L_0^2,
\end{equation}
and
\begin{equation}\label{eq:farfield-static-two-energy}
 \mathcal E_1(\rho-\mu)\le R+2(s+s_\mu+\varepsilon_\rho+\varepsilon_\mu).
\end{equation}
\end{lemma}
\begin{proof}
Let $A=\{|x-x_\mu|\ge R\}$ and $e_x=(x-x_\mu)/|x-x_\mu|$. The source
ball is contained in $A$, so $\rho(A)\ge1-\eta$. For $x\in A$, put
$u=|x-x_\mu|\ge R$. Write $y-x_\mu=ae_x+b$, where $b\perp e_x$.
If $y\in B_\mu$, then $a^2+|b|^2\le s_\mu^2$, and
\begin{equation}
 \frac{u-a}{\sqrt{(u-a)^2+|b|^2}}\ge\sqrt{1-s_\mu^2/u^2}.
\end{equation}
This follows by squaring and using $(ua-s_\mu^2)^2\ge0$.
The target tail contributes at least $-\eta_\mu$; hence
\begin{equation}\label{eqn:mu-grad-geo-lb}
 \nabla F_\mu(x)\cdot e_x\ge S\qquad(x\in A).
\end{equation}
The map $x\mapsto e_x$ is $R^{-1}$-Lipschitz on $A$, although $A$ is not
convex. In fact, for nonzero $u,v$,
\begin{equation}
 \left|\frac{u}{|u|}-\frac{v}{|v|}\right|^2
 =\frac{|u-v|^2-(|u|-|v|)^2}{|u||v|}.
\end{equation}
Pair symmetrization on $A^2$, followed by the bound on the mass of $A^c$,
gives
\begin{equation}
 \int_A\nabla F_\rho(x)\cdot e_x\,d\rho(x)
 \le I_1(\rho)/(2R)+\eta.
\end{equation}
The cone estimate \eqref{eq:farfield-cone-Taylor}, in direction $e_x$, and
\eqref{eqn:beta-choice-2} yield
\begin{equation}\label{eqn:kappa-cone-integral}
 Z(x)\ge\frac{a_d\beta}{2}
                 \bigl(S-2\nabla F_\rho(x)\cdot e_x\bigr)
 \qquad(x\in A).
\end{equation}
Therefore $\int_A Z\,d\rho\ge(a_d\beta/2)L_0>0$.
Jensen's inequality and the rate lower bound on the same set give
\eqref{eq:farfield-static-two-dissipation}. The center-and-tail estimate
from the preceding proof, with both nonpositive self-interaction terms
dropped, gives \eqref{eq:farfield-static-two-energy}.
\end{proof}

We next bound how long the far-field concentration and separation conditions
can persist along the flow. Two elementary estimates suffice.
\begin{lemma}\label{prop:W_1-enlargement-lemma}
Let $\rho,\nu\in\mathcal P_1(\mathbb R^d)$, let $B$ be a Borel set, and let
$B_\delta=\{x:\operatorname{dist}(x,B)<\delta\}$ for $\delta>0$. Then
\begin{equation}\label{eq:farfield-mass-transfer}
 \nu(B_\delta)\ge\rho(B)-\delta^{-1}W_1(\rho,\nu).
\end{equation}
\end{lemma}
\begin{proof}
Every coupling between $\rho$ and $\nu$ transports its mass on
$B\times B_\delta^c$ over a distance at least $\delta$. Thus
$\rho(B)\le\nu(B_\delta)+\delta^{-1}\iint|x-y|\,d\pi(x,y)$.
Taking the infimum over couplings proves the result.
\end{proof}

\begin{lemma}\label{prop:expl-w1-stability}
Under the hypotheses of Theorem~\ref{thm:continuum-wellposedness} with $q=1$,
\begin{equation}\label{eq:farfield-W1-speed}
 W_1(\rho_t,\rho_s)\le L^\star|t-s|,\qquad
 L^\star=2+2\|\lambda\|_\infty\mathcal M_2(\kappa).
\end{equation}
\end{lemma}
\begin{proof}
The force bound gives $|v_\rho|\le2$ and
$\zeta_\rho(x,x+z)\le2|z|$. For a smooth compactly supported
$1$-Lipschitz test function, the weak equation gives
\begin{equation}
 \left|\frac{d}{dt}\int\varphi\,d\rho_t\right|
 \le2+2\|\lambda\|_\infty\int|z|^2\kappa(z)\,dz.
\end{equation}
Integrate between $s$ and $t$. Cutoff and mollification extend this estimate
to all $1$-Lipschitz tests, using the finite first moments and finite-horizon
bounds of the solution. Kantorovich--Rubinstein duality gives
\eqref{eq:farfield-W1-speed}.
\end{proof}

\paragraph{Bounds uniform in a preliminary bandwidth interval.}
The bandwidth must not be chosen from estimates whose dependence on that
same bandwidth is unspecified. Fix $T>0$, the initial and target densities,
and the capped rate. Write $L=\|\lambda\|_\infty$ and set
\begin{equation}\label{eq:farfield-preliminary-bandwidth}
 b=\min\{1,(1+L)^{-1}\}.
\end{equation}
For $0<\beta\le b$, both $L\beta$ and $L\beta^2$ are at most one.
\begin{lemma}\label{lem:farfield-uniform-bandwidth}
There are finite constants $\overline M_{r,T}$ and $\overline H_T$,
determined before $\beta\in(0,b]$ is chosen, such that every corresponding
solution satisfies
\begin{equation}\label{eq:farfield-uniform-bounds}
 \sup_{t\le T}\mathcal M_r(\rho_t)\le\overline M_{r,T},\qquad
 \mathop{\rm ess\,sup}_{t\le T}\|D^2f_{\rho_t}\|_\infty\le\overline H_T.
\end{equation}
One can take
\begin{equation}\label{eq:farfield-uniform-moment}
 \overline M_{r,T}=\left(\widetilde{\mathcal M}_r(\rho_0)^{1/r}
                             +C_{d,r}T\right)^r.
\end{equation}
If $m_1=\int|z|\gamma_d(z)\,dz$ and
$V_\mu=\|(-\Delta K)*\mu\|_\infty$, then
\begin{equation}\label{eq:farfield-uniform-density}
 \sup_{t\le T}\|\rho_t\|_p\le P_T
 :=\|\rho_0\|_p\exp\bigl((V_\mu+2m_1)T\bigr).
\end{equation}
For $d\ge2$ one may choose
$\overline H_T=C_{d,p}(2+\|\mu\|_p+P_T)$; for $d=1$ take
$\overline H_T=2(\|\mu\|_\infty+P_T)$. In addition,
\begin{equation}\label{eq:farfield-uniform-speed}
 W_1(\rho_t,\rho_s)\le\overline L|t-s|,\qquad
 \overline L=2+2d,\qquad 0\le s,t\le T.
\end{equation}
\end{lemma}
\begin{proof}
For $q=1$ the bounds $|v_\rho|\le2$ and $\zeta_\rho(x,x+z)\le2|z|$
do not depend on the density or its moments. Let
$m(t)=\int\langle x\rangle^r\,d\rho_t(x)\ge1$. The mean-value theorem gives
\begin{equation}
 |\langle x+z\rangle^r-\langle x\rangle^r|
 \le C_r|z|\bigl(\langle x\rangle^{r-1}+\langle z\rangle^{r-1}\bigr).
\end{equation}
The moment testing argument from Section~\ref{sec:continuum}, first with
cutoffs and then by localization, yields
\begin{equation}
 m'(t)\le C_r\int\langle x\rangle^{r-1}\,d\rho_t(x)
 +C_rL\left(d\beta^2\int\langle x\rangle^{r-1}\,d\rho_t(x)
       +\int|z|^2\langle z\rangle^{r-1}\kappa_\beta(z)\,dz\right).
\end{equation}
For $\beta\le1$, the last integral is bounded by
$\beta^2\int|z|^2\langle z\rangle^{r-1}\gamma_d(z)\,dz$.
H\"older's inequality, $L\beta^2\le1$, and $m\ge1$ give
$m'\le C_{d,r}m^{1-1/r}$. Integrating the inequality for $m^{1/r}$ proves
\eqref{eq:farfield-uniform-moment}.

For finite $p$, the density estimate used in Section~\ref{sec:continuum}
discards the nonpositive self-interaction divergence and jump outflow, and
bounds the positive inflow by its $L^p$ operator norm. It gives
\begin{equation}
 \frac{d}{dt}\|\rho_t\|_p^p
 \le\bigl((p-1)V_\mu+2pL\mathcal M_1(\kappa_\beta)\bigr)\|\rho_t\|_p^p.
\end{equation}
Since $\mathcal M_1(\kappa_\beta)=\beta m_1$ and $L\beta\le1$,
Gronwall's inequality gives \eqref{eq:farfield-uniform-density}.
The maximum estimate of Theorem~\ref{thm:continuum-wellposedness} gives the
same bound when $p=\infty$. These estimates hold first for the regularized
solutions, with constants independent of the regularization, and pass to
the weak solution.

For $d\ge2$, the Hessian kernel is bounded by $C_d|x|^{-1}$ away from the
origin and has no distributional point mass. A near--far decomposition
and H\"older's inequality give the asserted Hessian bound because $p_*<d$;
they also give $V_\mu<\infty$. For $d=1$, $K''=-2\delta_0$, so
$f_{\rho_t}''=-2(\rho_t-\mu)$. Finally,
\eqref{eq:farfield-W1-speed} and
$L\mathcal M_2(\kappa_\beta)=Ld\beta^2\le d$ give
\eqref{eq:farfield-uniform-speed}.
\end{proof}

\paragraph{The order of parameter choices.}
For a density $\nu$, write
\begin{equation}
 s_\nu(\alpha)=\inf\{s>0:\nu(B_s(x_\nu))\ge1-\alpha\},\qquad0<\alpha<1.
\end{equation}
Its radial distribution function is continuous, so this radius is positive,
finite, and attains the required mass. Fix
$0<\eta_\mu<\sqrt3/(2+\sqrt3)$ and put
$a_0=\frac{\sqrt3}{2}(1-\eta_\mu)-\eta_\mu>0$.
Choose $0<\eta\le a_0/(8+a_0)$, set
$s_0=s_{\rho_0}(\eta/2)$ and $s_\mu=s_\mu(\eta_\mu)$, and choose
\begin{subequations}\label{eq:farfield-ordered-parameters}
\begin{align}
 \varepsilon_\rho&=\eta^{1/r_*}\overline M_{r,T}^{1/r},
 &\varepsilon_\mu&=\eta_\mu^{1/r_*}\mathcal M_r(\mu)^{1/r},\\
 \theta_0&\ge\max\{8\eta_\mu,16\varepsilon_\rho/s_0\},
 &\widetilde R&\ge\max\{56/\theta_0,4\theta_0/(a_0(1-\eta))\},\\
 0<\beta&\le\min\left\{b,
 \frac{a_d\min\{\theta_0,a_0\}}{16d\max\{1,\overline H_T\}}\right\}.
\end{align}
\end{subequations}
All bounds used to choose $\beta$ have already been determined. This
specifies a positive bandwidth; the initial separation conditions below
remain conditions to check on the prescribed source and target.

For the resulting solution, let
\begin{equation}\label{eq:farfield-moving-balls}
 x_t=x_{\rho_t},\qquad s_t=\max\{s_{\rho_t}(\eta),s_0\},\qquad
 B_t=B_{s_t}(x_t),\qquad R_t=\operatorname{dist}(B_t,B_\mu).
\end{equation}
In particular $s_t|_{t=0}=s_0$ and
$R_0=\operatorname{dist}(B_{s_0}(x_0),B_\mu)$.
With the convention $\inf\varnothing=\infty$, define the deterministic exit times
\begin{align}
 \tau_1&=\min\{T,\inf\{t\in[0,T]:R_t\le\widetilde R s_t
                                    \text{ or }R_t\le2s_\mu\}\},
 \label{eq:farfield-first-exit}\\
 \tau_2&=\min\{\tau_1,\inf\{t\in[0,T]:
                     R_t\le s_t+s_\mu+\varepsilon_\rho+\varepsilon_\mu\}\}.
 \label{eq:farfield-second-exit}
\end{align}

\begin{proposition}[Far-field inequality and stopped decay]\label{thm:farfield-pl}
Assume \eqref{eq:farfield-ordered-parameters} and
$R_0\ge\max\{4s_\mu,\widetilde R s_0\}$. Let
\begin{equation}\label{eq:farfield-cap-constant}
 c_\lambda=\min\left\{\frac{(1-\eta_\mu)^2}{4},
 \left(\frac{(1-\eta-\eta_\mu)D_0}{D_0+\overline LT}\right)^2\right\}>0,
 \qquad D_0=W_1(\rho_0,\mu),
\end{equation}
and set
\begin{equation}\label{eq:farfield-PL-constant}
 C_{\rm PL}=c_\lambda a_d^2
   \min\{\theta_0^2/1024,a_0^2(1-\eta)^2/16\}>0.
\end{equation}
For a.e. $0\le t<\tau_1$,
\begin{equation}\label{eq:farfield-stopped-PL}
 \mathcal D(\rho_t)\ge C_{\rm PL}\beta^2R_t^2
 \frac{\mathcal E_1(\rho_t-\mu)}
      {2(R_t+s_t+s_\mu+\varepsilon_\rho+\varepsilon_\mu)}.
\end{equation}
For every $0\le t\le\tau_2$,
\begin{equation}\label{eq:farfield-stopped-decay}
 \mathcal E_1(\rho_t-\mu)
 \le\exp\left(-\frac{C_{\rm PL}\beta^2}{4}\int_0^tR_u\,du\right)
 \mathcal E_1(\rho_0-\mu)
 \le e^{-C_{\rm PL}\beta^2s_\mu t/2}\mathcal E_1(\rho_0-\mu).
\end{equation}
The exit times satisfy
\begin{equation}\label{eq:farfield-first-exit-lower}
 \tau_1\ge\min\left\{T,
 \frac{R_0-\widetilde R s_0}{(2+2/\eta+\widetilde R(1+2/\eta))\overline L},
 \frac{R_0-2s_\mu}{(2+2/\eta)\overline L}\right\},
\end{equation}
\begin{equation}\label{eq:farfield-second-exit-lower}
 \tau_2\ge\min\left\{\tau_1,
 \frac{R_0-s_0-s_\mu-\varepsilon_\rho-\varepsilon_\mu}
      {(3+4/\eta)\overline L}\right\}.
\end{equation}
In particular, the right-hand sides guarantee a positive interval if
\begin{equation}\label{eq:farfield-strict-margins}
 R_0>\max\{\widetilde R s_0,2s_\mu,s_0+s_\mu+\varepsilon_\rho+\varepsilon_\mu\}.
\end{equation}
\end{proposition}
\begin{proof}
For $t<\tau_1$, every coupling of $\rho_t$ and $\mu$ assigns mass at least
$1-\eta-\eta_\mu$ to $B_t\times B_\mu$. Thus
\begin{equation}\label{eqn:W_1-rho_0-cap-est}
 (1-\eta-\eta_\mu)R_t\le W_1(\rho_t,\mu)\le D_0+\overline LT.
\end{equation}
If $|x-x_\mu|\ge R_t>2s_\mu$, then
$F_\mu(x)\ge(1-\eta_\mu)(R_t-s_\mu)\ge(1-\eta_\mu)R_t/2$.
By the two terms in the minimum defining the fixed spatial rate,
\begin{equation}\label{eq:farfield-cap-lower}
 \lambda(x)\ge c_\lambda R_t^2\qquad(|x-x_\mu|\ge R_t).
\end{equation}
This lower bound also holds on $B_t$.

If $I_1(\rho_t)\ge\theta_0s_t$, then
\begin{equation}\label{eqn:pl-first-case}
 \eta^{1/r_*}\mathcal M_r(\rho_t)^{1/r}
 \le\varepsilon_\rho\le\theta_0s_0/16\le\theta_0s_t/16.
\end{equation}
All assumptions of \Cref{prop:far-field-pl-prop} hold for
a.e. such $t$, and hence
\begin{equation}\label{eqn:pl-1-appl}
 \mathcal D(\rho_t)\ge
       c_\lambda a_d^2\theta_0^2\beta^2R_t^2/1024.
\end{equation}
If $I_1(\rho_t)<\theta_0s_t$, then $R_t>2s_\mu$ gives $S\ge a_0$, and
\begin{equation}
 I_1(\rho_t)/R_t\le\theta_0/\widetilde R\le a_0(1-\eta)/4,
 \qquad 2\eta\le a_0(1-\eta)/4.
\end{equation}
Therefore $L_0\ge a_0(1-\eta)/2$. The bandwidth bound is sufficient for
\Cref{prop:spr-pl-ineq}, which gives
\begin{equation}\label{eqn:pl-2-appl}
 \mathcal D(\rho_t)\ge
       c_\lambda a_d^2a_0^2(1-\eta)^2\beta^2R_t^2/16.
\end{equation}
In both cases, $\mathcal D(\rho_t)\ge C_{\rm PL}\beta^2R_t^2$.
The corresponding energy upper bounds are at most
$2(R_t+s_t+s_\mu+\varepsilon_\rho+\varepsilon_\mu)$, proving
\eqref{eq:farfield-stopped-PL}.

For $t<\tau_2$, the denominator in \eqref{eq:farfield-stopped-PL} is at
most $4R_t$. Proposition~\ref{prop:energy-dissipation} therefore implies,
for a.e. such $t$,
\begin{equation}
 \frac{d}{dt}\mathcal E_1(\rho_t-\mu)
 \le-\frac{C_{\rm PL}\beta^2}{4}R_t\mathcal E_1(\rho_t-\mu).
\end{equation}
The functions $s_t$ and $R_t$ are Borel measurable: for any fixed radius,
the mass of the ball centered at $x_t$ is continuous in $t$, because
$\rho_t$ is narrowly continuous, $x_t$ is continuous, and sphere boundaries
have zero mass. The radius can then be tested against rational radii.
The finite-horizon moment bound makes the radii and $R_t$ bounded.
The integrating-factor argument proves the first inequality in
\eqref{eq:farfield-stopped-decay} for $t<\tau_2$; continuity of the energy
and boundedness of $R_t$ pass it to $t=\tau_2$. Since $R_t>2s_\mu$ before
the exit, the second inequality follows. No continuity of the
concentration quantile at the exit time is assumed.

For the exit-time bounds, $|x_t-x_0|\le\overline Lt$.
Apply Lemma~\ref{prop:W_1-enlargement-lemma} to the initial ball,
which contains mass $1-\eta/2$, with enlargement
$2W_1(\rho_t,\rho_0)/\eta$, then shift the center to $x_t$.
When that distance is zero the same conclusion is immediate. We obtain
\begin{equation}\label{eq:farfield-radius-motion}
 s_t-s_0\le(1+2/\eta)\overline Lt,\qquad
 R_t\ge R_0-(2+2/\eta)\overline Lt.
\end{equation}
These inequalities show that none of the exit conditions can hold at a
time strictly below the corresponding right-hand side in
\eqref{eq:farfield-first-exit-lower}--\eqref{eq:farfield-second-exit-lower}.
Taking infima proves those bounds. Their numerators are strictly positive
under \eqref{eq:farfield-strict-margins}; if a numerator is nonpositive,
that particular lower bound is merely trivial.
\end{proof}

\begin{remark}[Nonempty far-field parameter class]\label{rem:farfield-nonempty}
The conditions above can be met by sufficiently separated translates.
Fix smooth compactly supported probability densities $\vartheta,\mu$
with zero barycenters, and let $\rho_0^R(x)=\vartheta(x-Re_1)$.
For fixed concentration fractions, $s_0$ is independent of $R$ and the
initial ball distance is $R-s_0-s_\mu$. The moment bound
\eqref{eq:farfield-uniform-moment} gives
$\overline M_{r,T}^{1/r}\le R+C$, with $C$ independent of $R$.
Choose a fixed $\eta>0$ small enough that
\begin{equation}
 \eta\le a_0/(8+a_0),\qquad
 \frac{64\eta^{1/r_*}}{a_0(1-\eta)}<\frac14.
\end{equation}
Taking
$\theta_0=8\eta_\mu+16\eta^{1/r_*}(R+C)/s_0$ and
$\widetilde R=\max\{56/\theta_0,4\theta_0/(a_0(1-\eta))\}$ gives
\begin{equation}
 \limsup_{R\to\infty}\frac{\widetilde R s_0}{R}
 \le\frac{64\eta^{1/r_*}}{a_0(1-\eta)}<\frac14,
 \qquad
 \limsup_{R\to\infty}\frac{\varepsilon_\rho}{R}\le\eta^{1/r_*}<1.
\end{equation}
Thus the strict margins and $R_0>4s_\mu$ hold for large $R$, and
\eqref{eq:farfield-ordered-parameters} permits a positive bandwidth.
Both the bandwidth and the decay coefficient may depend on $R$;
no uniform rate as $R\to\infty$ is asserted.
\end{remark}

\section{Particle approximation and mean-field limit}
\label{sec:particle-approximation}
The goal of this section is to construct a global collision-free particle
system and prove its quantitative approximation of the continuum evolution.
Sections~\ref{subsec:particle-construction}, \ref{subsec:particle-mean-field},
\ref{subsec:particle-initial-matching}, and
\ref{subsec:particle-long-time} treat construction and noncollision,
quantitative mean-field comparison, initial matching and propagation of chaos,
and long-time target approximation, respectively.

Throughout this section, we assume the hypotheses of
Theorem~\ref{thm:continuum-wellposedness} on
\(d,q,p,r,\mu,\kappa,\lambda\), and fix \(N\ge2\).
Let $\mathcal N_1,\ldots,\mathcal N_N$ be independent Poisson random measures on
$[0,\infty)\times\mathbb R^d\times(0,\infty)$ with intensity
$dt\,\kappa(z)\,dz\,du$. The initial configuration is independent of this
entire family. Coordinates of the initial configuration need not be independent.
The auxiliary coordinate $u$ implements the state-dependent acceptance rule.
An atom $(s,z,u)$ of $\mathcal N_i$ is an \emph{accepted mark} if
$u\le\lambda(X_i(s-))\zeta_i^N(X(s-),z)$, with $\zeta_i^N$ defined in
\eqref{eq:particle-potential-acceptance}. The threshold is evaluated at the
pre-jump configuration, and only accepted marks contribute to the particle path.

\subsection{Construction and global noncollision}
\label{subsec:particle-construction}
We construct the process by following the deterministic particle flow
until an accepted jump occurs, updating one coordinate, and restarting
the flow. This is a piecewise deterministic Markov process in the sense
of \cite{davis1984piecewise}. 

The deterministic flow is noncolliding. The accepted displacement kernel is absolutely continuous, even though acceptance changes its density from the proposal density. We verify below that an accepted mark cannot create a collision.
The moment and intensity estimates below exclude finite-time accumulation of jumps.
In the stopped moment estimate, we include the positive increment of the jump
that first exits the localization set $\{\Psi_r<\ell\}$, even if its
post-jump value exceeds $\ell$; this increment is controlled using the
pre-jump state. See \eqref{eq:particle-stopped-moment-identity}.

Recall \(\rho_X^N\) and \(\mathsf D_N\) from
\eqref{eq:particle-configuration-notation}. For \(X\in\mathsf D_N\), define the diagonal-free transport drift
\begin{equation}
\label{eq:particle-transport-drift}
  b_i^N(X)
  :=-\frac1N\sum_{\substack{1\le j\le N\\j\ne i}}
       \nabla K(x_i-x_j)
     +\nabla K*\mu(x_i),
  \qquad 1\le i\le N.
\end{equation}
The first variation and corrected acceptance factor at the \(i\)-th particle are
\begin{equation}
\label{eq:particle-potential-acceptance}
  f_X^N:=f_{\rho_X^N},
  \qquad
  \zeta_i^N(X,z)
  :=\left[f_X^N(x_i)-f_X^N(x_i+z)-\frac{|z|^q}{N}\right]_+.
\end{equation}
The correction $-|z|^q/N$ inside the positive part accounts for the
quadratic term $|z|^q/N^2$ in the exact one-particle energy increment
\eqref{eq:particle-exact-jump-energy}. It makes every accepted jump decrease
\(\mathcal E_q(\rho_X^N-\mu)\); see
Proposition~\ref{prop:particle-pathwise-energy-dissipation}.
\begin{definition}[Collision-free particle solution]
\label{def:particle-solution}
Let \(T>0\), and let \(X^0\) be a \(\mathsf D_N\)-valued random
variable independent of the family $(\mathcal N_i)_{i=1}^N$.
Let \((\mathcal F_t)_{t\ge0}\) be the right-continuous completion
of the filtration generated by \(X^0\) and the restrictions of
\(\mathcal N_1,\ldots,\mathcal N_N\) to
\([0,t]\times\mathbb R^d\times[0,\infty)\). An \((\mathcal F_t)\)-adapted process
\(X=(X_1,\ldots,X_N)\) is a \emph{collision-free particle solution}
on \([0,T]\) with initial configuration \(X^0\) if, almost surely,
its paths are right-continuous with left limits,
\begin{equation}
\label{eq:particle-solution-noncollision}
  X(t)\in\mathsf D_N\quad(0\le t\le T),\qquad
  X(t-)\in\mathsf D_N\quad(0<t\le T),
\end{equation}
and, for every \(t\in[0,T]\) and \(1\le i\le N\),
\begin{equation}
\label{eq:jump-particle-system}
  X_i(t)
  ={}X_i^0+\int_0^t b_i^N(X(s-))\,ds
  +\int_{(0,t]\times\mathbb R^d\times(0,\infty)}
  z\,\mathbf 1_{\{
    u\le\lambda(X_i(s-))\zeta_i^N(X(s-),z)
  \}}\,\mathcal N_i(ds,dz,du).
\end{equation}
A process on \([0,\infty)\) is a global collision-free particle
solution if its restriction to every interval \([0,T]\) satisfies
these conditions.
\end{definition}
For \(1\le i\le N\), let
\begin{equation}
\label{eq:particle-coordinate-insertion}
  \iota_i(z):=(0,\ldots,0,z,0,\ldots,0)\in(\mathbb R^d)^N,
\end{equation}
where the nonzero entry is in the \(i\)-th component.  For every
configuration in \(\mathsf D_N\), set
\begin{equation}
\label{eq:particle-Lyapunov-function}
  \Psi_r(X):=\frac1N\sum_{i=1}^N\langle x_i\rangle^r, \quad   \mathscr R_N(X)
  :=\sum_{i=1}^N\lambda(x_i)\int_{\R^d}
     \zeta_i^N(X,z)\kappa(z)\,dz.
\end{equation}

\begin{lemma}
\label{lem:particle-Lyapunov-intensity}
There exists \(C<\infty\), independent of \(N\) and
\(X\in\mathsf D_N\), such that
\begin{equation}
\label{eq:particle-positive-variation-Lyapunov-bound}
  \left[
    \sum_{i=1}^N b_i^N(X)\cdot\nabla_i\Psi_r(X)
  \right]_+
  +
  \sum_{i=1}^N\int_{\mathbb R^d}
  \left[\Psi_r(X+\iota_i(z))-\Psi_r(X)\right]_+
  \lambda(x_i)\zeta_i^N(X,z)\kappa(z)\,dz
  \le C\Psi_r(X).
\end{equation}
Moreover,
\begin{equation}
\label{eq:particle-total-jump-intensity-bound}
  \mathscr R_N(X)\le CN\Psi_r(X).
\end{equation}
\end{lemma}

\begin{proof}
By \cite[Lemma~3.2]{rosenzweig2026wassersteingradientflowsmaximum}
for \(0<q<1\), and \(|\nabla K(x)|=q|x|^{q-1}\) for \(1\le q<2\),
we have
\begin{equation}
\label{eq:particle-target-field-growth}
  |(\nabla K*\mu)(x)|
  \le C
  \begin{cases}
    1,&0<q\le1,\\
    1+\langle x\rangle^{q-1},&1<q<2.
  \end{cases}
\end{equation}
Since \(\nabla_i\Psi_r(X)=rN^{-1}\langle x_i\rangle^{r-2}x_i\),
\eqref{eq:particle-transport-drift} yields
\begin{align}
  \sum_{i=1}^N &b_i^N(X)\cdot\nabla_i\Psi_r(X)
  =\nonumber \\
  &-\frac{r}{2N^2}\sum_{i\ne j}
  \left(\langle x_i\rangle^{r-2}x_i
       -\langle x_j\rangle^{r-2}x_j\right)
  \cdot\nabla K(x_i-x_j)
  +\frac rN\sum_{i=1}^N
  \langle x_i\rangle^{r-2}x_i\cdot(\nabla K*\mu)(x_i).\label{eq:particle-transport-Lyapunov-symmetrization}
\end{align}
By the mean-value theorem and the facts that
\(
  |x-y|\le2(\langle x\rangle\vee\langle y\rangle)
\) and \(r\ge q\),
we have, for any $x, y \in \R^d$,
\begin{equation}
\label{eq:particle-pair-moment-inequality}
\begin{aligned}
 &\left|
 \left(\langle x\rangle^{r-2}x-\langle y\rangle^{r-2}y\right)
 \cdot(x-y)|x-y|^{q-2}
 \right|\le C_{q,r}
 \left(\langle x\rangle^r+\langle y\rangle^r\right).
\end{aligned}
\end{equation}
Substituting
\eqref{eq:particle-target-field-growth} and
\eqref{eq:particle-pair-moment-inequality} into
\eqref{eq:particle-transport-Lyapunov-symmetrization} gives
\begin{equation}
\label{eq:particle-transport-Lyapunov-bound}
  \left[
  \sum_{i=1}^N b_i^N(X)\cdot\nabla_i\Psi_r(X)
  \right]_+
  \le C\Psi_r(X).
\end{equation}
Meanwhile, \eqref{eq:moment-weight-increment} and the definition of \(\Psi_r(X)\) give
\begin{equation}
\label{eq:particle-positive-moment-increment}
  \left[\Psi_r(X+\iota_i(z))-\Psi_r(X)\right]_+
  \le\frac{C_r}{N}|z|
  \left(\langle x_i\rangle^{r-1}+\langle z\rangle^{r-1}\right).
\end{equation}
By the definition of \(\zeta_i^N\) and
\eqref{eq:jump-increment-bound}, applied with
\(\nu=\rho_X^N\) and \(\varepsilon=0\),
\(
  \zeta_i^N(X,z)\le A_{\rho_X^N}|z|^{1\wedge q}.
\)
Using this bound in \eqref{eq:particle-positive-moment-increment}
and integrating against \(\kappa\), we obtain
\begin{equation}
\label{eq:particle-positive-jump-Lyapunov-bound}
\begin{aligned}
\sum_{i=1}^N\int_{\mathbb R^d}
  [\Psi_r(X+\iota_i(z))-\Psi_r(X)]_+
  \lambda(x_i)\zeta_i^N(X,z)\kappa(z)\,dz
  \le C\Psi_r(X).
\end{aligned}
\end{equation}
Here we have used that \(\kappa\in\mathcal P_{r+1}\), along with \eqref{eq:acceptance-coefficient} and \eqref{eq:bracket-moment-product-bound}, since
\((q-1)_++r-1\le r\).
Together with \eqref{eq:particle-transport-Lyapunov-bound},
this proves \eqref{eq:particle-positive-variation-Lyapunov-bound}. The same bound on \(\zeta_i^N\) also gives
\begin{equation}
  \mathscr R_N(X)
  \le N\|\lambda\|_{L^\infty}
      A_{\rho_X^N}\mathcal M_{1\wedge q}(\kappa)
  \le CN\left(
      1+\widetilde{\mathcal M}_{(q-1)_+}(\rho_X^N)
    \right)
  \le CN\Psi_r(X),
\end{equation}
where the last inequality follows from
\((q-1)_+\le r\) and
\(\Psi_r(X)=\widetilde{\mathcal M}_r(\rho_X^N)\).

\end{proof}

\begin{proof}[Proof of Theorem~\ref{thm:particle-wellposedness}]
For every \(X_\star\in\mathsf D_N\),
\cite[Theorem~1.8, cases~\textup{(i)}--\textup{(ii)}]
{rosenzweig2026wassersteingradientflowsmaximum} and
\eqref{eq:continuum-admissible-exponent} give a unique global solution of
\begin{equation}
\label{eq:particle-deterministic-flow}
  \dot{\mathsf x}_i(t)=b_i^N(\mathsf x(t)),
  \qquad \mathsf x(0)=X_\star,
  \qquad 1\le i\le N,
\end{equation}
with \(\mathsf x(t)\in\mathsf D_N\) for every \(t\ge0\).

We now interlace this deterministic flow with accepted Poisson marks. Let
\((\mathcal F_t)_{t\ge0}\) be the right-continuous completion of the
filtration generated by \(X^0\) and the restrictions of
\(\mathcal N_1,\ldots,\mathcal N_N\) to
\([0,t]\times\mathbb R^d\times[0,\infty)\). Set
\(\tau_0:=0\) and \(X(0):=X^0\). Suppose that
\(\tau_k\) and \(X(\tau_k)\) have been defined, and let
\(\mathsf x^{(k)}\) solve the differential equation in
\eqref{eq:particle-deterministic-flow} on \([\tau_k,\infty)\) with
\(\mathsf x^{(k)}(\tau_k)=X(\tau_k)\). Let $\tau_{k+1}$ be the time of the first point $(s, z, u)$ of one of the $\mathcal N_i$ with $ s> \tau_k$ that is accepted, i.e. such that 
\begin{equation}
u \leq \lambda(\mathsf x_i^{(k)}(s))\zeta_i^N(\mathsf x^{(k)}(s), z).
\end{equation}
The construction makes $\tau_k$ a stopping time and $X(\tau_k)$ measurable
with respect to $\mathcal F_{\tau_k}$. On every bounded interval after $\tau_k$,
the deterministic trajectory has bounded $\Psi_r$. The bound
$\zeta_i^N(X,z)\le A_{\rho_X^N}|z|^{\min\{1,q\}}$ therefore gives an integrable
envelope for the accepted marks on that interval. Equivalently,
\begin{equation}
 \int_{\tau_k}^t\mathscr R_N(\mathsf x^{(k)}(s))\,ds<\infty
 \qquad(t>\tau_k).
\end{equation}
Conditioning at $\tau_k$ and localizing its initial state if necessary,
the independent-increment property of the Poisson measures shows that
there are only finitely many such marks on each bounded interval.
This assertion concerns the accepted marks, not all points in the
unbounded auxiliary $u$-coordinate. Absolute continuity of the time
intensity excludes simultaneous accepted marks, after localization and
then a countable union over bounded time and state intervals.
On the event \(\tau_{k+1}<\infty\), let \(i_{k+1}\) and \(z_{k+1}\) be the index and displacement of the unique point accepted at time \(\tau_{k+1}\). Then set
\begin{equation}
\label{eq:particle-recursive-state-update}
  X(t):=\mathsf x^{(k)}(t)\quad(\tau_k\le t<\tau_{k+1}),\qquad
  X(\tau_{k+1}):=\mathsf x^{(k)}(\tau_{k+1})+\iota_{i_{k+1}}(z_{k+1})\quad(\tau_{k+1}<\infty).
\end{equation}
If \(\tau_{k+1}=\infty\), set
\(X(t):=\mathsf x^{(k)}(t)\) for \(t\ge\tau_k\) and
\(\tau_j:=\infty\) for every \(j\ge k+1\). Concatenate the deterministic flow \eqref{eq:particle-deterministic-flow} between the jumps \eqref{eq:particle-recursive-state-update} to obtain an adapted c\`adl\`ag solution of \eqref{eq:jump-particle-system} on $[0, \tau_\infty)$, where $\tau_\infty:=\lim_{k\to\infty}\tau_k$.

The recursion remains in the collision complement. To verify this without
identifying the accepted law with the proposal law, let
\begin{equation}
 C_i(X)=\{x_j-x_i:j\ne i\}.
\end{equation}
For each pre-jump state this is a finite set. After localization, the
expected number of accepted marks with $z\in C_i(X(s-))$ equals
\begin{equation}
 \mathbb E\left[\int_0^t\int_{\mathbb R^d}
 \mathbf1_{\{z\in C_i(X(s-))\}}
 \lambda(X_i(s-))\zeta_i^N(X(s-),z)\kappa(z)\,dz\,ds\right]=0.
\end{equation}
The indicator is predictable in the state variable, and the spatial
integral vanishes because a finite set has Lebesgue measure zero.
The compensation formula first applies to nonnegative integrands, so no
unproved integrability assertion is needed here. Removing the localization
shows that accepted jumps do not create collisions. Induction, together
with the deterministic noncollision result, gives both $X(t)$ and $X(t-)$
in $\mathsf D_N$ before $\tau_\infty$. It remains to exclude finite-time
accumulation of accepted jumps.

For \(\ell\ge1\), set
\begin{equation}
\label{eq:particle-localization-stopping-time}
  \theta_\ell:=\ell\wedge\tau_\ell
  \wedge\inf\{t<\tau_\infty:\Psi_r(X(t))\ge\ell\},
\end{equation}
where \(\inf\varnothing=\infty\). Every compact subinterval of
\([0,\tau_\infty)\) contains finitely many times
\(\tau_k\), and \(\Psi_r(X)\) is bounded on each such interval. Conclude that $\theta_\ell \to \tau_\infty$ almost surely. 

Since \(\theta_\ell\le\tau_\ell\), the path has only finitely many jumps. Applying the chain rule along the deterministic pieces and summing the jump increments, then keeping only the positive contributions, gives
\begin{align}&
    \sup_{0 \leq v \leq t}\Psi_r(X(v\wedge\theta_\ell))-\Psi_r(X^0)\leq\int_0^{t\wedge\theta_\ell}
  \left[\sum_{i=1}^N b_i^N(X(s-))\cdot\nabla_i\Psi_r(X(s-))\right]_+\,ds \nonumber \\& +\sum_{i=1}^N
  \int_{(0,t\wedge\theta_\ell]\times\mathbb R^d\times(0,\infty)}
  \left[\Psi_r(X(s-)+\iota_i(z))-\Psi_r(X(s-))\right]_+\nonumber \\& \times
  \mathbf1_{\{u\le
  \lambda(X_i(s-))\zeta_i^N(X(s-),z)\}}
  \,\mathcal N_i(ds,dz,du).\label{eq:particle-stopped-moment-identity}
\end{align}
The jump that exits the set $\{\Psi_r<\ell\}$ is included in the
stopped sum. Its post-jump value need not be bounded by $\ell$.
The indicator $\mathbf1_{\{s\le\theta_\ell\}}$ is predictable, and before
this jump the state satisfies $\Psi_r(X(s-))\le\ell$ whenever
$\theta_\ell>0$. Tonelli's theorem and the compensation formula for
nonnegative integrands, followed by
\eqref{eq:particle-positive-variation-Lyapunov-bound}, first give the
finite bound $\mathbb E[\Psi_r(X^0)]+C\ell t$ for the expectation of the
stopped running supremum. We may therefore take expectations in
\eqref{eq:particle-stopped-moment-identity} and then use the sharper
bound by the stopped running supremum itself to obtain
\begin{equation}
\label{eq:particle-stopped-running-moment-bound}
  \mathbb E\left[\sup_{0\le u\le t}
    \Psi_r(X(u\wedge\theta_\ell))\right]
  \le \mathbb E[\Psi_r(X^0)]
  +C\int_0^t
    \mathbb E\left[\sup_{0\le u\le s}
      \Psi_r(X(u\wedge\theta_\ell))\right]ds.
\end{equation}
Gronwall's inequality therefore gives, uniformly in \(\ell\),
\begin{equation}
\label{eq:particle-stopped-moment-bound}
  \mathbb E\left[\sup_{0\le t\le T}
    \Psi_r(X(t\wedge\theta_\ell))\right]
  \le e^{CT}\mathbb E[\Psi_r(X^0)].
\end{equation}

To exclude accumulation of jumps, let
\(\mathcal J_{T,\ell}:=\#\{k\ge1:\tau_k\le T\wedge\theta_\ell\}\).
The compensation formula,
\eqref{eq:particle-total-jump-intensity-bound}, and
\eqref{eq:particle-stopped-moment-bound} yield
\begin{equation}
\begin{aligned}
  \mathbb E[\mathcal J_{T,\ell}]
  =\mathbb E\left[\int_0^{T\wedge\theta_\ell}
      \mathscr R_N(X(s-))\,ds\right]\leq CN\int_0^T
      \mathbb E\left[\sup_{0\le u\le s}
        \Psi_r(X(u\wedge\theta_\ell))\right]ds
  \le CNTe^{CT}\mathbb E[\Psi_r(X^0)].
\end{aligned}
\end{equation}
Since \(\theta_\ell\uparrow\tau_\infty\), monotone convergence gives
the same bound for
\(\mathcal J_T:=\lim_{\ell\to\infty}\mathcal J_{T,\ell}\);
in particular, \(\mathcal J_T<\infty\) almost surely.
On \(\{\tau_\infty\le T\}\), however, every jump is eventually
counted, so \(\mathcal J_T=\infty\).
Thus \(\mathbb P(\tau_\infty\le T)=0\) for every \(T<\infty\),
and hence \(\tau_\infty=\infty\) almost surely.
This proves \eqref{eq:particle-solution-noncollision}.
Finally, \(\theta_\ell\uparrow\infty\), so Fatou's lemma applied above gives
\eqref{eq:particle-uniform-moment-bound} with \(C_T=e^{CT}\).

For uniqueness, two solutions with the same initial configuration and the same Poisson
random measures agree on \([0,\tau_1)\) by uniqueness of the deterministic flow 
\eqref{eq:particle-deterministic-flow}. If they agree on
\([0,\tau_k]\), then they must have the same 
\(\tau_{k+1}\), \(i_{k+1}\), and \(z_{k+1}\), so they agree on $[0, \tau_{k+1}]$, too. Equation
\eqref{eq:particle-finite-jump-count} therefore gives equality on every
bounded time interval.
\end{proof}

For completeness, moving one particle gives the exact quadratic expansion
\begin{equation}\label{eq:particle-exact-jump-energy}
 \mathcal E_q(\rho_{X+\iota_i(z)}^N-\mu)-\mathcal E_q(\rho_X^N-\mu)
 =\frac1N\bigl(f_X^N(x_i+z)-f_X^N(x_i)\bigr)+\frac{|z|^q}{N^2}.
\end{equation}
Indeed, the linear term pairs $f_X^N$ with
$N^{-1}(\delta_{x_i+z}-\delta_{x_i})$, and the quadratic term is
$|z|^q/N^2$ because $K(0)=0$ and $K(z)=-|z|^q$.
At every accepted mark, the expression in the positive part defining
$\zeta_i^N$ is positive, so the increment equals $-\zeta_i^N(X,z)/N$.
Between jumps, the transport derivative is
$-N^{-1}\sum_i|b_i^N|^2$. Integrating between jumps and summing the
finitely many jump increments therefore gives, almost surely for all
$0\le s\le t<\infty$,
\begin{equation}
\label{eq:particle-pathwise-energy-dissipation}
  \mathcal E_q(\rho_{X(t)}^N-\mu)
  +\frac1N\int_s^t\sum_{i=1}^N|b_i^N(X(u))|^2\,du
  +\frac1N\sum_{s<\tau\le t}
    \zeta_{i_\tau}^N(X(\tau-),z_\tau)
  =\mathcal E_q(\rho_{X(s)}^N-\mu),
\end{equation}
where \((i_\tau,z_\tau)\) denotes the label and displacement of each
accepted jump. As a consequence:
\begin{proposition}[Monotonicity of the particle MMD]
\label{prop:particle-pathwise-energy-dissipation}
Let \(X\) be the solution of
Theorem~\ref{thm:particle-wellposedness}. Almost surely,
\(t\mapsto\operatorname{MMD}_q(\rho_{X(t)}^N,\mu)\) is non-increasing
on \([0,\infty)\).
\end{proposition}

Under the hypotheses of Theorem~\ref{thm:particle-wellposedness}, the
particle process is determined by the flow of $b^N$, the total jump rate
$\mathscr R_N$, and the post-jump kernel below. These are the local
characteristics of a piecewise deterministic Markov process
\cite{davis1984piecewise}. When $\mathscr R_N(X)>0$, the post-jump kernel is
\begin{equation}
 Q_N(X,A)=\frac1{\mathscr R_N(X)}\sum_i\int
 \mathbf1_A(X+\iota_i(z))\lambda(x_i)\zeta_i^N(X,z)\kappa(z)\,dz.
\end{equation}
When the rate is zero, the choice of $Q_N$ is immaterial. Starting from a
state $X$, the probability of no jump up to time $t$ is
$\exp(-\int_0^t\mathscr R_N(\mathsf x(s;X))\,ds)$.
In particular, for smooth bounded test functions $\varphi$ with bounded
first derivatives, the local generator is
\begin{equation}
 \mathcal L_N\varphi(X)=\sum_i b_i^N(X)\cdot\nabla_i\varphi(X)
 +\sum_i\int\bigl(\varphi(X+\iota_i(z))-\varphi(X)\bigr)
       \lambda(x_i)\zeta_i^N(X,z)\kappa(z)\,dz.
\end{equation}
The integral is finite at each collision-free configuration by
Lemma~\ref{lem:particle-Lyapunov-intensity}. These formulas specify the
meaning of the piecewise-deterministic terminology for the present process.

\subsection{Mean-field convergence}
\label{subsec:particle-mean-field}
Throughout this subsection, let $(\rho_t)$ be the solution of
Theorem~\ref{thm:continuum-wellposedness} from the fixed reference density
$\rho_0$, and let $X$ be the particle solution of
Theorem~\ref{thm:particle-wellposedness}. The initial interacting law may
be arbitrary subject to the initial noncollision, finite expected average
$r$-moment, and noise-independence hypotheses of that theorem.
All continuum data, including $\lambda$, are fixed independently of this
law. The moment exponent is $r\ge\max\{1,q\}$; only the long-time
consequence additionally assumes $r\ge2$. Constants $C_T$ may depend on
the fixed continuum and model data and on $T$, but not on $N$ or the
joint initial law used in the comparison.
Recall \(s_q=(d+q)/2\), and write
\(M_{r,T}:=\sup_{0\le t\le T}\mathcal M_r(\rho_t)\).

The following quantity collects the finite-horizon velocity bounds used
in the transport commutator estimate and the acceptance envelope for the
comparison processes:
\begin{equation}
\label{eq:particle-continuum-control}
  \Lambda_T:=1+\mathop{\rm ess\,sup}_{0\le t\le T}
  \left(
  \|v_{\rho_t}\|_{W^{1,\infty}}
  +\mathbf1_{\{d+q>2\}}
  \bigl\||\nabla|^{s_q}v_{\rho_t}\bigr\|_{L^{2d/(d+q-2)}}
  \right).
\end{equation}
Indeed, \eqref{eq:weak-solution-velocity-class} gives the first norm.
When \(d+q>2\), the Hardy-Littlewood-Sobolev inequality, as applied in
\cite[Lemma~4.16, equations~(4.188)--(4.191)]
{rosenzweig2026wassersteingradientflowsmaximum}, gives
\begin{equation}
\label{eq:particle-velocity-fractional-bound}
  \bigl\||\nabla|^{s_q}v_{\rho_t}\bigr\|_{L^{2p_c}}
  \le C_{d,q}\|\rho_t-\mu\|_{L^{p_c}}.
\end{equation}
Since \(1<p_c<p\le\infty\), interpolation between \(L^1\) and
\(L^p\), together with \eqref{eq:limit-apriori-bounds} and
\eqref{eq:infinite-p-weak-solution-class}, bounds the right-hand side
uniformly on \([0,T]\). At \((d,q)=(1,1)\), the second norm in
\eqref{eq:particle-continuum-control} is absent.

Define
\begin{equation}
\label{eq:particle-mean-field-discrepancy}
  \vartheta_t^N:=\rho_t^N-\rho_t,
  \qquad g_t^N:=K*\vartheta_t^N,
\end{equation}
and
\begin{equation}
\label{eq:particle-modulated-energy}
  F_N(t):=\mathcal E_q(\vartheta_t^N)
  =\operatorname{MMD}_q^2(\rho_t^N,\rho_t).
\end{equation}
Here, the modulated energy is simply the MMD discrepancy between the
empirical measure and the continuum solution. Using this quadratic
quantity preserves cancellations between the two evolutions that would
be lost by estimating each singular pair force separately. This is the
modulated-energy method of \cite{serfaty2020meanfield}; see also
\cite[Sections~5--7]{rosenzweigserfaty2023global}
and \cite[Theorem~1.5 and Section~6]{rosenzweigserfaty2026sharp}.
For \(K=-|\cdot|^q\), the estimate
\eqref{eq:particle-transport-commutator-bound} follows from
\cite[Proposition~4.15]{rosenzweig2026wassersteingradientflowsmaximum};
see also \cite[Remark~1.4]{hesschildsrosenzweigserfaty2025sharp}
for the relation with Riesz modulated energies.

The additional difficulty here is the evolution caused by the jumps.
Proposition~\ref{prop:particle-jump-obstruction} gives smooth data and
pairwise distinct empirical configurations for \((d,q)=(3,1)\) such that
\begin{equation}
\label{eq:particle-false-energy-only-jump-bound}
  \frac{\mathcal J_0^N}{F_N(0)+N^{-1}}\to\infty,
\end{equation}
where \(\mathcal J_t^N\) is defined in \eqref{eq:particle-jump-drift}.
Thus, the jump term cannot, in general, be bounded by
\(C(F_N(t)+N^{-1})\) uniformly over particle configurations.
Lemma~\ref{lem:particle-critical-product} instead yields a bound involving both
\(F_N(t)\) and the bounded-Lipschitz distance between
\(\rho_t^N\) and \(\rho_t\).
To control the latter, we couple each particle
\(X_i\) to a comparison process \(Y_i\) with law \(\rho_t\), using
the same Poisson measure \(\mathcal N_i\) to generate their jump
proposals. Whenever both processes accept a proposal, they move
by the same displacement, leaving their separation unchanged.
The processes \(Y_i\) are independent across labels, so the
bounded-Lipschitz discrepancy splits into a coupling error
between \(X_i\) and \(Y_i\) and a sampling error for the empirical
measure of the \(Y_i\). The dissipative terms in the modulated-energy
inequality absorb the velocity and acceptance errors arising in
the coupling estimate. Combining the two inequalities with the sampling bound then allows us to apply Gronwall's inequality.

We first record two consequences of \eqref{eq:energy}. Since \(  \mathcal E_q(\delta_x-\delta_y)=|x-y|^q\),
Cauchy-Schwarz for the bilinear form associated with \(\mathcal E_q\)
gives
\begin{equation}
\label{eq:particle-discrepancy-potential-Holder}
  |g_t^N(x)-g_t^N(y)|
  \le2F_N(t)^{1/2}|x-y|^{q/2}.
\end{equation}
Moreover, \(K\le0\), \(r\ge q\), and
\eqref{eq:particle-uniform-moment-bound} imply
\begin{equation}
  0\le F_N(t)
  \le\iint_{\R^d\times\R^d}|x-y|^q\,d\rho_t^N(x)\,d\rho_t(y)
   \le C_q\bigl(\Psi_r(X(t))+1+M_{r,T}\bigr)
\end{equation}
and hence
\begin{equation}
\label{eq:particle-energy-moment-bound}
  \mathbb E\left[\sup_{0\le t\le T}F_N(t)\right]
  \le C_T\bigl(1+\mathbb E[\Psi_r(X^0)]\bigr)<\infty.
\end{equation}
We next define the terms that occur in the evolution of \(F_N\).  At the
particle locations, set
\begin{equation}
\label{eq:particle-error-field}
  \mathsf e_t^N(X_i(t))
  :=\frac1N\sum_{j\ne i}\nabla K(X_i(t)-X_j(t))
    -(\nabla K*\rho_t)(X_i(t)).
\end{equation}
Then \eqref{eq:particle-transport-drift} implies
\begin{equation}
\label{eq:particle-drift-error-decomposition}
  b_i^N(X(t))=v_{\rho_t}(X_i(t))-\mathsf e_t^N(X_i(t)).
\end{equation}
Define
\begin{align}
  \mathcal I_t^N
  &:=\frac1N\sum_{i=1}^N|\mathsf e_t^N(X_i(t))|^2,
  \label{eq:particle-transport-dissipation}\\
  \mathcal C_t^N
  &:=\frac12\iint_{x\ne y}
  (v_{\rho_t}(x)-v_{\rho_t}(y))\cdot\nabla K(x-y)
  \,d\vartheta_t^N(x)\,d\vartheta_t^N(y).
  \label{eq:particle-transport-commutator}
\end{align}
By
\eqref{eq:particle-continuum-control},
\eqref{eq:weak-solution-common-class}, and
\eqref{eq:particle-solution-noncollision},
\cite[Proposition~4.15]{rosenzweig2026wassersteingradientflowsmaximum}
applies to \(v_{\rho_t}\) and \(\rho_t^N-\rho_t\), and yields
\begin{equation}
\label{eq:particle-transport-commutator-bound}
  |\mathcal C_t^N|\le C_{d,q}\Lambda_TF_N(t)
  \qquad\text{for a.e. }t\in[0,T].
\end{equation}

For the jump terms, write the continuum acceptance and the corrected
acceptance from \eqref{eq:particle-potential-acceptance} as
\begin{equation}
\label{eq:particle-continuum-and-empirical-acceptance}
  \zeta_t(x,z):=\zeta_{\rho_t}(x,x+z),
  \qquad
  \zeta_t^N(x,z)
  :=\left[f_{\rho_t^N}(x)-f_{\rho_t^N}(x+z)-\frac{|z|^q}{N}\right]_+,
\end{equation}
so that \(\zeta_t^N(X_i(t),z)=\zeta_i^N(X(t),z)\), and define
\begin{align}
  \mathcal J_t^N
  &:=\int_{\R^d}\int_{\R^d}
  (g_t^N(x+z)-g_t^N(x))\lambda(x)
  \bigl(\zeta_t^N(x,z)\,d\rho_t^N(x)
       -\zeta_t(x,z)\,d\rho_t(x)\bigr)\kappa(z)\,dz,
  \label{eq:particle-jump-drift}\\
  \mathcal Q_t^N
  &:=\frac1N\int_{\R^d}\int_{\R^d}
  \lambda(x)|z|^q\zeta_t^N(x,z)\kappa(z)\,dz\,d\rho_t^N(x),
  \label{eq:particle-jump-quadratic-term}\\
  (\mathcal A_t^N)^2
  &:=\frac1N\sum_{i=1}^N\int_{\R^d}\lambda(X_i(t))
  |\zeta_t^N(X_i(t),z)-\zeta_t(X_i(t),z)|^2\kappa(z)\,dz,
  \qquad\mathcal A_t^N\ge0.
  \label{eq:particle-weighted-rate-difference}
\end{align}
The decomposition separates the change in the coefficients from the
evolution driven by the continuum coefficients. For transport, the
identity
\(b_i^N=v_{\rho_t}(X_i)-\mathsf e_t^N(X_i)\)
isolates a velocity directed down the modulated energy, producing
the dissipation \(-\mathcal I_t^N\). Transporting both measures by
\(v_{\rho_t}\) can still change their discrepancy because a spatially
varying velocity deforms pairwise distances. Symmetrization expresses
this contribution as \(\mathcal C_t^N\), whose velocity difference
allows it to be controlled by the regularity of \(v_{\rho_t}\). The jump contribution admits the analogous splitting
\begin{equation}
\begin{aligned}
  \zeta_t^N(x,z)\,d\rho_t^N(x)
       -\zeta_t(x,z)\,d\rho_t(x)=
    \bigl(\zeta_t^N(x,z)-\zeta_t(x,z)\bigr)\,d\rho_t^N(x)
       +\zeta_t(x,z)\,d\vartheta_t^N(x).
\end{aligned}
\end{equation}
After pairing with the increment of \(g_t^N\), monotonicity of the
positive-part function makes the first term dissipative, yielding
\(-(\mathcal A_t^N)^2\) up to a correction of order \(N^{-1}\).
The second term describes the evolution of the discrepancy under
the continuum jump rule and is estimated below by the product
inequality. Finally, each particle jump moves a mass \(1/N\) at
once, so the quadratic energy has an additional second-order
increment, giving \(\mathcal Q_t^N\). Keeping the dissipation terms
\(\mathcal I_t^N\) and \((\mathcal A_t^N)^2\) on the left-hand side of the
energy inequality will allow us to absorb the velocity and acceptance
errors in the coupling estimate.

\begin{lemma}
\label{prop:particle-modulated-evolution}
There is a martingale \(M_N\), with
\(M_N(0)=0\), such that, almost surely, for every \(t\in[0,T]\),
\begin{equation}
\label{eq:particle-modulated-evolution}
  F_N(t)=F_N(0)+M_N(t)
  +\int_0^t\bigl(-\mathcal I_s^N+\mathcal C_s^N
                    +\mathcal J_s^N+\mathcal Q_s^N\bigr)\,ds.
\end{equation}
\end{lemma}

\begin{proof}
By \eqref{eq:particle-finite-jump-count} and
\eqref{eq:particle-deterministic-flow}, a.e. particle path on
\([0,T]\) consists of finitely many flow segments with bounded,
pairwise separated positions, including their endpoint limits.
We first compute the continuous evolution. With the regularization \eqref{eq:kernel-regularization}, set
\begin{equation}
  F_N^\varepsilon(t):=\frac12
  \iint_{\R^d\times\R^d}K_\varepsilon(x-y)
  \,d\vartheta_t^N(x)\,d\vartheta_t^N(y).
\end{equation}
Since \(\zeta_t(x,z)\le\Lambda_T|z|\), changing variables in the gain
term and using the moment bounds gives
\begin{equation}
\label{eq:particle-source-weighted-integrability}
  \int_{\R^d}\langle x\rangle^q|J(\rho_t)(x)|\,dx
  \le C_T\bigl((1+M_{r,T})\mathcal M_1(\kappa)
                   +\mathcal M_{q+1}(\kappa)\bigr)<\infty.
\end{equation}
On each flow segment, \(t\mapsto\int_{\R^d}\varphi\,d\vartheta_t^N\)
is absolutely continuous for \(\varphi\in C_c^\infty(\R^d)\).
Apply the product rule \cite[Theorem~2.3]{Eberle2013} to products
of these integrals. Approximation by smooth tensor products, followed
by removal of spatial cutoffs, is justified by
\eqref{eq:particle-source-weighted-integrability},
the uniform \(q\)-moment bounds, and
\(|K_\varepsilon(x-y)|+|\nabla K_\varepsilon(x-y)|
\le C_\varepsilon(1+|x|^q+|y|^q)\).
The weak equation and \(\nabla K_\varepsilon(0)=0\) therefore give
\begin{align}
  \frac d{dt}F_N^\varepsilon(t)
  ={}&\frac1N\sum_{i=1}^N b_i^N(X(t))\cdot
           (\nabla K_\varepsilon*\vartheta_t^N)(X_i(t))\nonumber \\
  &-\int_{\R^d}v_{\rho_t}\cdot
           (\nabla K_\varepsilon*\vartheta_t^N)\,d\rho_t
   -\int_{\R^d}(K_\varepsilon*\vartheta_t^N)J(\rho_t)\,dx.
\end{align}
We pass to the limit in the time-integrated form of the preceding
identity. Since
\(
  |v_{\rho_t}^j\rho_t|\le\Lambda_T\rho_t\), \(1\le j\le d
\),
the inequality \(|D^2K|\le C_{d,q}(-\Delta K)\) and
\eqref{eq:transport-laplacian-bound}, applied with \(\mu=\rho_t\),
give, when \(d+q>2\),
\begin{equation}
  \|D^2K*\rho_t\|_\infty
  +\sum_{j=1}^d
    \|D^2K*(v_{\rho_t}^j\rho_t)\|_\infty
  \le C_T.
\end{equation}
For \(p=\infty\), we first interpolate to a finite admissible
exponent. At \((d,q)=(1,1)\), the same bounds follow from
\eqref{eq:one-dimensional-velocity-identities}.
Thus \(\nabla K*\rho_t\) and
\(\nabla K*(v_{\rho_t}^j\rho_t)\) are uniformly Lipschitz.
Since \(K_\varepsilon=\eta_\varepsilon*K\), mollification yields
\begin{equation}
\label{eq:particle-continuum-force-mollification}
\begin{aligned}
  \|\nabla K_\varepsilon*\rho_t-\nabla K*\rho_t\|_\infty+\sum_{j=1}^d
    \|\nabla K_\varepsilon*(v_{\rho_t}^j\rho_t)
       -\nabla K*(v_{\rho_t}^j\rho_t)\|_\infty
  \le C_T\varepsilon.
\end{aligned}
\end{equation}
After expanding \(\vartheta_t^N=\rho_t^N-\rho_t\), the first
term in \eqref{eq:particle-continuum-force-mollification}
controls the convolutions against \(\rho_t\).
The remaining terms control the mixed particle-continuum
pairings, since oddness of \(\nabla K_\varepsilon\) gives
\begin{equation}
  \int_{\R^d}
    v_{\rho_t}(x)\cdot\nabla K_\varepsilon(x-X_i(t))
    \,d\rho_t(x)
  =-\sum_{j=1}^d
    \bigl((\partial_jK_\varepsilon)*
          (v_{\rho_t}^j\rho_t)\bigr)(X_i(t)).
\end{equation}
On each closed flow segment, the particle positions remain
bounded and uniformly separated, so the particle sums converge
uniformly and the drifts \(b_i^N(X(t))\) are bounded. For the source term and the energies,
\eqref{eq:particle-discrepancy-potential-Holder} gives
\begin{equation}
  \|K_\varepsilon*\vartheta_t^N-g_t^N\|_\infty
  +|F_N^\varepsilon(t)-F_N(t)|
  \le4\varepsilon^{q/2}\sqrt{F_N(t)}.
\end{equation}
Combining this with
\eqref{eq:particle-source-weighted-integrability}, we obtain
\begin{equation}
\begin{aligned}
  \left|\int_{\R^d}
    (K_\varepsilon*\vartheta_t^N-g_t^N)J(\rho_t)\,dx\right|\le
    \|K_\varepsilon*\vartheta_t^N-g_t^N\|_\infty
    \|J(\rho_t)\|_1
  \le C_T\varepsilon^{q/2}\sqrt{F_N(t)}.
\end{aligned}
\end{equation}
By \eqref{eq:particle-energy-moment-bound}, these errors tend
to zero uniformly on \([0,T]\), almost surely.
We may therefore pass to the limit.
Substituting \eqref{eq:particle-drift-error-decomposition}
and symmetrizing by the oddness of \(\nabla K\), we obtain
\begin{equation}
\label{eq:particle-continuous-energy-derivative}
  \frac d{dt}F_N(t)
  =-\mathcal I_t^N+\mathcal C_t^N
    -\int_{\R^d}g_t^N(x)J(\rho_t)(x)\,dx
\end{equation}
a.e. on each flow segment.

For a
reference measure \(\rho_s\), the change in \(F_N\) when
particle \(i\) jumps by \(z\) is
\begin{equation}
  \Delta_i F_N(s,z)
  :=\frac{g_{s-}^N(X_i(s-)+z)-g_{s-}^N(X_i(s-))}{N}
    +\frac{|z|^q}{N^2}.
\end{equation}
The Lipschitz property of the positive part and
\eqref{eq:particle-discrepancy-potential-Holder} give
\begin{equation}
\label{eq:particle-empirical-rate-upper-bound}
  \zeta_t^N(x,z)
  \le\Lambda_T|z|+2\sqrt{F_N(t)}\,|z|^{q/2}.
\end{equation}
It then follows that
\begin{equation}
  \sum_{i=1}^N\int_{\R^d}|\Delta_iF_N(s,z)|\lambda(X_i(s-))
    \zeta_{s-}^N(X_i(s-),z)\kappa(z)\,dz
  \le C_T\bigl(1+F_N(s-)\bigr).
\end{equation}
Its time integral has finite expectation by
\eqref{eq:particle-energy-moment-bound}. Define
\begin{equation}
\label{eq:particle-compensated-Poisson-measure}
  \widetilde{\mathcal N}_i(ds,dz,du)
  :=\mathcal N_i(ds,dz,du)-ds\,\kappa(z)\,dz\,du.
\end{equation}
The integrands below are predictable for the filtration generated
by \(X^0\) and \(\mathcal N_1,\ldots,\mathcal N_N\).
The preceding integrability estimate verifies the hypothesis of
\cite[Theorem~1.11]{Eberle2013}, so
\begin{equation}
  M_N(t):=\sum_{i=1}^N
  \int_0^t\int_{\R^d}\int_0^\infty
  \Delta_iF_N(s,z)\,
  \mathbf1_{\{u\le\lambda(X_i(s-))
                        \zeta_{s-}^N(X_i(s-),z)\}}
  \,\widetilde{\mathcal N}_i(ds,dz,du)
\end{equation}
is an integrable martingale with \(M_N(0)=0\). Its paths are
c\`adl\`ag because the uncompensated integrals are c\`adl\`ag by
the same theorem and their compensators are absolutely continuous.
Integrate \eqref{eq:particle-continuous-energy-derivative} over the
flow segments and add the jump increments. Replacing their sum by
\(M_N\) and its compensator, the gain-loss formula for \(J\)
identifies the linear jump contribution as \(\mathcal J_s^N\),
and the quadratic increment gives \(\mathcal Q_s^N\).
This proves \eqref{eq:particle-modulated-evolution} for every
\(t\in[0,T]\) almost surely.
\end{proof}

To estimate \(\mathcal J_t^N\), we use
Lemma~\ref{lem:particle-critical-product} in
Appendix~\ref{app:particle-product-estimate}. It splits a product into a
\(\dot H^{s_q}\) term and a bounded Lipschitz term.

\begin{lemma}
\label{lem:particle-jump-drift}
For a.e. \(t\in[0,T]\),
\begin{equation}
\label{eq:particle-jump-drift-bound}
  \mathcal J_t^N
  \le-(\mathcal A_t^N)^2
     +C_T\bigl(F_N(t)+d_{\rm BL}(\rho_t^N,\rho_t)^2+N^{-1}\bigr).
\end{equation}
Moreover,
\begin{equation}
\label{eq:particle-rate-difference-energy-bound}
  (\mathcal A_t^N)^2
  \le8\|\lambda\|_{L^\infty}\mathcal M_q(\kappa)F_N(t)
     +\frac{2\|\lambda\|_{L^\infty}\Lambda_T}{N}
       \mathcal M_{q+1}(\kappa).
\end{equation}
\end{lemma}

\begin{proof}
For \(a,b\in\mathbb R\) and \(c\ge0\), set
\(d=[a-c]_+-[b]_+\). Then it holds that
\begin{equation}
\label{eq:particle-positive-part-firm}
  (b-a)d\le-|d|^2+c[b]_+,
  \qquad
  |d|\le|a-b|+c\wedge[b]_+.
\end{equation}
Applying the first inequality with \( a=f_{\rho_t^N}(x)-f_{\rho_t^N}(x+z)\), \(b=f_{\rho_t}(x)-f_{\rho_t}(x+z)\), and \(c=\frac{|z|^q}{N}\),
we obtain
\begin{equation}
  (g_t^N(x+z)-g_t^N(x))
  (\zeta_t^N(x,z)-\zeta_t(x,z))
  \le-|\zeta_t^N(x,z)-\zeta_t(x,z)|^2
     +\frac{|z|^q}{N}\zeta_t(x,z).
\end{equation}
It follows that
\begin{equation}
    \begin{aligned}
                  \mathcal J_t^N\le-(\mathcal A_t^N)^2
  +\int_{\R^d}\int_{\R^d}u^N_{t, z}(x)
       \lambda(x)\zeta_t(x,z)\,d\vartheta_t^N(x)\,\kappa(z)\,dz +\frac{\|\lambda\|_\infty\Lambda_T}{N}
                \mathcal M_{q+1}(\kappa),
    \end{aligned}
    \label{eq:particle-jump-monotonicity-splitting}
\end{equation}
where we have written $a_{t,z}(x)=\lambda(x)\zeta_t(x,z)$ and $u_{t,z}^N(x)=g_t^N(x+z)-g_t^N(x)$. To apply Lemma~\ref{lem:particle-critical-product} in the form
\eqref{eq:particle-critical-product-bound}, note first that 
\(\|a_{t,z}\|_{W^{1,\infty}}
\le C\|\lambda\|_{W^{1,\infty}}\Lambda_T|z|\).
The Fourier representation \eqref{eq:extended-fourier-energy} and
\eqref{eq:particle-discrepancy-potential-Holder} yield
\begin{equation}
\label{eq:particle-jump-product-norms}
\begin{aligned}
  \|u_{t,z}^N\|_{\dot H^{s_q}}^2
   \le C_{d,q}F_N(t),\quad
  \|u_{t,z}^N\|_\infty
  &\le2\sqrt{F_N(t)}\,|z|^{q/2},
  \quad [u_{t,z}^N]_{C^{q/2}}\le4\sqrt{F_N(t)},
\end{aligned}
\end{equation}
so \eqref{eq:particle-critical-product-bound} gives
\begin{equation}
\label{eq:particle-jump-product-pointwise-bound}
  \left|\int_{\R^d}a_{t,z}u_{t,z}^N\,d\vartheta_t^N\right|
  \le C_T|z|\bigl(F_N(t)
       +(1+|z|^{q/2})\sqrt{F_N(t)}\,
                      d_{\rm BL}(\rho_t^N,\rho_t)\bigr).
\end{equation}
Integrating in \(z\)
and applying Young's inequality in
\eqref{eq:particle-jump-monotonicity-splitting} proves
\eqref{eq:particle-jump-drift-bound}.
Finally, the second inequality in
\eqref{eq:particle-positive-part-firm}, with the same choices
of \(a,b,c\), gives
\begin{equation}
\label{eq:particle-pointwise-rate-difference}
  |\zeta_t^N(x,z)-\zeta_t(x,z)|
  \le2\sqrt{F_N(t)}\,|z|^{q/2}
     +\left(\frac{|z|^q}{N}\wedge\zeta_t(x,z)\right).
\end{equation}
Squaring, using \((u+v)^2\le2u^2+2v^2\) and
\((c\wedge\zeta_t)^2\le c\zeta_t\), and then
\(\zeta_t(x,z)\le\Lambda_T|z|\), yields
\begin{equation}
  |\zeta_t^N(x,z)-\zeta_t(x,z)|^2
  \le8F_N(t)|z|^q+\frac{2\Lambda_T}{N}|z|^{q+1}.
\end{equation}
Integrating against
\(\lambda(x)\rho_t^N(dx)\kappa(z)\,dz\)
proves \eqref{eq:particle-rate-difference-energy-bound}.
\end{proof}

Substituting \eqref{eq:particle-empirical-rate-upper-bound} into
\eqref{eq:particle-jump-quadratic-term} gives
\begin{equation}
\label{eq:particle-jump-quadratic-bound}
  \mathcal Q_t^N
  \le\frac{\|\lambda\|_\infty}{N}
     \bigl(\Lambda_T\mathcal M_{q+1}(\kappa)
       +2\sqrt{F_N(t)}\,\mathcal M_{3q/2}(\kappa)\bigr)
  \le\frac{C_T}{N}\bigl(1+F_N(t)\bigr).
\end{equation}

We now estimate the additional distance in
\eqref{eq:particle-jump-drift-bound} by coupling to independent processes
whose common law at time \(t\) is \(\rho_t\).

Following the coupling used by Andreis, Dai Pra, and
Fischer~\cite{andreis2018mckean}, we drive each comparison process
by the same Poisson measure as its particle. A proposal accepted
by both produces the same displacement, leaving their separation
unchanged. The comparison velocity $v_{\rho_t}$ and acceptance intensity
$\lambda(x)\zeta_{\rho_t}(x,x+z)$ are determined by the fixed deterministic
curve \((\rho_t)_{t\in[0,T]}\). Choose
\begin{equation}
\label{eq:particle-comparison-initial-law}
  Y^0=(Y_1^0,\ldots,Y_N^0)\sim\rho_0^{\otimes N},
\end{equation}
with any joint law of \((X^0,Y^0)\) having the prescribed marginals,
and assume
\begin{equation}
\label{eq:particle-driving-independence}
  \sigma(X^0,Y^0)\quad\text{is independent of}\quad
  \sigma(\mathcal N_1,\ldots,\mathcal N_N).
\end{equation}
From this point on, \((\mathcal F_t)_{t\ge0}\) denotes the
right-continuous completion of the filtration generated by
\((X^0,Y^0)\) and the restrictions of
\(\mathcal N_1,\ldots,\mathcal N_N\) to
\([0,t]\times\R^d\times[0,\infty)\).
Solutions of the comparison equation below are understood to be
\((\mathcal F_t)\)-adapted with c\`adl\`ag paths, and subsequent
martingale statements refer to this filtration. By
\eqref{eq:particle-driving-independence}, adjoining the whole initial pair
leaves the Poisson intensities and compensators unchanged: future noise
increments are independent of the initial pair and of the past noises.
The compensation and martingale arguments below therefore apply to this
filtration for arbitrary dependence between $X^0$ and $Y^0$.
Independence of the coordinates of $Y^0$ is unconditional; no product-law
assertion conditional on $X^0$ is used. With \(1\leq i\leq N\) and using the same \(\mathcal N_i\) as in
\eqref{eq:jump-particle-system}, consider
\begin{equation}
\label{eq:particle-comparison-process}
\begin{aligned}
  Y_i(t)=Y_i^0+\int_0^t v_{\rho_s}(Y_i(s-))\,ds+\int_0^t\int_{\R^d}\int_0^\infty
  z\,\mathbf1_{\{u\le\lambda(Y_i(s-))\zeta_s(Y_i(s-),z)\}}
  \,\mathcal N_i(ds,dz,du).
\end{aligned}
\end{equation}
We identify the marginals as \(\rho_t\) using uniqueness for the
linear forward equation. Independence across labels follows
because each \(Y_i\) is determined by its own pair
\((Y_i^0,\mathcal N_i)\), and these pairs are independent.

\begin{lemma}
\label{prop:particle-comparison-process}
Equation \eqref{eq:particle-comparison-process} has a pathwise unique solution on \([0,T]\). The processes
\(Y_1,\ldots,Y_N\) are independent, and
\begin{equation}
\label{eq:particle-comparison-law}
  \operatorname{Law}(Y_i(t))=\rho_t,
  \qquad 0\le t\le T,\quad 1\le i\le N.
\end{equation}
\end{lemma}

\begin{proof}
Equation \eqref{eq:particle-continuum-control} gives \(  0\le\lambda(x)\zeta_t(x,z)
  \le\|\lambda\|_\infty\Lambda_T|z|\).
Hence the number of points \((s,z,u)\) of \(\mathcal N_i\) with
\(s\le T\) and \(u\le\|\lambda\|_\infty\Lambda_T|z|\) is Poisson
with finite mean
\(T\|\lambda\|_\infty\Lambda_T\mathcal M_1(\kappa)\).
Order these finitely many points by their time coordinates. Between
successive times solve \(\dot y=v_{\rho_t}(y)\); at \((s,z,u)\), set
\begin{equation}
  Y_i(s)=Y_i(s-)+z\,
  \mathbf1_{\{u\le\lambda(Y_i(s-))\zeta_s(Y_i(s-),z)\}}.
\end{equation}
The bounded Lipschitz estimate in
\eqref{eq:weak-solution-velocity-class} proves existence and uniqueness
on \([0,T]\) in this construction.

To identify the law, fix \(1<a<\infty\) with
\(\rho_0\in L^a\): take \(a=p\) if \(p<\infty\), and \(a=2\)
if \(p=\infty\). Conditional on the time and displacement coordinates
\((s,z)\) of the preceding points, the corresponding \(u\) is uniform
on \((0,\|\lambda\|_\infty\Lambda_T|z|)\).
Integrating over \(u\) therefore replaces a density \(h\ge0\) by
\begin{equation}
\label{eq:particle-comparison-density-update}
  \left(1-\frac{\lambda(x)\zeta_s(x,z)}
                    {\|\lambda\|_\infty\Lambda_T|z|}\right)h(x)
  +\frac{\lambda(x-z)\zeta_s(x-z,z)}
                    {\|\lambda\|_\infty\Lambda_T|z|}\,h(x-z).
\end{equation}
Here \(z\ne0\) almost surely. The expression in
\eqref{eq:particle-comparison-density-update} has integral
\(\|h\|_1\) and \(L^a\)-norm at most \(2\|h\|_a\).
The bound \eqref{eq:transport-evolution-bounds} applies between these
times. Conditional on their number being \(m\), the density at any
\(t\le T\) thus has \(L^a\)-norm at most
\begin{equation}
  2^m\exp\!\left((1-a^{-1})
             \int_0^T\|\nabla\!\cdot v_{\rho_s}\|_\infty\,ds\right)
  \|\rho_0\|_a.
\end{equation}
Summing over the Poisson probabilities proves that
\(\gamma_t:=\operatorname{Law}(Y_i(t))\) has a density and
\begin{equation}
\label{eq:particle-comparison-density-bounds}
  \sup_{0\le t\le T}\|\gamma_t\|_a
  \le\exp\!\left(
       (1-a^{-1})\int_0^T\|\nabla\!\cdot v_{\rho_s}\|_\infty\,ds
       +T\|\lambda\|_\infty\Lambda_T\mathcal M_1(\kappa)
       \right)\|\rho_0\|_a.
\end{equation}
For \(\varphi\in C_c^\infty(\R^d)\), the chain rule between jumps
and \eqref{eq:particle-comparison-process} give
\begin{align}
  \varphi&(Y_i(t))=\varphi(Y_i^0)
  +\int_0^t\nabla\varphi(Y_i(s))\cdot v_{\rho_s}(Y_i(s))\,ds\nonumber \\
  &+\int_0^t\int_{\R^d}\int_0^\infty
   \bigl(\varphi(Y_i(s-)+z)-\varphi(Y_i(s-))\bigr)
   \mathbf1_{\{u\le\lambda(Y_i(s-))\zeta_s(Y_i(s-),z)\}}
   \,\mathcal N_i(ds,dz,du).\label{eq:particle-comparison-test-identity}
\end{align}
The integral of the absolute value of the Poisson integrand has expectation at most\newline
\(2T\|\varphi\|_\infty\|\lambda\|_\infty
\Lambda_T\mathcal M_1(\kappa)\). Thus the compensation formula applies, and taking expectations gives
\begin{align}
  \int_{\R^d}\varphi\,d\gamma_t
  ={}&\int_{\R^d}\varphi\,d\rho_0
     +\int_0^t\int_{\R^d}\nabla\varphi(x)\cdot v_{\rho_s}(x)
                                         \,d\gamma_s(x)\,ds\nonumber \\
    &+\int_0^t\iint_{\R^d\times\R^d}
       (\varphi(x+z)-\varphi(x))\lambda(x)\zeta_s(x,z)
                            \kappa(z)\,dz\,d\gamma_s(x)\,ds.\label{eq:particle-comparison-forward-equation}
\end{align}
The weak formulation \eqref{eq:weak-jump-transport-formulation} gives
the same equation with \(\gamma\) replaced by \(\rho\).
The operators in $J^{\varepsilon}_{\rho,\rm in}$ and $J^{\varepsilon}_{\rho, \rm out}$, with
\(\varepsilon=0\) and coefficients fixed by \(\rho\), satisfy the
hypotheses of Lemma~\ref{lem:linear-jump-transport}. Both densities
belong to \(L^\infty(0,T;L^1\cap L^a)\), by
\eqref{eq:particle-comparison-density-bounds},
\eqref{eq:finite-p-weak-solution-class}, and
\eqref{eq:infinite-p-weak-solution-class}. The distributional uniqueness
assertion of Lemma~\ref{lem:linear-jump-transport} gives \(\gamma_t=\rho_t\) for a.e.
\(t\). For every deterministic \(t\in(0,T]\),
\(\mathbb P(Y_i(t)\ne Y_i(t-))=0\).
Bounded convergence along the paths therefore makes \(\gamma\)
narrowly continuous. Since \(\rho\) is
also narrowly continuous, equality holds for every \(t\), proving
\eqref{eq:particle-comparison-law}.

Finally, each \(Y_i\) is a measurable function of
\((Y_i^0,\mathcal N_i)\). These pairs are independent by
\eqref{eq:particle-comparison-initial-law} and
\eqref{eq:particle-driving-independence}, together with the mutual
independence of \(\mathcal N_1,\ldots,\mathcal N_N\).
Therefore \(Y_1,\ldots,Y_N\) are independent.
\end{proof}

Set
\begin{equation}
\label{eq:particle-coupling-quantities}
  \overline\rho_t^N:=\frac1N\sum_{i=1}^N\delta_{Y_i(t)},
  \qquad
  Z_N(t):=\frac1N\sum_{i=1}^N
            \bigl(1-e^{-|X_i(t)-Y_i(t)|}\bigr).
\end{equation}
For every function admissible in \eqref{eq:particle-BL-distance},
\begin{equation}
  |\varphi(x)-\varphi(y)|
  \le2\wedge|x-y|
  \le\frac{2}{1-e^{-1}}\bigl(1-e^{-|x-y|}\bigr).
\end{equation}
Hence
\begin{equation}
\label{eq:particle-BL-coupling-triangle}
  d_{\rm BL}(\rho_t^N,\rho_t)
  \le\frac{2}{1-e^{-1}}Z_N(t)
       +d_{\rm BL}(\overline\rho_t^N,\rho_t).
\end{equation}
The second summand is estimated by Lemma~\ref{lem:particle-iid-BL},
in the form \eqref{eq:particle-iid-BL-rate}, since
\Cref{prop:particle-comparison-process} gives independence and
common law $\rho_t$.
To estimate the first, we use
\begin{equation}
\label{eq:particle-modified-modulated-energy}
  \widetilde F_N(t):=F_N(t)+Z_N(t)^2.
\end{equation}
The bounded coupling cost handles large displacements without higher particle moments. The
next estimate lets the dissipations in \(F_N\) absorb the drift and acceptance errors in $Z^2_N$.

\begin{lemma}
\label{lem:particle-coupling}
There is a square-integrable martingale \(M_N^Z\), with
\(M_N^Z(0)=0\), such that, for every \(\varepsilon>0\),
almost surely, for every \(t\in[0,T]\),
\begin{equation}
\label{eq:particle-coupling-evolution}
\begin{aligned}
  Z_N(t)^2
  \le{}&Z_N(0)^2+M_N^Z(t)+\int_0^t\left[
    \varepsilon\mathcal I_s^N
    +\varepsilon(\mathcal A_s^N)^2
    +C_{\varepsilon,T}Z_N(s)^2
    +\frac{C_{\varepsilon,T}}{N^2}
  \right]ds.
\end{aligned}
\end{equation}
\end{lemma}

\begin{proof}
Write \(D_i(t)=X_i(t)-Y_i(t)\). On each interval without a jump,
\(D_i\) is absolutely continuous, and
\eqref{eq:particle-drift-error-decomposition} gives, a.e.,
\begin{subequations}\label{eq:particle-coupling-continuous-bound}
\begin{align}
  \frac d{dt}\bigl(1-e^{-|D_i|}\bigr)
  &\le e^{-|D_i|}\bigl(
         |\mathsf e_t^N(X_i)|
         +\|Dv_{\rho_t}\|_\infty|D_i|\bigr),\\
  \frac d{dt}Z_N(t)
  &\le(\mathcal I_t^N)^{1/2}+\Lambda_TZ_N(t).
\end{align}
\end{subequations}
The first inequality also holds a.e. on \(\{D_i=0\}\)
by the chain rule for the norm of an absolutely continuous function.
The second uses Cauchy-Schwarz and
\(\ell e^{-\ell}\le1-e^{-\ell}\). Thus
\begin{equation}
\label{eq:particle-coupling-continuous-square}
  \frac d{dt}Z_N(t)^2
  \le\varepsilon\mathcal I_t^N+C_{\varepsilon,T}Z_N(t)^2.
\end{equation}

Recall \(a_{t,z}:=\lambda(x)\zeta_t(x,z)\) from the proof of Lemma~\ref{lem:particle-jump-drift},
and define \(a_{t,z}^N(x):=\lambda(x)\zeta_t^N(x,z)\).
Notice that
\begin{equation}
\label{eq:particle-coupling-spatial-rate}
\begin{aligned}
  |a_{t,z}(X_i)-a_{t,z}(Y_i)|
  \le C_T|z|(1-e^{-|D_i|}).
\end{aligned}
\end{equation}
Consequently, by \eqref{eq:particle-weighted-rate-difference} and
Cauchy-Schwarz,
\begin{equation}
\label{eq:particle-coupling-rate-integrals}
\begin{aligned}
  \frac1N\sum_{i=1}^N\int_{\R^d}
      (1\wedge|z|)|a_{t,z}^N(X_i)-a_{t,z}(Y_i)|\kappa(z)\,dz\le\|\lambda\|_\infty^{1/2}\mathcal A_t^N+C_T Z_N(t).
\end{aligned}
\end{equation}
At a point \((t,z,u)\) of \(\mathcal N_i\), equations
\eqref{eq:jump-particle-system} and
\eqref{eq:particle-comparison-process} give

\begin{equation}
\label{eq:particle-coupling-jump-bound}
\begin{aligned}
  |Z_N(t)-Z_N(t-)|
  &\le\frac{1\wedge|z|}{N}
       \left|\mathbf1_{\{u\le a_{t-,z}^N(X_i(t-))\}}
             -\mathbf1_{\{u\le a_{t,z}(Y_i(t-))\}}\right|.
\end{aligned}
\end{equation}
In particular, a jump accepted by both equations leaves \(D_i\)
unchanged. Using \eqref{eq:particle-compensated-Poisson-measure}, define
\begin{align}
  M_N^Z(t):=\sum_{i=1}^N\int_0^t\int_{\R^d}\int_0^\infty
  \Biggl[\Biggl(Z_N(s-)+\frac{e^{-|D_i(s-)|}}N
      -\frac1N e^{-\left|D_i(s-)+z\square
       \right|}
       \Biggr)^2&-Z_N(s-)^2\Biggr]
  \nonumber \\
    &\times\widetilde{\mathcal N}_i(ds,dz,du),\label{eq:particle-coupling-martingale}
\end{align}
where we write \(\square= \mathbf1_{\{u\le a_{s-,z}^N(X_i(s-))\}}
        -\mathbf1_{\{u\le a_{s,z}(Y_i(s-))\}}\) for brevity.
Both values of \(Z_N\) in \eqref{eq:particle-coupling-jump-bound}
belong to \([0,1]\). The absolute integrand in
\eqref{eq:particle-coupling-martingale} is therefore at most
\(
  2(1\wedge|z|)
  |\square|/N
\).
By \eqref{eq:particle-coupling-rate-integrals},
the expectation of the integral of its square is finite:
\begin{align}
  &\frac4{N^2}\sum_{i=1}^N\mathbb E\left[\int_0^T\int_{\R^d}
       (1\wedge|z|)^2
       |a_{s,z}^N(X_i(s))-a_{s,z}(Y_i(s))|
       \kappa(z)\,dz\,ds\right]\nonumber \\
  &\qquad\qquad\qquad\le\frac{C_T}{N}\int_0^T
       \mathbb E\bigl[\mathcal A_s^N+Z_N(s)\bigr]\,ds<\infty.\label{eq:particle-coupling-square-integrand-bound}
\end{align}
Here we used \(Z_N\le1\),
\eqref{eq:particle-rate-difference-energy-bound}, and
\eqref{eq:particle-energy-moment-bound}.

The integrands in \eqref{eq:particle-coupling-martingale}
are predictable: they depend on the adapted c\`adl\`ag
processes \(X\) and \(Y\) only through their left limits,
and the remaining coefficients are deterministic.
The estimate
\eqref{eq:particle-coupling-square-integrand-bound}
verifies square-integrability. Thus,
\cite[Theorem~1.12]{Eberle2013} shows that \(M_N^Z\)
is a square-integrable martingale with \(M_N^Z(0)=0\).
The theorem's It\^o isometry for the independent compensated
Poisson measures also gives
\begin{equation}
\label{eq:particle-coupling-quadratic-variation}
  \mathbb E\left[|M_N^Z(T)|^2\right]
  \le\frac{C_T}{N}\int_0^T
       \mathbb E\bigl[\mathcal A_s^N+Z_N(s)\bigr]\,ds
  <\infty.
\end{equation}
In particular, \(\mathbb E[M_N^Z(t)]=0\) for every
\(t\in[0,T]\).

Expand the square in \eqref{eq:particle-coupling-martingale} and use
\eqref{eq:particle-coupling-jump-bound}. After integrating in \(u,z\)
and summing over \(i\), the integral with respect to
\(ds\,\kappa(z)\,dz\,du\) is bounded above by the time integral of
\begin{align}
  &\frac{2Z_N(t)}N\sum_{i=1}^N\int_{\R^d}
       (1\wedge|z|)|a_{t,z}^N(X_i)-a_{t,z}(Y_i)|\kappa(z)\,dz\nonumber \\
  &\qquad\qquad\qquad\qquad+\frac1{N^2}\sum_{i=1}^N\int_{\R^d}
       (1\wedge|z|)^2|a_{t,z}^N(X_i)-a_{t,z}(Y_i)|\kappa(z)\,dz\nonumber \\
  &\quad\le C_T\bigl(2Z_N(t)+N^{-1}\bigr)
                   \bigl(\mathcal A_t^N+Z_N(t)\bigr)\nonumber \\
  &\quad\le\varepsilon(\mathcal A_t^N)^2
          +C_{\varepsilon,T}Z_N(t)^2
          +\frac{C_{\varepsilon,T}}{N^2}.\label{eq:particle-coupling-jump-drift-bound}
\end{align}
Here we used \eqref{eq:particle-coupling-rate-integrals} and Young's
inequality. Integrating \eqref{eq:particle-coupling-continuous-square}
between jumps and adding \eqref{eq:particle-coupling-martingale} and
\eqref{eq:particle-coupling-jump-drift-bound} proves
\eqref{eq:particle-coupling-evolution}.
\end{proof}

Adding the energy and coupling estimates absorbs the coupling errors. 
Lemma~\ref{lem:particle-iid-BL} controls
the remaining sampling error, so Gronwall’s inequality closes the estimate.

\begin{proposition}[Modified modulated energy closure]
\label{prop:modulated-energy-closure}
Let \(Y\) solve \eqref{eq:particle-comparison-process}, with
\eqref{eq:particle-comparison-initial-law} and
\eqref{eq:particle-driving-independence}.
There exists $C_T<\infty$, independent of $N$ and of the joint initial law satisfying the above hypotheses, such that
\begin{equation}
\label{eq:particle-modified-mean-field-bound}
\begin{aligned}
  \sup_{0\le t\le T}\mathbb E\left[\widetilde F_N(t)\right]
  &+\mathbb E\left[\int_0^T\bigl(\mathcal I_s^N+(\mathcal A_s^N)^2\bigr)\,ds\right]\le C_T\bigl(\mathbb E\left[\widetilde F_N(0)\right]
                    +\mathfrak n_d(N)^{r/(r+1)}\bigr).
\end{aligned}
\end{equation}
\end{proposition}

\begin{proof}
For each fixed \(t\), \Cref{prop:particle-comparison-process}
and \eqref{eq:particle-iid-BL-rate}, with
\(\nu=\rho_t\) and \(\mathfrak m=M_{r, T}\), give
\begin{equation}
\label{eq:particle-iid-BL-uniform-time}
  \sup_{0\le t\le T}\mathbb E\left[d_{\rm BL}(\overline\rho_t^N,\rho_t)^2\right]
  \le C_T\mathfrak n_d(N)^{r/(r+1)}.
\end{equation}
Squaring \eqref{eq:particle-BL-coupling-triangle}, we have
\begin{equation}
\label{eq:particle-BL-coupling-square}
  d_{\rm BL}(\rho_t^N,\rho_t)^2
  \le C Z_N(t)^2+2d_{\rm BL}(\overline\rho_t^N,\rho_t)^2.
\end{equation}
Using \eqref{eq:particle-modulated-evolution},
\eqref{eq:particle-transport-commutator-bound},
\eqref{eq:particle-jump-drift-bound},
\eqref{eq:particle-jump-quadratic-bound}, and
\eqref{eq:particle-BL-coupling-square}, we obtain
\begin{align}
  F_N(t)+\int_0^t\bigl(\mathcal I_s^N+(\mathcal A_s^N)^2\bigr)\,ds
  \le{}&F_N(0)+M_N(t)\nonumber \\
  &+C_T\int_0^t\left(
     F_N(s)+Z_N(s)^2
     +d_{\rm BL}(\overline\rho_s^N,\rho_s)^2+N^{-1}\right)ds.\label{eq:particle-modulated-energy-inequality}
\end{align}
Add \eqref{eq:particle-coupling-evolution} with
\(\varepsilon=1/2\). By
\eqref{eq:particle-modified-modulated-energy},
\begin{align}
  \widetilde F_N(t)
    +\frac12\int_0^t\bigl(\mathcal I_s^N+(\mathcal A_s^N)^2\bigr)\,ds
  \le{}&\widetilde F_N(0)+M_N(t)+M_N^Z(t)\nonumber \\
  &+C_T\int_0^t\left(
      \widetilde F_N(s)
      +d_{\rm BL}(\overline\rho_s^N,\rho_s)^2+N^{-1}\right)ds.\label{eq:particle-modified-energy-Gronwall}
\end{align}
The right-hand side is integrable by
\eqref{eq:particle-energy-moment-bound}, \(Z_N\le1\), and the
martingale assertions of
\Cref{prop:particle-modulated-evolution,lem:particle-coupling}. Taking expectations and using
\eqref{eq:particle-iid-BL-uniform-time} gives
\begin{align}
  \mathbb E\left[\widetilde F_N(t)\right]
  &+\frac12\mathbb E\left[\int_0^t
                    \bigl(\mathcal I_s^N+(\mathcal A_s^N)^2\bigr)\,ds\right]\le\mathbb E\left[\widetilde F_N(0)\right]
    +C_T\int_0^t\left(
       \mathbb E\left[\widetilde F_N(s)\right]
       +\mathfrak n_d(N)^{r/(r+1)}+N^{-1}\right)ds.
\end{align}
Since \(N^{-1}\le\mathfrak n_d(N)^{r/(r+1)}\), Gronwall's
inequality first bounds \(\sup_{0\le t\le T}\mathbb E\left[\widetilde F_N(t)\right]\).
Substituting this bound into the above inequality at \(t=T\) controls the
integrated dissipation and proves
\eqref{eq:particle-modified-mean-field-bound}.
\end{proof}

\subsection{Initial matching and mean-field consequences}\label{subsec:particle-initial-matching}

For a deterministic $X^0$, coupling its labelled coordinates to $Y^0\sim\rho_0^{\otimes N}$ need not give a small $Z_N(0)$. Sorting or optimally reordering only the sample $Y^0$ generally destroys its product law. The independent comparison construction and its sampling estimate would then no longer apply. We instead relabel both configurations by a common independent uniform permutation after matching. This preserves the product law of the comparison vector and changes only the labels of the interacting empirical measure.

\begin{lemma}\label{lem:particle-initial-matching}
Let $X^0$ have any law on $(\R^d)^N$, let $\nu\in\mathcal P_r(\R^d)$, and write $h(w)=1-e^{-|w|}$. There is a jointly exchangeable pair $(X'^0,Y'^0)$ such that $X'^0$ has the symmetrized law of $X^0$, $Y'^0\sim\nu^{\otimes N}$, $\rho_{X'^0}^N=\rho_{X^0}^N$, and
\begin{equation}\label{eq:particle-initial-matching-bound}
 \mathbb E\left[\left(\frac1N\sum_i h(X_i'^0-Y_i'^0)\right)^2\right]
 \le8\mathbb E\left[d_{\rm BL}(\rho_{X^0}^N,\nu)^2\right]+C_{d,r,\mathcal M_r(\nu)}\mathfrak n_d(N)^{r/(r+1)}.
\end{equation}
The pair can be chosen independently of all future driving noises.
\end{lemma}
\begin{proof}
Introduce the standard Kantorovich transport distance $W_h$ with cost $h(x-y)=1-e^{-|x-y|}$. This notation is used only for the matching argument. The cost is a bounded complete metric inducing the Euclidean topology: the triangle inequality follows from monotonicity and
$1-e^{-(a+b)}=(1-e^{-a})+e^{-a}(1-e^{-b})$, while completeness follows by converting a small $h$-distance to a small Euclidean distance. By Kantorovich--Rubinstein duality for a bounded metric,
\begin{equation}\label{eq:particle-matching-BL}
 W_h(\nu_1,\nu_2)\le2d_{\rm BL}(\nu_1,\nu_2).
\end{equation}
Indeed, after subtracting its value at zero, an $h$-Lipschitz test has supremum norm at most one and Euclidean Lipschitz seminorm at most one. Its half is admissible for $d_{\rm BL}$. The bounded-metric duality and comparison are also recorded in \cite[Section~2.1, (2.2)--(2.5)]{hauray2014kac}.

For equal-weight empirical measures, a minimum-cost coupling is realized by a permutation. Here is the finite-dimensional argument. A transport plan can be lifted to the labels, even if some locations coincide, giving a matrix of nonnegative weights with all row and column sums $1/N$. Multiply by $N$. Among cost-minimizing matrices choose one with the fewest positive entries. If its positive-entry bipartite graph contains a cycle, adding and subtracting a small number alternately along the cycle preserves all row and column sums. A nonzero cost slope would contradict minimality in one direction. A zero slope permits a perturbation until an entry vanishes, contradicting the minimal number of positive entries. Thus the graph is a forest. A leaf has a single incident edge of weight one; the vertex at its other end can have no other positive edge because its total weight is one. Removing this pair and repeating shows that the matrix is a permutation matrix. Therefore
\begin{equation}
 W_h(\rho_x^N,\rho_y^N)=\min_{\sigma\in S_N}\frac1N\sum_i h(x_i-y_{\sigma(i)}).
\end{equation}
Choose the first minimizing permutation in a fixed ordering of $S_N$. This choice is Borel measurable, since it is a minimizer of finitely many continuous costs with a fixed tie-breaking rule.

Sample $V^0\sim\nu^{\otimes N}$ independently of $X^0$. Match $X^0$ to $V^0$ by this permutation $\sigma$. Let $\tau$ be an independent uniform permutation and define
\begin{equation}
 X_i'^0=X_{\tau(i)}^0,\qquad
 Y_i'^0=V_{\sigma(\tau(i))}^0.
\end{equation}
Conditional on $(X^0,V^0)$, composition of $\tau$ with any fixed permutation leaves its uniform law unchanged. Thus the pair is jointly exchangeable and $X'^0$ has the symmetrized law of $X^0$. Moreover, for every bounded Borel function $\Phi$,
\begin{align}
 \mathbb E\left[\Phi(Y'^0)\right]
 &=\mathbb E\left[\frac1{N!}\sum_{\tau\in S_N}\Phi((V^0_{\sigma(\tau(i))})_{i=1}^N)\right]\nonumber \\
 &=\mathbb E\left[\frac1{N!}\sum_{\pi\in S_N}\Phi((V^0_{\pi(i)})_{i=1}^N)\right]
 =\mathbb E\left[\Phi(V^0)\right].
\end{align}
The middle equality holds for each fixed pair $(X^0,V^0)$ even though $\sigma$ depends on it. The last uses the product law of $V^0$. Hence $Y'^0$ has exactly the product law, not just the correct one-coordinate marginals.

The average matched cost is unchanged by the final common permutation. The triangle inequality for $W_h$ and \eqref{eq:particle-matching-BL} give
\begin{equation}
 \frac1N\sum_i h(X_i'^0-Y_i'^0)
 =W_h(\rho_{X^0}^N,\rho_{V^0}^N)
 \le2d_{\rm BL}(\rho_{X^0}^N,\nu)+2d_{\rm BL}(\rho_{V^0}^N,\nu).
\end{equation}
Squaring and applying Lemma~\ref{lem:particle-iid-BL} proves \eqref{eq:particle-initial-matching-bound}. All these initial random variables can be sampled before and independently of the Poisson measures. This proves the final assertion.
\end{proof}

\begin{proof}[Proof of Theorem~\ref{thm:particle-empirical-mean-field}]
Use Lemma~\ref{lem:particle-initial-matching} with $\nu=\rho_0$, and run the coupled systems from $(X'^0,Y'^0)$. Initial collision-freeness and the moment condition are preserved by permutations. The initial $F_N$ is unchanged, and \eqref{eq:particle-initial-matching-bound} gives
\begin{equation}
 \mathbb E\left[\widetilde F_N(0)\right]+\mathfrak n_d(N)^{r/(r+1)}\le C\Theta_N.
\end{equation}
Apply Proposition~\ref{prop:modulated-energy-closure}, then use \eqref{eq:particle-BL-coupling-square} and \eqref{eq:particle-iid-BL-uniform-time} to bound the additional bounded-Lipschitz error.

It remains to transfer the conclusion from the symmetrized initial law to the original one. The particle drift, acceptance rule, and diagonal convention are equivariant under permutations. Pathwise uniqueness implies that permuting both an initial vector and its labelled driving noises permutes the solution in the same way. The independent noises have a permutation-invariant law. Consequently, for each deterministic initial vector, the law of the empirical path is unchanged by permuting that vector. Averaging over the initial law and permutations proves equality of empirical-path laws for $X$ and $X'$. The quantities $F_N$, $d_{\rm BL}(\rho_t^N,\rho_t)$, $\mathcal I_t^N$, and $(\mathcal A_t^N)^2$ are all invariant under label permutations. Their laws, including their time integrals, therefore agree. The same argument gives the stronger estimate
\begin{equation}\label{eq:particle-empirical-integrated-error}
 \sup_{0\le t\le T}\mathbb E\left[F_N(t)+d_{\rm BL}(\rho_t^N,\rho_t)^2\right]
 +\mathbb E\left[\int_0^T\bigl(\mathcal I_t^N+(\mathcal A_t^N)^2\bigr)\,dt\right]
 \le C_T\Theta_N.
\end{equation}
In particular, \eqref{eq:particle-empirical-mean-field} holds for the original system.
\end{proof}

\begin{proof}[Proof of Corollary~\ref{thm:particle-iid-mean-field}]
    Choose \(Y_i^0=X_i^0\) in
\eqref{eq:particle-comparison-process}. Then \(Z_N(0)=0\), so
\(\widetilde F_N(0)=F_N(0)\). Independence gives
\begin{equation}
\label{eq:particle-initial-iid-MMD}
  \mathbb E[F_N(0)]
  =\frac1{2N}\iint_{\mathbb R^d\times\mathbb R^d}
       |x-y|^q\,d\rho_0(x)\,d\rho_0(y)
  \le\frac{C_q\mathcal M_q(\rho_0)}N.
\end{equation}
Using \eqref{eq:particle-initial-iid-MMD},
\(N^{-1}\le\mathfrak n_d(N)^{r/(r+1)}\) in
\eqref{eq:particle-modified-mean-field-bound}, and Cauchy-Schwarz proves
\eqref{eq:particle-iid-MMD-first-power}. 
\end{proof}

\begin{remark}[Qualitative convergence from the initial MMD]
\label{rem:particle-initial-MMD}
Under the other hypotheses of Theorem~\ref{thm:particle-empirical-mean-field},
$\mathbb E[F_N(0)]\to0$ suffices for $\Theta_N\to0$.
Indeed, \cite[Proposition~A.4]{rosenzweig2026wassersteingradientflowsmaximum}
shows that MMD convergence to the fixed $\rho_0$ implies narrow convergence,
hence bounded-Lipschitz convergence. For every $\eta>0$ there is $\delta>0$ such that
\begin{equation}
 \nu\in\mathcal P_q,\quad\mathcal E_q(\nu-\rho_0)<\delta
 \quad\Longrightarrow\quad d_{\rm BL}(\nu,\rho_0)<\eta.
\end{equation}
Otherwise a sequence contradicts that proposition. Since $d_{\rm BL}\le2$,
Markov's inequality gives
\begin{equation}
 \mathbb E\left[d_{\rm BL}(\rho_0^N,\rho_0)^2\right]
 \le\eta^2+4\mathbb P(F_N(0)\ge\delta)
 \le\eta^2+4\delta^{-1}\mathbb E[F_N(0)].
\end{equation}
Letting $N\to\infty$ and then $\eta\downarrow0$ proves the assertion.
This implication is qualitative. It does not supply a linear quantitative
estimate in the initial MMD alone.
\end{remark}

For exchangeable initial configurations, the mean-field estimate also yields
a result for finitely many labelled paths. Let $Q_T$ denote the law of one comparison process from
\Cref{prop:particle-comparison-process} on
$D([0,T];\mathbb R^d)$, with the Skorokhod $J_1$ topology; see \cite[Chapter~3, Section~5]{ethierkurtz1986markov}.
An exchangeable initial interacting law need not have independent
coordinates. Exchangeability is enough to pass from an averaged coupling
estimate to a bound for specified labels.

\begin{corollary}[Path-space propagation of chaos]
\label{thm:particle-propagation-of-chaos}
Assume the hypotheses of Theorem~\ref{thm:particle-empirical-mean-field}
and that the law of $X^0$ is exchangeable. There is a coupling of $X$
with independent comparison paths of law $Q_T$ such that, for every
$1\le k\le N$,
\begin{equation}\label{eq:particle-general-path-coupling}
 \mathbb E\left[\sum_{i=1}^k\sup_{0\le t\le T}
 \left(1-e^{-|X_i(t)-Y_i(t)|}\right)\right]\le C_T k\sqrt{\Theta_N}.
\end{equation}
If $\Theta_N\to0$, then, for every fixed $k$,
\begin{equation}\label{eq:particle-path-propagation-of-chaos}
 \operatorname{Law}(X_1|_{[0,T]},\ldots,X_k|_{[0,T]})
 \rightharpoonup Q_T^{\otimes k}
 \qquad\text{as }N\to\infty,\quad N\ge k,
\end{equation}
in the product Skorokhod topology. At every fixed deterministic time $t\ge0$,
\begin{equation}
 \operatorname{Law}(X_1(t),\ldots,X_k(t))\rightharpoonup\rho_t^{\otimes k}.
\end{equation}
\end{corollary}
\begin{proof}
Apply Lemma~\ref{lem:particle-initial-matching} with $\nu=\rho_0$.
Since the initial law is exchangeable, its symmetrization is the original
law. The joint initial pair is exchangeable and independent of the
exchangeable family of noises. Permutation equivariance and pathwise
uniqueness imply that the paired processes $(X_i,Y_i)_{i=1}^N$ are
exchangeable.

Write $h(w)=1-e^{-|w|}$, $D_i=X_i-Y_i$, and set
\begin{equation}
 H_N(t)=\frac1N\sum_i\sup_{0\le u\le t}h(D_i(u)),\qquad
 V_N(t)=\frac1N\sum_i\sum_{0<s\le t}|h(D_i(s))-h(D_i(s-))|.
\end{equation}
The jump sum is finite on compact intervals. Integrating the first
inequality in \eqref{eq:particle-coupling-continuous-bound} between jumps,
adding the absolute jump increments, and taking the supremum for each
label before averaging gives
\begin{equation}
 H_N(t)\le Z_N(0)+\int_0^t(\mathcal I_s^N)^{1/2}\,ds
             +\Lambda_T\int_0^tH_N(s)\,ds+V_N(t).
\end{equation}
Here Cauchy--Schwarz estimates the drift error, and
$\ell e^{-\ell}\le1-e^{-\ell}$ bounds the remaining continuous term.
The compensation formula, \eqref{eq:particle-coupling-jump-bound}, and
\eqref{eq:particle-coupling-rate-integrals} give
\begin{equation}
 \mathbb E[V_N(t)]\le C_T\int_0^t
                \mathbb E[\mathcal A_s^N+Z_N(s)]\,ds.
\end{equation}
Since $Z_N\le H_N$, Gronwall's inequality and Cauchy--Schwarz yield
\begin{equation}
 \mathbb E[H_N(T)]
 \le C_T\left(\mathbb E[Z_N(0)]
       +\left[\mathbb E\left[\int_0^T
            (\mathcal I_s^N+(\mathcal A_s^N)^2)\,ds\right]\right]^{1/2}\right)
 \le C_T\sqrt{\Theta_N},
\end{equation}
where the final bound follows from Lemma~\ref{lem:particle-initial-matching}
and Proposition~\ref{prop:modulated-energy-closure}.
Exchangeability implies
\begin{equation}
 \mathbb E\left[\sum_{i=1}^k\sup_{t\le T}h(D_i(t))\right]=k\mathbb E[H_N(T)],
\end{equation}
proving \eqref{eq:particle-general-path-coupling}.
For fixed $k$ and $\epsilon>0$, the probability that any of the first
$k$ uniform path separations exceeds $\epsilon$ is at most the left-hand
side divided by $1-e^{-\epsilon}$. Thus uniform separation tends to zero
in probability. The identity time change gives Skorokhod closeness, and
the comparison paths have law $Q_T^{\otimes k}$ for every $N$.
Testing against bounded Lipschitz functions in a metric generating the
product Skorokhod topology proves \eqref{eq:particle-path-propagation-of-chaos}.
The fixed-time assertion follows directly from the uniform coupling,
without requiring continuity of a time-evaluation map on all of path space.
\end{proof}

For independent initial sampling, the direct choice $Y^0=X^0$ has
$Z_N(0)=0$. The same proof and \eqref{eq:particle-initial-iid-MMD} give
\begin{equation}\label{eq:particle-pathwise-coupling-bound}
 \mathbb E\left[\sum_{i=1}^k\sup_{0\le t\le T}
 (1-e^{-|X_i(t)-Y_i(t)|})\right]
 \le C_T k\,\mathfrak n_d(N)^{r/(2(r+1))},\qquad1\le k\le N.
\end{equation}
More generally, \eqref{eq:particle-general-path-coupling} permits
$k=k_N$ whenever $k_N\sqrt{\Theta_N}\to0$. This is the displayed
bounded-cost comparison on changing product spaces, not convergence in
total variation or relative entropy; no optimality of that range is claimed.

Exchangeability cannot be omitted from a conclusion about the first
labels. For example, approximate a bounded density on a cube by distinct
deterministic grid points, but fix the first particle at a point outside
the cube. The configurations can remain in one compact set. Their
empirical measures converge narrowly and in MMD to the reference density,
yet the first initial marginal remains a point mass. The empirical
estimate applies, but cannot imply convergence of that distinguished
label on an interval containing zero.

\subsection{Long-time target approximation}
\label{subsec:particle-long-time}

Finally, the finite-time mean-field estimate and particle energy monotonicity
combine with continuum convergence to control the particle system on
diverging time intervals.
\begin{corollary}[Long-time target approximation]
\label{thm:particle-long-time-convergence}
Assume the hypotheses of Theorem~\ref{thm:particle-empirical-mean-field},
$r\ge2$, and $\Theta_N\to0$. Then
\begin{equation}\label{eq:particle-long-time-convergence}
 \lim_{S\to\infty}\limsup_{N\to\infty}
 \mathbb E\left[\sup_{t\ge S}\operatorname{MMD}_q(\rho_t^N,\mu)\right]=0.
\end{equation}
For every deterministic sequence $t_N\to\infty$, one also has
\begin{equation}\label{eq:particle-long-time-diagonal-convergence}
 \lim_{N\to\infty}\mathbb E\left[\sup_{t\ge t_N}
                    \operatorname{MMD}_q(\rho_t^N,\mu)\right]=0.
\end{equation}
\end{corollary}
\begin{proof}
By Proposition~\ref{prop:particle-pathwise-energy-dissipation} and the
triangle inequality, for fixed $S$,
\begin{equation}
 \mathbb E\left[\sup_{t\ge S}\operatorname{MMD}_q(\rho_t^N,\mu)\right]
 =\mathbb E\left[\operatorname{MMD}_q(\rho_S^N,\mu)\right]
 \le\sqrt{\mathbb E[F_N(S)]}+\operatorname{MMD}_q(\rho_S,\mu).
\end{equation}
First let $N\to\infty$ using
Theorem~\ref{thm:particle-empirical-mean-field}, then let $S\to\infty$
using Theorem~\ref{thm:continuum-convergence}.
For every fixed $S$, one has $t_N\ge S$ eventually, so the tail supremum
starting at $t_N$ is bounded by that starting at $S$.
Taking expectations, then the upper limit in $N$, and finally $S\to\infty$
proves \eqref{eq:particle-long-time-diagonal-convergence}.
\end{proof}
These conclusions do not give target convergence at fixed $N$ or a
uniform-in-time mean-field rate. The expectation of a tail supremum is
available here because the discrepancy from the target is pathwise
non-increasing; it does not change the order of supremum and expectation
in the finite-time modulated-energy estimate.

\section{Conclusions and further questions}\label{sec:conclusion}
We have established global well-posedness and target convergence for the
jump-modified energy-kernel flow under the stated density and moment hypotheses.
Among Borel probability measures with finite energy, the target is the only measure with zero jump dissipation. This allows us to identify narrow limits without a uniform-in-time density bound. The particle construction preserves energy
monotonicity pathwise, and the mean-field estimate gives quantitative finite-time
approximation and joint long-time target approximation. These results leave several questions concerning the continuum theory, quantitative convergence rates, and particle discretization.

The finite critical density exponent $p=d/(d+q-2)$ remains outside the
continuum construction when $d+q-2>0$. At this endpoint, the target field
$(-\Delta K)*\mu$ need not be bounded, and the density estimate used here
does not close. Extending the Cauchy theory therefore requires both an
appropriate uniqueness estimate and propagation of the critical density class.

The jump mechanism gives qualitative target convergence, but quantitative
rates beyond the stopped far-field regime remain to be determined. One may
seek a \L{}ojasiewicz-type estimate of the form
\begin{equation}\label{eq:future-local-dissipation}
 \mathcal D(\rho)\ge c\,\mathcal E_q(\rho-\mu)^{1+\alpha},
 \qquad c>0,\quad\alpha\ge0,
\end{equation}
for the total continuum dissipation relative to the fixed target $\mu$.
One must identify a subclass of the admissible densities in
Theorem~\ref{thm:continuum-wellposedness}, with $r\ge2$ and the target,
proposal, and spatial rate fixed, that is invariant after entry or contains
$\rho_t$ for every $t\ge T_0$, and on which $c$ is uniform.
A bound for the entry time $T_0$ is also required. Combined with the energy identity in
Proposition~\ref{prop:energy-dissipation}, such an estimate would
give exponential energy decay for $\alpha=0$ and algebraic energy decay of
order $(t-T_0)^{-1/\alpha}$ as $t\to\infty$ for $\alpha>0$; the MMD rates
follow by taking square roots. The size of the jump dissipation near the
ordinary flow's singular critical measures is one issue for such an estimate.

For pure transport, the work
\cite[Theorems~1.7 and~5.23, Proposition~5.21, and Corollary~5.27]
{rosenzweig2026wassersteingradientflowsmaximum} gives compactness and asymptotic
criticality in its stated regimes. The remaining target-selection question
is to determine which hypotheses exclude non-target singular limits.

A discrete-time particle algorithm should account for the singular transport
force and the exact energy constraint on accepted moves. General approximation
results for piecewise-deterministic processes \cite{bertazzi2022approximations}
provide a starting point, but do not by themselves give an energy-monotone
scheme for this model. A useful objective is a finite-time consistency estimate
with explicit dependence on the time step, particle number, and time horizon.
Energy monotonicity could then be combined with such an estimate for long-time
target approximation; it does not alone justify exchanging all three limits.
Numerical comparisons with transport-only and other MMD methods should use
matched computational costs and independent target samples, including in the
conditional-generation and inverse-problem settings of
\cite{hagemann2024posterior}.

\appendix
\section{Auxiliary results}
\label{app:auxiliary-estimates}
This appendix collects auxiliary results for positive linear jump-transport equations, the jump-term product estimate, and empirical bounded-Lipschitz bounds.
\subsection{Positive linear jump-transport equations}
\label{app:positive-linear}

\begin{lemma}
\label{lem:linear-jump-transport}
Let \(X=L^1(\mathbb R^d)\cap L^p(\mathbb R^d)\), where
\(1<p<\infty\), and let \(T>0\). Suppose
\begin{equation}
\label{eq:linear-coefficient-assumptions}
  v\in L^\infty(0,T;W^{1,\infty}(\mathbb R^d;\mathbb R^d)),
\end{equation}
and let \(J_{t,1},J_{t,2}:X\to X\) be strongly
measurable families of bounded operators, uniformly in \(t\). Assume that
\(J_{t,1}\) is positive and that
\(J_{t,2}\) is multiplication by a function
\(j_2\in L^\infty((0,T)\times\mathbb R^d)\) satisfying
\(j_2\ge0\). Then
\begin{equation}
\label{eq:linear-frozen-jump-transport}
  \partial_tu+\nabla\!\cdot(uv)
  =J_{t,1}u-J_{t,2}u,
  \qquad u(0)=u_0\in X,
\end{equation}
has a unique mild solution \(u\in C([0,T];X)\). If \(u_0\ge0\), then
\(u_t\ge0\). If
\begin{equation}
\label{eq:linear-mass-balance}
  \int_{\mathbb R^d}\left(J_{t,1}h\right)(x)\,dx
  =\int_{\mathbb R^d}\left(J_{t,2}h\right)(x)\,dx
\end{equation}
for every nonnegative \(h\in X\), then the solution preserves its total
mass. Every \(L^\infty(0,T;X)\) distributional solution of
\eqref{eq:linear-frozen-jump-transport} agrees
a.e. in time with this mild solution.
\end{lemma}

\begin{proof}
Let \(U_v(t,s)\) be the transport evolution induced by the globally
bi-Lipschitz flow of \(v\); see, for example,
\cite[Lemma~8.1.4]{AGS08}. Its Jacobian bounds give
\begin{equation}
\label{eq:transport-evolution-bounds}
  \|U_v(t,s)g\|_{L^1}=\|g\|_{L^1},\qquad
  \|U_v(t,s)g\|_{L^p}
  \le e^{\|\nabla\cdot v\|_\infty|t-s|/p_*}\|g\|_{L^p}.
\end{equation}
The mild formulation of \eqref{eq:linear-frozen-jump-transport} is
\begin{equation}
\label{eq:linear-frozen-duhamel-formula}
  u(t)=U_v(t,0)u_0+
  \int_0^tU_v(t,s)
  \left(J_{s,1}u(s)-J_{s,2}u(s)\right)\,ds.
\end{equation}
The bounds \eqref{eq:transport-evolution-bounds} show that
\eqref{eq:linear-frozen-duhamel-formula} is a contraction on sufficiently short time intervals in
\(C([t_0,t_1];X)\), and iteration gives the unique solution on \([0,T]\).
Positivity follows by adding and subtracting \(\|j_{2}\|_{L^\infty((0,T)\times\mathbb R^d)} u\), and applying
Picard iteration to the resulting positive integral equation; indeed,
\(\|j_{2}\|_{L^\infty((0,T)\times\mathbb R^d)}I-J_{t,2}\) is positive. Since \(U_v(t,s)\) preserves the integral, Fubini's theorem in
\eqref{eq:linear-frozen-duhamel-formula} gives
\begin{equation}
  \int_{\mathbb R^d}u(t,x)\,dx
  =\int_{\mathbb R^d}u_0(x)\,dx
   +\int_0^t\int_{\mathbb R^d}
       \bigl(J_{s,1}u(s)-J_{s,2}u(s)\bigr)(x)\,dx\,ds.
\end{equation}
The last integrand vanishes by \eqref{eq:linear-mass-balance}, whose
validity for signed functions follows by linearity from their positive
and negative parts. This proves mass conservation.
Finally, testing the difference of a distributional and a mild solution
against the backward transport equation yields
\eqref{eq:linear-frozen-duhamel-formula};
Gronwall's inequality then gives uniqueness.
\end{proof}

\subsection{A product estimate for the jump term}
\label{app:particle-product-estimate}

As before, let $0<q<2$ with $d+q\ge2$, and set $s_q=(d+q)/2$, as in Section~\ref{subsec:particle-mean-field}. The bounded-Lipschitz distance is
defined in \eqref{eq:particle-BL-distance}; we use the H\"older seminorm
$[u]_{C^\alpha}=\sup_{x\ne y}|u(x)-u(y)|/|x-y|^\alpha$ and the Fourier
convention \eqref{eq:fourier-convention}.

\begin{lemma}
\label{lem:particle-critical-product}
Let \(\nu,\omega\in\mathcal P_q(\R^d)\), \(\vartheta=\nu-\omega\),
and
\begin{equation}
\label{eq:particle-product-assumptions}
  a\in W^{1,\infty}(\R^d),
  \qquad u\in\dot H^{s_q}(\R^d)\cap L^\infty(\R^d)\cap C^{q/2}(\R^d).
\end{equation}
Then
\begin{equation}
\label{eq:particle-critical-product-bound}
\begin{aligned}
  \left|\int_{\R^d}au\,d\vartheta\right|
  \le C_{d,q}\|a\|_{W^{1,\infty}}\bigl(
    \mathcal E_q(\vartheta)^{1/2}\|u\|_{\dot H^{s_q}}
    +d_{\rm BL}(\nu,\omega)
       (\|u\|_{L^\infty}+[u]_{C^{q/2}})\bigr).
\end{aligned}
\end{equation}
\end{lemma}

We use the standard Bony decomposition \cite{bony1981calcul}; the estimates below separate a Sobolev term from a bounded Lipschitz term.

\begin{proof}[Proof of Lemma~\ref{lem:particle-critical-product}]
We decompose \(au\) into a term controlled in \(\dot H^{s_q}\)
and a bounded Lipschitz term controlled by \(d_{\rm BL}\).
We first suppose that \(a\) and \(u\) are smooth; the mollification by $\eta_\epsilon$ used later removes this
restriction.

Since \(\nu,\omega\in\mathcal P_q(\R^d)\), the representation
\eqref{eq:energy} implies \(\mathcal E_q(\vartheta)<\infty\).
Set \(\alpha_q:=q/2\in(0,1)\). Choose a radial function
\(\chi\in C_c^\infty(\R^d)\), with \(0\le\chi\le1\), such that
\(\chi=1\) on \(B(0,3/4)\) and
\(\operatorname{supp}\chi\subset B(0,4/3)\), and set
\(\phi(\xi):=\chi(\xi)-\chi(2\xi)\). Then
\begin{equation}
  \chi(\xi)+\sum_{j\ge1}\phi(2^{-j}\xi)=1.
\end{equation}
Define the Fourier multipliers
\begin{subequations}\label{eq:particle-critical-LP-decomposition}
\begin{align}
  \widehat{\Delta_0u}(\xi)&:=\chi(\xi)\widehat u(\xi),
  &\widehat{\Delta_ju}(\xi)&:=\phi(2^{-j}\xi)\widehat u(\xi),
  \quad j\ge1,\\
  P_{\le0}&:=\Delta_0,
  &P_{>0}&:=I-P_{\le0},\\
  S_{j-1}&:=\sum_{0\le k\le j-2}\Delta_k,
  &S_{j-1}^+&:=\sum_{1\le k\le j-2}\Delta_k,
\end{align}
\end{subequations}
where empty sums are zero. Set
\begin{equation}
\label{eq:particle-critical-low-high-functions}
  u_{\rm lo}:=\Delta_0u,
  \qquad
  u_{\rm hi}:=\sum_{j\ge1}\Delta_ju=P_{>0}u.
\end{equation}
The operators \(P_{\le0}\) and \(\nabla P_{\le0}\) have fixed
\(L^1\) convolution kernels. Young's inequality, the product rule,
and the \(L^2\) Fourier-multiplier bound give
\begin{subequations}\label{eq:particle-critical-frequency-splitting}
\begin{align}
  \|u_{\rm hi}\|_{L^\infty}+[u_{\rm hi}]_{C^{\alpha_q}}
  &\le C_{d,q}
       \left(\|u\|_{L^\infty}+[u]_{C^{\alpha_q}}\right),\\
  \|u_{\rm hi}\|_{\dot H^{s_q}}
  &\le C_{d,q}\|u\|_{\dot H^{s_q}},\\
  \|a u_{\rm lo}\|_{L^\infty}
  +\operatorname{Lip}(a u_{\rm lo})
  &\le C_{d,q}\|a\|_{W^{1,\infty}}\|u\|_{L^\infty}.
\end{align}
\end{subequations}
The same convolution bounds yield
\begin{equation}
\label{eq:particle-critical-low-output-bounds}
  \|P_{\le0}(a u_{\rm hi})\|_{L^\infty}
  +\operatorname{Lip}\bigl(P_{\le0}(a u_{\rm hi})\bigr)
  \le C\|a\|_{L^\infty}\|u_{\rm hi}\|_{L^\infty},
\end{equation}
and, since \(P_{>0}=I-P_{\le0}\),
\begin{equation}
\label{eq:particle-critical-high-output-bounds}
  \|P_{>0}h\|_{L^\infty}+\operatorname{Lip}(P_{>0}h)
  \le C\left(\|h\|_{L^\infty}+\operatorname{Lip}(h)\right)
\end{equation}
for every bounded Lipschitz function \(h\).

Grouping the terms in
\(a u_{\rm hi}=(\sum_{j\ge0}\Delta_ja)(\sum_{k\ge1}\Delta_ku)\)
according to whether \(j\le k-2\), \(k\le j-2\), or
\(|j-k|\le1\), we obtain
\begin{equation}
\label{eq:particle-critical-Bony-decomposition}
  a u_{\rm hi}=T_a u_{\rm hi}+T_{u_{\rm hi}}a+R(a,u_{\rm hi}),
\end{equation}
where
\begin{equation}
\label{eq:particle-critical-paraproducts}
  T_a u_{\rm hi}:=\sum_{j\ge1}S_{j-1}a\,\Delta_j u,
  \quad
  T_{u_{\rm hi}}a:=\sum_{j\ge1}S_{j-1}^+u\,\Delta_j a, \quad   R(a,u_{\rm hi})
  :=\sum_{k\ge1}\sum_{\substack{j\ge0\\|j-k|\le1}}
    \Delta_j a\,\Delta_k u.
\end{equation}
Define
\begin{equation}
\label{eq:particle-critical-product-decomposition}
\begin{aligned}
  G:=P_{>0}\left(T_a u_{\rm hi}+R(a,u_{\rm hi})\right),\quad H:=a u_{\rm lo}+P_{\le0}(a u_{\rm hi})
      +P_{>0}T_{u_{\rm hi}}a,
\end{aligned}
\end{equation}
so that \(au=G+H\). 

For the chosen cutoffs, each nonzero term
\(S_{j-1}a\,\Delta_j u\) has Fourier support in an annulus
of size \(2^j\), and these annuli have bounded overlap.
Since \(\|S_{j-1}a\|_\infty\le C\|a\|_\infty\),
Plancherel's theorem gives
\begin{equation}
\label{eq:particle-critical-high-low-bound}
  \|P_{>0}T_a u_{\rm hi}\|_{\dot H^{s_q}}
  \le C_{d,q}\|a\|_{L^\infty}\|u\|_{\dot H^{s_q}}.
\end{equation}
The cancellation of the kernels of \(\Delta_j\), \(j\ge1\),
also gives
\begin{equation}
\label{eq:particle-critical-coefficient-block-bound}
  \|\Delta_j a\|_{L^\infty}
  \le C2^{-j}\|a\|_{W^{1,\infty}},
  \qquad j\ge0,
\end{equation}
where \(j=0\) follows from the \(L^\infty\) convolution bound.
In the remainder, \(\Delta_\ell(\Delta_j a\,\Delta_k u)\)
vanishes unless \(k\ge\ell-C_0\), for a fixed integer \(C_0\).
Consequently, for \(\ell\ge1\),
\begin{align}
  2^{s_q\ell}\|\Delta_\ell R(a,u_{\rm hi})\|_{L^2}
  &\le C\|a\|_{W^{1,\infty}}
       \sum_{k\ge1\vee(\ell-C_0)}
       2^{s_q\ell-k}\|\Delta_k u\|_{L^2}\nonumber \\
  &=C\|a\|_{W^{1,\infty}}2^{-\ell}
       \sum_{m\ge(1-\ell)\vee(-C_0)}
       2^{-(s_q+1)m}
       2^{s_q(\ell+m)}\|\Delta_{\ell+m}u\|_{L^2}.\label{eq:particle-critical-remainder-summability}
\end{align}
Since \(s_q>0\), the sequence
\((2^{-(s_q+1)m})_{m\ge-C_0}\) belongs to \(\ell^1\).
Taking the \(\ell^2\)-norm in \(\ell\), applying Young's
convolution inequality on \(\mathbb Z\), and using the
Littlewood--Paley characterization of \(\dot H^{s_q}\), we obtain
\begin{equation}
\label{eq:particle-critical-remainder-bound}
  \|P_{>0}R(a,u_{\rm hi})\|_{\dot H^{s_q}}
  \le C_{d,q}\|a\|_{W^{1,\infty}}\|u\|_{\dot H^{s_q}}.
\end{equation}
It remains to estimate \(T_{u_{\rm hi}}a\).
For the chosen cutoffs,
\(\Delta_k u=\Delta_k u_{\rm hi}\) for \(k\ge2\), while
\(\Delta_1u=\chi(D/2)u_{\rm hi}\), where \(\chi(D/2)\)
denotes the multiplier with symbol \(\chi(\xi/2)\).
The convolution bounds and cancellation of the dyadic kernels
therefore give, for \(j\ge1\),
\begin{subequations}\label{eq:particle-critical-low-high-input-bounds}
\begin{align}
  \|\Delta_j a\|_{L^\infty}
  &\le C2^{-j}\|\nabla a\|_{L^\infty},\\
  \|S_{j-1}^+u\|_{L^\infty}
  &\le C\|u_{\rm hi}\|_{L^\infty},\\
  \|u_{\rm hi}-S_{j-1}^+u\|_{L^\infty}
  &\le C_{d,q}2^{-\alpha_qj}[u_{\rm hi}]_{C^{\alpha_q}},\\
  \|\nabla S_{j-1}^+u\|_{L^\infty}
  &\le C_{d,q}2^{(1-\alpha_q)j}[u_{\rm hi}]_{C^{\alpha_q}}.
\end{align}
\end{subequations}
Taking norms and regrouping the distributional derivative gives
\begin{align}
  \|T_{u_{\rm hi}}a\|_{W^{1,\infty}}
  &\le \sum_{j\ge1}\Bigl(
    \|S_{j-1}^+u\|_\infty\|\Delta_j a\|_\infty
    +\|\nabla S_{j-1}^+u\|_\infty\|\Delta_j a\|_\infty
  \Bigr)\nonumber \\
  &\quad+\sum_{j\ge1}
    \|S_{j-1}^+u-u_{\rm hi}\|_\infty
    \|\nabla\Delta_j a\|_\infty
    +\|u_{\rm hi}\|_\infty
    \|\nabla(a-\Delta_0a)\|_\infty\nonumber \\
  &\le C_{d,q}\|\nabla a\|_\infty
  \left(\|u_{\rm hi}\|_\infty+[u_{\rm hi}]_{C^{\alpha_q}}\right).
\end{align}
Here, we used that \(0<\alpha_q<1\). Combining the preceding
estimates gives
\begin{subequations}\label{eq:particle-critical-product-components}
\begin{align}
  \|G\|_{\dot H^{s_q}}
  &\le C_{d,q}\|a\|_{W^{1,\infty}}
       \|u\|_{\dot H^{s_q}},\\
  \|H\|_{L^\infty}+\operatorname{Lip}(H)
  &\le C_{d,q}\|a\|_{W^{1,\infty}}
       \left(\|u\|_{L^\infty}+[u]_{C^{\alpha_q}}\right).
\end{align}
\end{subequations}

For \(\delta>0\), define \(\nu_\delta:=\eta_\delta*\nu\), \(\omega_\delta:=\eta_\delta*\omega\), and \(\vartheta_\delta:=\nu_\delta-\omega_\delta\).
Then \(\vartheta_\delta\in L^1(\R^d)\cap L^2(\R^d)\).  Since \(G\) in
\eqref{eq:particle-critical-product-decomposition} has Fourier support
bounded away from zero, \eqref{eq:extended-fourier-energy} and
Plancherel's theorem give
\begin{equation}
\label{eq:particle-critical-high-pairing}
  \left|\int_{\R^d}G\,d\vartheta_\delta\right|
  \le C_{d,q}\mathcal E_q(\vartheta_\delta)^{1/2}
       \|G\|_{\dot H^{s_q}},
\end{equation}
while \eqref{eq:particle-BL-distance} gives
\begin{equation}
\label{eq:particle-critical-low-pairing}
  \left|\int_{\R^d}H\,d\vartheta_\delta\right|
  \le d_{\rm BL}(\nu_\delta,\omega_\delta)
       \left(\|H\|_{L^\infty}+\operatorname{Lip}(H)\right).
\end{equation}
For \(\varepsilon>0\), set \(a_\varepsilon:=\eta_\varepsilon*a\) and \(u_\varepsilon:=\eta_\varepsilon*u\).
Convolution does not increase \(\|a\|_{W^{1,\infty}}\),
\(\|u\|_{L^\infty}\), \([u]_{C^{\alpha_q}}\), or
\(\|u\|_{\dot H^{s_q}}\); moreover,
\(a_\varepsilon u_\varepsilon\to au\) uniformly and
\(u_\varepsilon\to u\) in \(\dot H^{s_q}\).  Apply $au=G+H$,
\eqref{eq:particle-critical-product-components},
\eqref{eq:particle-critical-high-pairing}, and
\eqref{eq:particle-critical-low-pairing} to
\((a_\varepsilon,u_\varepsilon,\vartheta_\delta)\), and let
\(\varepsilon\downarrow0\).  The limit is
\begin{equation}
\label{eq:particle-critical-mollified-product-bound}
\begin{aligned}
  \left|\int_{\R^d}au\,d\vartheta_\delta\right|
  \le C_{d,q}\|a\|_{W^{1,\infty}}
  \bigg(&\mathcal E_q(\vartheta_\delta)^{1/2}
          \|u\|_{\dot H^{s_q}}+d_{\rm BL}(\nu_\delta,\omega_\delta)
          \left(\|u\|_{L^\infty}+[u]_{C^{\alpha_q}}\right)
  \bigg).
\end{aligned}
\end{equation}
If \(\varphi\) is admissible in \eqref{eq:particle-BL-distance}, then
\(\eta_\delta*\varphi\) is also admissible and hence \(  d_{\rm BL}(\nu_\delta,\omega_\delta)
  \le d_{\rm BL}(\nu,\omega)\).
Dominated convergence in \eqref{eq:extended-fourier-energy} gives \(  \mathcal E_q(\vartheta_\delta)\to
  \mathcal E_q(\vartheta)\).
Finally, \(au\) is bounded and uniformly continuous, so
\(\int au\,d\vartheta_\delta\to\int au\,d\vartheta\).  Letting
\(\delta\downarrow0\) in
\eqref{eq:particle-critical-mollified-product-bound} thus proves
\eqref{eq:particle-critical-product-bound}.
\end{proof}

\subsection{Empirical bounded-Lipschitz estimates}
\label{app:particle-iid-BL-estimate}

As in Section~\ref{subsec:particle-mean-field}, let $r\ge\max\{1,q\}$ and $N\ge2$. The reader will recall the bounded-Lipschitz distance and the dimensional sampling factor from \eqref{eq:particle-BL-distance} and \eqref{eq:particle-BL-dimensional-rate}, respectively.

\begin{lemma}
\label{lem:particle-iid-BL}
Let \(Y_1,\ldots,Y_N\) be independent with common law
\(\nu\in\mathcal P_r(\R^d)\), and suppose
\(\mathcal M_r(\nu)\le\mathfrak m\). With $\mathfrak n_d(N)$ from
\eqref{eq:particle-BL-dimensional-rate}, one has
\begin{equation}
\label{eq:particle-iid-BL-rate}
  \mathbb E\left[d_{\rm BL}\left(\frac1N\sum_{i=1}^N\delta_{Y_i},\nu\right)^2\right]
  \le C_{d,r,\mathfrak m}\mathfrak n_d(N)^{r/(r+1)}.
\end{equation}
\end{lemma}

On a ball, smoothing and independence give a Fourier variance estimate. Moment truncation
then removes the bounded-support assumption.

\begin{proof}[Proof of Lemma~\ref{lem:particle-iid-BL}]
For \(N=1\), \eqref{eq:particle-iid-BL-rate} follows from
\(d_{\rm BL}\le2\).  We may therefore assume \(N\ge2\).
Write \(  \overline\nu^N:=\frac1N\sum_{i=1}^N\delta_{Y_i}\).
Assume first that
\(\nu(B_R)=1\) for some \(R\ge1\), and set \(\sigma_N:=\overline\nu^N-\nu\),  \(\sigma_{N,h}:=\eta_h*\sigma_N\) for $0<h\leq R$. If
\(\varphi\) is admissible in \eqref{eq:particle-BL-distance}, then
\begin{equation}
\label{eq:particle-BL-mollification-error}
  \left|\int_{\R^d}\varphi\,d\sigma_N
  -\int_{\R^d}\varphi(x)\sigma_{N,h}(x)\,dx\right|
  \le C_\eta h.
\end{equation}
Indeed, the difference on the left-hand side of
\eqref{eq:particle-BL-mollification-error} equals
\(\int(\varphi-\eta_h*\varphi)\,d\sigma_N\).  The inequalities
\(\operatorname{Lip}(\varphi)\le1\) and
\(|\sigma_N|(\R^d)\le2\) imply
\eqref{eq:particle-BL-mollification-error}.

Choose \(\chi_R\in C_c^\infty(B_{3R})\) such that
\(\chi_R=1\) on \(B_{2R}\), \(0\le\chi_R\le1\), and
\(\|\nabla\chi_R\|_{L^\infty}\le C/R\).  The support condition on
\(\eta\) in \eqref{eq:kernel-regularization} implies that
\(\operatorname{supp}\sigma_{N,h}\subset B_{2R}\).  Moreover,
\begin{equation}
\label{eq:particle-BL-cutoff-gradient-bound}
  \|\nabla(\chi_R\varphi)\|_{L^2}
  \le C_dR^{d/2}.
\end{equation}
The function \(\chi_R\varphi\) belongs to \(\dot H^1(\R^d)\), with
\eqref{eq:particle-BL-cutoff-gradient-bound}.  Since
\(\sigma_{N,h}(\R^d)=0\), a duality argument,
first for a smooth approximation of \(\chi_R\varphi\) and then by
passage to the limit, gives
\begin{equation}
\label{eq:particle-BL-Hminus1-pairing}
  \left|\int_{\R^d}\varphi(x)\sigma_{N,h}(x)\,dx\right|
  \le C_dR^{d/2}\|\sigma_{N,h}\|_{\dot H^{-1}}.
\end{equation}
Taking the supremum over the functions in
\eqref{eq:particle-BL-distance}, then using
\eqref{eq:particle-BL-mollification-error} and
\eqref{eq:particle-BL-Hminus1-pairing}, yields
\begin{equation}
\label{eq:particle-BL-Hminus1-reduction}
  \mathbb E\left[d_{\rm BL}(\overline\nu^N,\nu)^2\right]
  \le C_\eta h^2
      +C_dR^d\mathbb E\left[\|\sigma_{N,h}\|_{\dot H^{-1}}^2\right].
\end{equation}

The independence of the variables in the definition of $\overline{\nu}^N$ implies
\begin{equation}
\begin{aligned}
  \mathbb E\left[|\widehat\sigma_N(\xi)|^2\right]
  =\frac{1-|\widehat\nu(\xi)|^2}{N}&=\frac1{2N}\iint_{\R^d\times\R^d}
    |e^{-2\pi i\xi\cdot x}-e^{-2\pi i\xi\cdot y}|^2
    \,d\nu(x)\,d\nu(y)\le\frac{C}{N}(1\wedge R^2|\xi|^2).
\end{aligned}
\end{equation}
Consequently,
\begin{equation}
\label{eq:particle-BL-Hminus1-variance}
  \mathbb E\left[\|\sigma_{N,h}\|_{\dot H^{-1}}^2\right]
  \le\frac CN\int_{\R^d}|\xi|^{-2}
  (1\wedge R^2|\xi|^2)
  |\widehat\eta(h\xi)|^2\,d\xi.
\end{equation}
Let \(J\ge0\) be the integer satisfying
\(2^J\le R/h<2^{J+1}\).  Split the integral in
\eqref{eq:particle-BL-Hminus1-variance} into the ball
\(\{|\xi|\le R^{-1}\}\), the annuli
\(\{2^jR^{-1}<|\xi|\le2^{j+1}R^{-1}\}\) for \(0\le j\le J\), and
the annuli with indices \(j\ge J+1\).  The ball
\(\{|\xi|\le R^{-1}\}\) contributes at most \(CR^{2-d}\), and the
\(j\)-th annulus contributes at most
\(CR^{2-d}2^{j(d-2)}\) for \(0\le j\le J\).  The rapid decay of
\(\widehat\eta\) bounds the sum over \(j\ge J+1\) by a constant times the
bound at \(j=J\).  Thus
\begin{equation}
\label{eq:particle-BL-dyadic-frequency-sum}
  \mathbb E\left[\|\sigma_{N,h}\|_{\dot H^{-1}}^2\right]
  \le\frac{C_dR^{2-d}}N
  \begin{cases}
    1,&d=1,\\
    1+\log(R/h),&d=2,\\
    (R/h)^{d-2},&d\ge3.
  \end{cases}
\end{equation}
Choosing $h=RN^{-1/2}$ when $d\in \{1,2\}$ and $h=RN^{-1/d}$ when $d\geq 3$ thus proves
\begin{equation}
\label{eq:particle-BL-bounded-support-rate}
  \mathbb E\left[d_{\rm BL}(\overline\nu^N,\nu)^2\right]
  \le C_dR^2\mathfrak n_d(N)
  \qquad\text{if }\nu(B_R)=1.
\end{equation}

Now, for a general \(\nu\in\mathcal P_r(\R^d)\), define
\begin{equation}
\label{eq:particle-BL-tail-projection}
  \Pi_R(x):=\begin{cases}x,&|x|\le R,\\0,&|x|>R,\end{cases}
  \qquad
  \nu_R:=(\Pi_R)_\#\nu,
  \qquad
  \overline\nu_R^N:=(\Pi_R)_\#\overline\nu^N.
\end{equation}
Every function admissible in \eqref{eq:particle-BL-distance} has
\(L^\infty\)-norm at most one.  The triangle inequality and
\eqref{eq:particle-BL-tail-projection} therefore give
\begin{equation}
\label{eq:particle-BL-truncation-bound}
  d_{\rm BL}(\overline\nu^N,\nu)
  \le d_{\rm BL}(\overline\nu_R^N,\nu_R)
      +2\overline\nu^N(B_R^c)+2\nu(B_R^c).
\end{equation}
Set \(p_R:=\nu(B_R^c)\).  The inequality
\(\mathcal M_r(\nu)\le\mathfrak m\) and the binomial variance identity
imply
\begin{equation}
\label{eq:particle-BL-tail-second-moment}
  p_R\le\mathfrak mR^{-r},
  \qquad
  \mathbb E\left[\overline\nu^N(B_R^c)^2\right]
  =p_R^2+\frac{p_R(1-p_R)}N.
\end{equation}
Applying \eqref{eq:particle-BL-bounded-support-rate} to
\((\overline\nu_R^N,\nu_R)\), and then using
\eqref{eq:particle-BL-truncation-bound} and
\eqref{eq:particle-BL-tail-second-moment}, yields
\begin{equation}
\label{eq:particle-BL-preoptimization}
  \mathbb E\left[d_{\rm BL}(\overline\nu^N,\nu)^2\right]
  \le C_{d,r,\mathfrak m}\left(
    R^2\mathfrak n_d(N)+R^{-2r}+N^{-1}R^{-r}
  \right).
\end{equation}
Since \(N\ge2\), \( N^{-1}\le\mathfrak n_d(N)\le1\).
Choosing \(R:=\mathfrak n_d(N)^{-1/(2(r+1))}\)
gives
\begin{equation}
\label{eq:particle-BL-principal-optimization}
  R^2\mathfrak n_d(N)+R^{-2r}
  =2\mathfrak n_d(N)^{r/(r+1)},
\end{equation}
and therefore
\begin{equation}
\label{eq:particle-BL-tail-optimization}
  N^{-1}R^{-r}
  \le\mathfrak n_d(N)^{1+r/(2(r+1))}
  \le\mathfrak n_d(N)^{r/(r+1)}.
\end{equation}
Substitution of \eqref{eq:particle-BL-principal-optimization}--\eqref{eq:particle-BL-tail-optimization} into
\eqref{eq:particle-BL-preoptimization} proves
\eqref{eq:particle-iid-BL-rate}.
\end{proof}

\section{Failure of an energy-only estimate for the jump term}
\label{app:particle-jump-obstruction}

The continuum positive-part acceptance rule \eqref{eq:rate}, with the
source frozen at $\rho_0$, can create a corner even when the first variation $f_{\rho_0}$ is smooth. In the example below, the second derivative of the
frozen jump operator applied to a smooth function contains a measure
supported on a plane. The energy controls the second derivatives of a
perturbation potential in \(L^2\), but does not control their restriction
to this plane. We first use this distinction to construct smooth
perturbations $\nu_n$ of a fixed reference density $\rho_0$, with
$\nu_n-\rho_0\in C_c^\infty(\R^3)$ and the target $\mu$ unchanged, as in
Lemma~\ref{lem:particle-obstruction-smooth}. Their jump remainder is
unbounded relative to their energy, and we then pass to distinct particle
configurations.

Throughout this appendix, we take
\begin{equation}
\label{eq:particle-obstruction-parameters}
  d=3,\qquad q=1,\qquad \lambda\equiv1,\qquad
  \kappa(z)=(2\pi)^{-3/2}e^{-|z|^2/2}.
\end{equation}
For fixed \(\rho_0,\mu\in\mathcal P_1(\R^3)\), write
\(\zeta_0(x,z)=\zeta_{\rho_0}(x,x+z)\), with the first variation defined by
\eqref{eq:witness}. For a finite signed measure \(\vartheta\) of zero
mass and finite first absolute moment, define
\begin{equation}
\label{eq:particle-jump-remainder-functional}
  \mathcal R_{\rho_0}(\vartheta)
  :=\int_{\R^3}\int_{\R^3}
  \bigl((K*\vartheta)(x+z)-(K*\vartheta)(x)\bigr)
  \zeta_0(x,z)\kappa(z)\,dz\,d\vartheta(x).
\end{equation}
The target \(\mu\) is fixed and is suppressed from the notation, as in
\(f_{\rho_0}\). Since \(K=-|\cdot|\) is \(1\)-Lipschitz, subtracting
\(K(x)\) in the convolution with \(\rho_0-\mu\) bounds
\(f_{\rho_0}\) by \(\mathcal M_1(\rho_0)+\mathcal M_1(\mu)\).
Thus \(\zeta_0\) is bounded. The same Lipschitz property bounds the
potential increment in \eqref{eq:particle-jump-remainder-functional}
by \(|z|\|\vartheta\|_{\mathrm{TV}}\), so this integral is absolutely
convergent. For \(\vartheta=\vartheta_0^N\), it is precisely the
remainder in \eqref{eq:particle-jump-monotonicity-splitting} at \(t=0\).

\begin{lemma}
\label{lem:particle-obstruction-smooth}
There exist strictly positive Schwartz probability densities
\(\rho_0,\mu\) and a sequence of strictly positive Schwartz probability
densities \(\nu_n\) such that \(\nu_n-\rho_0\in C_c^\infty(\R^3)\),
\begin{equation}
\label{eq:particle-obstruction-smooth-densities}
  0<\mathcal E_1(\nu_n-\rho_0)\to0,
  \qquad
  \frac{\mathcal R_{\rho_0}(\nu_n-\rho_0)}
       {\mathcal E_1(\nu_n-\rho_0)}\to\infty.
\end{equation}
\end{lemma}
\begin{proof}
The distributional identities \(\Delta|x|=2|x|^{-1}\) and
\(\Delta(|x|^{-1})=-4\pi\delta_0\) in \(\R^3\) give
\(\Delta^2K=8\pi\delta_0\). Consequently, for
\(u\in C_c^\infty(\R^3)\), convolution and integration by parts give
\begin{equation}
\label{eq:particle-obstruction-potential-energy}
  K*\frac{\Delta^2u}{8\pi}=u,
  \qquad
  \mathcal E_1\!\left(\frac{\Delta^2u}{8\pi}\right)
  =\frac1{16\pi}\int_{\R^3}|\Delta u|^2\,dx.
\end{equation}
In particular, prescribing a compactly supported potential produces a
compactly supported perturbation of zero mass. Choose \(G\in C_c^\infty(\R^3)\) equal to \(x_1\) on a ball \(U\)
centered at the origin, and put \(\chi=(8\pi)^{-1}\Delta^2G\).
Since \(\chi\) has zero integral and compact support, a sufficiently
small \(\varepsilon>0\) makes
\(\rho_0:=\kappa+\varepsilon\chi\) a strictly positive Schwartz
probability density. Take \(\mu:=\kappa\). Applying
\eqref{eq:particle-obstruction-potential-energy} to \(G\) yields
\begin{equation}
\label{eq:particle-obstruction-witness}
  f_{\rho_0}=\varepsilon G,
  \qquad
  \zeta_0(x,z)=\varepsilon[G(x)-G(x+z)]_+.
\end{equation}
These data satisfy the hypotheses of
Theorem~\ref{thm:continuum-wellposedness} for every admissible exponent
\(p\), and have moments of every order.

For the duration of this proof, denote the frozen density operator from
\eqref{eq:frozen-inflow-outflow-operators} by
\(J_0:=J_{\rho_0,\mathrm{in}}-J_{\rho_0,\mathrm{out}}\), and define
the corresponding operator on test functions by
\begin{equation}
  L_0\varphi(x)
  :=\int_{\R^3}\zeta_0(x,z)
       \bigl(\varphi(x+z)-\varphi(x)\bigr)\kappa(z)\,dz.
\end{equation}
Changing variables in the gain term gives
\(\int\varphi J_0v=\int L_0\varphi\,v\) for compactly supported
smooth \(\varphi,v\). It is easy to see by an application of Young's inequality that \(\|J_0v\|_{L^2}\le2\|\zeta_0\|_\infty\|v\|_{L^2}\).

Choose \(0\le\varphi\le1\) in \(C_c^\infty(\R^3)\), equal to one
on a neighborhood of \(\operatorname{supp}G\). For \(x\in U\),
\(\varphi(x)=1\) and \(G(x)=x_1\). At a destination \(y\) where
\(\varphi(y)-1\ne0\), one has \(G(y)=0\). Thus, after changing
variables to \(y=x+z\), only the positive part of \(x_1\) remains in
the acceptance rate, and
\begin{equation}
\label{eq:particle-obstruction-kink}
  L_0\varphi(x)=-\varepsilon(x_1)_+b(x),
  \qquad
  b(x):=\int_{\R^3}(1-\varphi(y))\kappa(y-x)\,dy,
  \qquad x\in U.
\end{equation}
The Gaussian is positive everywhere and \(\varphi\) has compact
support, so \(b\) is smooth and strictly positive. Differentiating
\((x_1)_+b\) twice, the derivative of
\(\mathbf1_{\{x_1>0\}}\) contributes surface measure on
\(\{x_1=0\}\). With \(x'=(x_2,x_3)\), this gives, in
\(\mathcal D'(U)\),
\begin{equation}
\label{eq:particle-obstruction-kink-laplacian}
  \Delta L_0\varphi
  =-\varepsilon b(0,x')\delta_{\{x_1=0\}}
   -\varepsilon\mathbf1_{\{x_1>0\}}
        \bigl(2\partial_1b+x_1\Delta b\bigr).
\end{equation}
Here \(\delta_{\{x_1=0\}}\) denotes integration over the plane with
respect to \(dx'\). This surface term is the source of the obstruction.

To detect it, choose \(h\in C_c^\infty(\R)\) with \(h''(0)=-1\)
and a nonnegative, nonzero \(\psi\in C_c^\infty(\R^2)\), with
supports contained in a sufficiently small cylinder about the origin.
Set
\begin{equation}
  u_n(x):=n^{-3/2}h(nx_1)\psi(x'),
  \qquad w_n:=\frac{\Delta^2u_n}{8\pi},
  \qquad \theta:=\frac{\Delta^2\varphi}{8\pi}.
\end{equation}
Then \(u_n\) and \(w_n\) are supported in a fixed compact subset of
\(U\) for every \(n\ge1\).
Each derivative in \(x_1\) contributes a factor \(n\), while changing
variables to \(s=nx_1\) contributes \(n^{-1/2}\) to the \(L^2\)
norm. Derivatives in \(x'\) do not change these factors. Applying
this observation to derivatives of orders zero, two, and four gives
\begin{equation}
\label{eq:particle-obstruction-potential-scaling}
  \|u_n\|_{L^2}\le Cn^{-2},
  \qquad \|\Delta u_n\|_{L^2}\le C,
  \qquad \|w_n\|_{L^2}\le Cn^2.
\end{equation}
The constants in this proof are independent of \(n\). Since
\(\varphi=1\) on \(U\), the supports of \(\Delta\varphi\) and
\(\Delta u_n\) are disjoint. Using
\eqref{eq:particle-obstruction-potential-energy}, we obtain
\begin{equation}
\label{eq:particle-obstruction-bounded-energy}
  0<\mathcal E_1(\theta)
  \le\mathcal E_1(\theta+w_n)
  =\frac1{16\pi}
       \bigl(\|\Delta\varphi\|_{L^2}^2
             +\|\Delta u_n\|_{L^2}^2\bigr)
  \le C.
\end{equation}

The remainder pairs this bounded bulk energy with the surface term in
\eqref{eq:particle-obstruction-kink-laplacian}. Indeed, duality and
the definition of \(w_n\) give
\begin{equation}
  \int_{\R^3}\varphi J_0w_n\,dx
  =\frac1{8\pi}\langle\Delta L_0\varphi,\Delta u_n\rangle.
\end{equation}
On the plane, direct differentiation yields
\begin{equation}
  \Delta u_n(0,x')
  =-n^{1/2}\psi(x')+n^{-3/2}h(0)\Delta_{x'}\psi(x').
\end{equation}
The surface term in \eqref{eq:particle-obstruction-kink-laplacian}
therefore contributes
\(\frac{\varepsilon}{8\pi}n^{1/2}\int b(0,x')\psi(x')\,dx'
+O(n^{-3/2})\). The remaining coefficient in that identity is
bounded on the fixed support of \(u_n\), so its pairing with
\(\Delta u_n\) is bounded by
\eqref{eq:particle-obstruction-potential-scaling}. It follows that
\begin{equation}
\label{eq:particle-obstruction-growing-pairing}
  \int_{\R^3}\varphi J_0w_n\,dx
  =c\sqrt n+O(1),
  \qquad
  c:=\frac{\varepsilon}{8\pi}
      \int_{\R^2}b(0,x')\psi(x')\,dx'>0.
\end{equation}
Since \(K*(\theta+w_n)=\varphi+u_n\), Fubini's theorem rewrites the
remainder as
\begin{align}
  \mathcal R_{\rho_0}(\theta+w_n)
  &=\int_{\R^3}(\varphi+u_n)J_0(\theta+w_n)\,dx\nonumber \\
  &=\int_{\R^3}\varphi J_0\theta\,dx+\int_{\R^3}\varphi J_0w_n\,dx
    +\int_{\R^3} u_nJ_0\theta\,dx+\int_{\R^3} u_nJ_0w_n\,dx.
\end{align}
The first term is fixed. The boundedness of $J_0$ and
\eqref{eq:particle-obstruction-potential-scaling} bound the last two terms by \(Cn^{-2}\) and \(C\), respectively. Combining
these bounds with \eqref{eq:particle-obstruction-growing-pairing}
gives
\begin{equation}
\label{eq:particle-obstruction-growing-remainder}
  \mathcal R_{\rho_0}(\theta+w_n)=c\sqrt n+O(1)
  \to\infty.
\end{equation}

Finally, all the functions \(\vartheta_n:=\theta+w_n\) have support
in one compact set \(C\), on which \(\rho_0\) has a positive
minimum. Choose \(a_n>0\) such that \(a_n\le n^{-1}\) and
\(a_n\|\vartheta_n\|_\infty\le\frac12\min_C\rho_0\).
Then \(\nu_n:=\rho_0+a_n\vartheta_n\) is a strictly positive
Schwartz probability density. Both \(\mathcal E_1\) and
\(\mathcal R_{\rho_0}\) are quadratic in their signed-measure
argument. Thus,
\begin{equation}
  0<\mathcal E_1(\nu_n-\rho_0)\le Cn^{-2},
  \qquad
  \frac{\mathcal R_{\rho_0}(\nu_n-\rho_0)}
       {\mathcal E_1(\nu_n-\rho_0)}
  =\frac{\mathcal R_{\rho_0}(\vartheta_n)}
        {\mathcal E_1(\vartheta_n)}\to\infty.
\end{equation}
\end{proof}

\begin{lemma}
\label{lem:particle-obstruction-approximation}
Fix \(\rho_0,\mu\in\mathcal P_1(\R^3)\). For every probability
density \(\nu\in\mathcal P_1(\R^3)\), there are configurations
\(X^{(N)}\in\mathsf D_N\) such that
\begin{equation}
\label{eq:particle-obstruction-approximation}
\begin{aligned}
  \mathcal E_1(\rho_{X^{(N)}}^N-\rho_0)
    \to\mathcal E_1(\nu-\rho_0),\quad
  \mathcal R_{\rho_0}(\rho_{X^{(N)}}^N-\rho_0)
    \to\mathcal R_{\rho_0}(\nu-\rho_0).
\end{aligned}
\end{equation}
\end{lemma}
\begin{proof}
We first show that approximation in \(W_1\) preserves the two
quantities in the statement.

Let \(\eta\in\mathcal P_1(\R^3)\) and
\(\pi\in\Pi(\eta,\nu)\). Since \(K\) is even, expanding the
difference of the quadratic energies and using the marginals
of \(\pi\) gives
\begin{equation}
\begin{aligned}
  &\mathcal E_1(\eta-\rho_0)-\mathcal E_1(\nu-\rho_0)
  =\frac12\int_{\R^3\times\R^3}\int_{\R^3}
     \bigl(K(x-z)-K(y-z)\bigr)
     \,d(\eta+\nu-2\rho_0)(z)\,d\pi(x,y).
\end{aligned}
\end{equation}
The \(1\)-Lipschitz property of \(K\) bounds the kernel difference
by \(|x-y|\), while
\(\|\eta+\nu-2\rho_0\|_{\mathrm{TV}}\le4\). Taking the infimum over \(\pi\) yields
\begin{equation}
\label{eq:particle-obstruction-energy-continuity}
  \bigl|\mathcal E_1(\eta-\rho_0)
        -\mathcal E_1(\nu-\rho_0)\bigr|
  \le2W_1(\eta,\nu).
\end{equation}
For the remainder, the increment of \(K\) removes its growth at
infinity. More precisely, define
\begin{equation}
  H(x,y):=\int_{\R^3}\zeta_0(x,z)
       \bigl(K(x+z-y)-K(x-y)\bigr)\kappa(z)\,dz.
\end{equation}
Recall that \(f_{\rho_0}\) is bounded and continuous, and hence
so is \(\zeta_0\). The integrand defining \(H\) is continuous in
\((x,y)\) and satisfies
\begin{equation}
  \left|
  \zeta_0(x,z)
  \bigl(K(x+z-y)-K(x-y)\bigr)\kappa(z)
  \right|
  \le \|\zeta_0\|_\infty |z|\kappa(z).
\end{equation}
The right-hand side is integrable. Dominated convergence therefore
gives \(H\in C_b(\R^3\times\R^3)\), with
\(\|H\|_\infty\le
\|\zeta_0\|_\infty\mathcal M_1(\kappa)\).
Substituting the convolution into
\eqref{eq:particle-jump-remainder-functional} and applying
Fubini's theorem now gives
\begin{align}
  \mathcal R_{\rho_0}(\eta-\rho_0)
  =&\iint_{\R^3\times\R^3}
       H(x,y)\,d(\eta\otimes\eta)(x,y)-\iint_{\R^3\times\R^3}
       H(x,y)\,d(\eta\otimes\rho_0)(x,y)\nonumber \\
     &-\iint_{\R^3\times\R^3}
       H(x,y)\,d(\rho_0\otimes\eta)(x,y)+\iint_{\R^3\times\R^3}
       H(x,y)\,d(\rho_0\otimes\rho_0)(x,y).
\end{align}
If \(\eta_n\) converges narrowly to \(\nu\), testing with \(H\) shows that
\(
  \mathcal R_{\rho_0}(\eta_n-\rho_0)
  \to\mathcal R_{\rho_0}(\nu-\rho_0).
\)

It remains to construct empirical measures with distinct atoms
that converge to \(\nu\) in \(W_1\). Let \((Y_i)_{i\ge1}\) be independent random variables with
common law \(\nu\), and set
\(
  \eta_N:=\frac1N\sum_{i=1}^N\delta_{Y_i}.
\)
Choose a countable family
\(\{\phi_j\}_{j\ge1}\subset C_c^\infty(\R^3)\) dense in
\(C_0(\R^3)\) in the uniform norm. The strong law of large
numbers gives, almost surely and simultaneously for every \(j\),
\begin{equation}
  \int_{\R^3}\phi_j\,d\eta_N\to\int_{\R^3}\phi_j\,d\nu,
  \qquad
  \mathcal M_1(\eta_N)\to\mathcal M_1(\nu).
\end{equation}
Uniform approximation extends the first convergence to every
function in \(C_0(\R^3)\). Since \(\nu\) has mass one, this
implies narrow convergence. Together with convergence of the
first moments, it yields \(W_1(\eta_N,\nu)\to0\) almost surely.
Moreover, independence and the absence of atoms in \(\nu\)
imply that \(Y_i\ne Y_j\) almost surely for each \(i<j\).
There are only countably many such pairs, so all sampled
points are pairwise distinct almost surely. Choose a realization \(y_1,y_2,\ldots\) for which these
properties hold, and set
\(
  X^{(N)}:=(y_1,\ldots,y_N)\in\mathsf D_N.
\)
Then \(\rho_{X^{(N)}}^N\) converges to \(\nu\) in \(W_1\).
Applying the two continuity statements proved above gives
\eqref{eq:particle-obstruction-approximation}.
\end{proof}

\begin{lemma}
\label{lem:particle-obstruction-drift}
Fix \(\rho_0,\mu\in\mathcal P_1(\R^3)\) and
\(X\in\mathsf D_N\). With \(\vartheta_0^N\), \(F_N(0)\), and
\(\mathcal J_0^N\) defined by
\eqref{eq:particle-mean-field-discrepancy},
\eqref{eq:particle-modulated-energy}, and
\eqref{eq:particle-jump-drift}, respectively, one has
\begin{equation}
\label{eq:particle-obstruction-drift-comparison}
  \mathcal J_0^N
  \ge\mathcal R_{\rho_0}(\vartheta_0^N)
     -4\mathcal M_1(\kappa)F_N(0)
     -\frac{2\mathcal M_{3/2}(\kappa)}{N}\sqrt{F_N(0)}.
\end{equation}
\end{lemma}

\begin{proof}
Write \(F=F_N(0)\), \(g=g_0^N\), and, within this proof, set \(a(x,z):=f_{\rho_0}(x)-f_{\rho_0}(x+z)\) and  \(b(x,z):=g(x)-g(x+z)\).    
Add and subtract \([a+b]_+\) and then
\([a]_+\) in \eqref{eq:particle-jump-drift}. This isolates the
remainder and gives the exact identity
\begin{align}
  \mathcal J_0^N
  ={}&\mathcal R_{\rho_0}(\vartheta_0^N)
  -\iint_{\R^3\times \R^3} b\bigl([a+b]_+-[a]_+\bigr)
        \kappa(z)\,dz\,d\rho_X^N(x)\nonumber \\
  &-\iint_{\R^3\times \R^3} b\bigl([a+b-|z|/N]_+-[a+b]_+\bigr)
        \kappa(z)\,dz\,d\rho_X^N(x).\label{eq:particle-obstruction-drift-decomposition}
\end{align}
Monotonicity and the \(1\)-Lipschitz property of the positive part
give \(0\le b([a+b]_+-[a]_+)\le b^2\). By
\eqref{eq:particle-discrepancy-potential-Holder},
\(|b(x,z)|\le2\sqrt F\,|z|^{1/2}\). Integrating these bounds
shows that the second term in
\eqref{eq:particle-obstruction-drift-decomposition} lies between
\(-4\mathcal M_1(\kappa)F\) and zero. For the last term, the same
Lipschitz property bounds the difference of positive parts by
\(|z|/N\). Multiplying by the bound for \(|b|\) and integrating
bounds its absolute value by
\(2\mathcal M_{3/2}(\kappa)\sqrt F/N\). Substitution in
\eqref{eq:particle-obstruction-drift-decomposition} proves
\eqref{eq:particle-obstruction-drift-comparison}.
\end{proof}

\begin{proposition}[Failure of an energy-only jump-drift estimate]
\label{prop:particle-jump-obstruction}
For the parameters \eqref{eq:particle-obstruction-parameters}, there
exist strictly positive Schwartz probability densities \(\rho_0,\mu\)
such that
\begin{equation}
\label{eq:particle-obstruction-smooth}
  \sup_{\substack{\vartheta\in C_c^\infty(\R^3),\ \int\vartheta=0\\
                   \mathcal E_1(\vartheta)>0}}
  \frac{\mathcal R_{\rho_0}(\vartheta)}{\mathcal E_1(\vartheta)}
  =\infty.
\end{equation}
Moreover, there are integers \(N_k\to\infty\) and configurations
\(X^{(k)}\in\mathsf D_{N_k}\) such that
\begin{equation}
\label{eq:particle-obstruction-empirical}
  F_{N_k}(0)=\mathcal E_1(\rho_{X^{(k)}}^{N_k}-\rho_0)\to0,
  \qquad
  \frac{\mathcal J_0^{N_k}}{F_{N_k}(0)+N_k^{-1}}
  \to\infty.
\end{equation}
Here \(\mathcal J_0^{N_k}\) is evaluated at the configuration
\(X^{(k)}\), with initial continuum density \(\rho_0\) and target
\(\mu\).
\end{proposition}

\begin{proof}
Choose \(\rho_0,\mu,\nu_k\) as in
Lemma~\ref{lem:particle-obstruction-smooth}, and put
\(e_k:=\mathcal E_1(\nu_k-\rho_0)\) and
\(r_k:=\mathcal R_{\rho_0}(\nu_k-\rho_0)/e_k\).
Then \(e_k>0\), \(e_k\to0\), and \(r_k\to\infty\).
Since \(\nu_k-\rho_0\) is smooth, compactly supported, and of zero
mass, this already proves \eqref{eq:particle-obstruction-smooth}. Discard finitely many terms so that \(r_k>0\). For each \(k\),
Lemma~\ref{lem:particle-obstruction-approximation} permits us to
choose strictly increasing integers \(N_k\) and
\(X^{(k)}\in\mathsf D_{N_k}\) with
\begin{equation}
  F_{N_k}(0)\le2e_k,
  \qquad N_k^{-1}\le e_k,
  \qquad
  \mathcal R_{\rho_0}(\rho_{X^{(k)}}^{N_k}-\rho_0)
  \ge\tfrac12 r_ke_k.
\end{equation}
Thus \(F_{N_k}(0)\to0\), and the remainder divided by
\(F_{N_k}(0)+N_k^{-1}\) is at least \(r_k/6\).
To control the finite-particle correction, use
\(2\sqrt{F/N}\le F+N^{-1}\); after division, this gives
\begin{equation}
  \frac{2\mathcal M_{3/2}(\kappa)\sqrt F/N}{F+N^{-1}}
  \le\frac{\mathcal M_{3/2}(\kappa)}{\sqrt N}.
\end{equation}
Applying Lemma~\ref{lem:particle-obstruction-drift} and dividing by
\(F_{N_k}(0)+N_k^{-1}\), we conclude that
\begin{align}
      \frac{\mathcal J_0^{N_k}}{F_{N_k}(0)+N_k^{-1}}
  \ge\frac{r_k}{6}-4\mathcal M_1(\kappa)
       -\frac{\mathcal M_{3/2}(\kappa)}{\sqrt{N_k}}
  \to\infty.
\end{align}
\end{proof}

\printbibliography
\end{document}